\documentclass[11pt,reqno]{article}

\usepackage[utf8]{inputenc}

\usepackage[a4paper, margin=1.5cm,top=3cm, bottom=3cm]{geometry}
\usepackage{amsmath, amsthm, amssymb, amsfonts}
\usepackage{mathtools}
\usepackage{hyperref}
\usepackage{enumitem}
\usepackage{setspace}
\usepackage{xcolor}
\usepackage{thmtools}
\usepackage{thm-restate}
\usepackage{verbatim}
\usepackage{enumitem}

\newtheorem{theorem}{Theorem}[section]
\newtheorem{lemma}[theorem]{Lemma}
\newtheorem{claim}[theorem]{Claim}
\newtheorem{definition}[theorem]{Definition}
\newtheorem{corollary}[theorem]{Corollary}
\newtheorem{proposition}[theorem]{Proposition}
\newtheorem{remark}[theorem]{Remark}
\newtheorem{condition}[theorem]{Condition}

\title{Minimum Degree Threshold for $H$-factors with High Discrepancy}
\author{Domagoj Brada{\v{c}}\and Micha Christoph \and Lior Gishboliner}
\date{February 2023}
\begin{document}

\maketitle

\begin{abstract}
    Given a graph $H$, a perfect $H$-factor in a graph $G$ is a collection of vertex-disjoint copies of $H$ spanning $G$. K{\"u}hn and Osthus showed that the minimum degree threshold for a graph $G$ to contain a perfect $H$-factor is either given by $1-1/\chi(H)$ or by $1-1/\chi_{cr}(H)$ depending on certain natural divisibility considerations. Given a graph $G$ of order $n$, a $2$-edge-coloring of $G$ and a subgraph $G'$ of $G$, we say that $G'$ has high discrepancy if it contains significantly (linear in $n$) more edges of one color than the other. 
    Balogh, Csaba, Pluh\'ar and Treglown asked for the minimum degree threshold guaranteeing that every 2-edge-coloring of $G$ has an $H$-factor with high discrepancy and they settled the case where $H$ is a clique. Here we completely resolve this question by determining the minimum degree threshold for high discrepancy of $H$-factors for every graph $H$. 
\end{abstract}
\section{Introduction}
Combinatorial discrepancy concerns itself with problems of the following form: Given a ground set $S$ and a family of subsets $\mathcal{H}$ of $S$, does there exist a $2$-coloring (or $k$-coloring) of $S$ such that each set in $\mathcal{H}$ contains roughly the same number of elements from each of the colors? The theory studies conditions guaranteeing that such a coloring does or does not exist. We refer the reader to \cite[Chapter 4]{Matousek} for an overview. In recent years there has been considerable interest in discrepancy-type problems on graphs. Here, $S$ is the set of edges of a graph $G$, and $\mathcal{H}$ is a family of subgraphs of $G$ (e.g. Hamilton cycles, perfect matching, clique-factors). Thus, the goal is to find conditions on $G$ which guarantee that in every 2-coloring of the edges of $G$, there exists a subgraph of a certain type whose coloring is unbalanced, namely one color appears significantly more than the other. 
One of the first investigations of this type is by Erd{\H{o}}s, F{\"u}redi, Loebl and S{\'o}s~\cite{erdHos1995discrepancy}, who studied the discrepancy of bounded-degree spanning trees in $2$-colorings of the complete graph. In recent years the subject was revived and there are many new works: subgraph discrepancy problems have been studied for Hamilton cycles \cite{Balogh2020HamiltonCycles,GKM_hamilton,GKM_trees,FHLT}, spanning trees \cite{GKM_trees}, clique-factors \cite{balogh2021discrepancy} and powers of Hamilton cycles \cite{bradavc2021powers}, among others. 
In this paper we settle the problem of minimum degree thresholds for the discrepancy of $H$-factors, resolving a question of Balogh, Csaba, Pluh\'{a}r and Treglown~\cite{balogh2021discrepancy}. Let us give the precise definitions.
\begin{definition}
For a graph $G$, a $2$-edge-coloring (or just $2$-coloring) of $G$ is a function $f:E(G)\rightarrow\{-1,1\}$. For a $2$-coloring $f$ and a subgraph $G'$ of $G$, the {\em discrepancy} of $G'$ is defined as 
$$ f(G')=\sum_{e\in E(G')}f(e).$$
\end{definition}

Given a graph $H$, an {\em $H$-factor} is a graph consisting of vertex-disjoint copies of $H$. A {\em perfect $H$-factor} of a graph $G$ is an $H$-factor which is a spanning (i.e. covering all vertices) subgraph of $G$. Clearly, this is only possible if $|G|$, the number of vertices in $G$, is divisible by $|H|$. 
Our main result determines the minimum degree threshold guaranteeing that in every $2$-edge-coloring of $G$, there is a perfect $H$-factor with discrepancy linear in $n$. Before stating this result, we give some background. 

The study of perfect $H$-factors of graphs has a long and rich history. Tutte's famous theorem gives a necessary and sufficient condition for a graph to have a perfect $K_2$-factor, namely a perfect matching. 
On the computational side, Kirkpatrick and Hell~\cite{kirkpatrick1983complexity} showed that for a fixed graph $H$, finding a perfect $H$-factor in an input graph $G$ is NP-complete whenever $H$ has a connected component of size at least three. 
It is therefore desirable to find sufficient conditions ensuring that a graph $G$ has a perfect $H$-factor. 
One such direction of research is the study of minimum degree conditions. 
The fundamental Hajnal-Szemer{\'e}di\cite{hajnal1970proof} theorem states that for every $r\geq 2$, every graph $G$ with order $n$ divisible by $r$ and with minimum $\delta(G)\geq (1-1/r)n$ has a perfect $K_r$-factor. This bound is tight, as can be seen by taking a balanced complete $r$-partite graph and moving one vertex from one part to another. Indeed, the resulting graph has minimum degree $(1-1/r)n-1$ but no perfect $K_r$-factor. 
Alon and Yuster \cite{AlonYuster} proved an asymptotic generalization of the Hajnal-Szemer\'edi theorem to all graphs, by showing that for every graph $H$, if $G$ is an $n$-vertex graph with $n$ divisible by $|H|$ and with $\delta(G) \geq (1 - \frac{1}{\chi(H)} + \varepsilon)n$, then $G$ has a perfect $H$ factor (where $\varepsilon > 0$ is arbitrary and $n$ is large enough in terms of $\varepsilon$). Later, using their celebrated blow-up lemma, Koml\'os, S\'ark\"ozy and Szemer\'edi \cite{KSS_AlonYuster} improved the error term $\varepsilon n$ to a constant depending on $H$. It turns out, however, that $1 - \frac{1}{\chi(H)}$ is not always the correct 
threshold for forcing a perfect $H$-factor. Koml\'os \cite{KomlosCriticalChi} (see also \cite{AlonFischer,ShokoufandehZhao}) introduced the so-called {\em critical chromatic number} $\chi_{cr}(H)$ and showed that having minimum degree $(1 - \frac{1}{\chi_{cr}(H)} + \varepsilon)n$ already suffices for guaranteeing an $H$-factor that covers {\em almost all} vertices of $G$ (we give the precise definition of $\chi_{cr}$ shortly).
Finally, the ultimate result in this direction was obtained by K{\"u}hn and Osthus \cite{existence}, who determined the minimum degree threshold for the existence of a perfect $H$-factor for every graph $H$, showing that this threshold is either $1 - \frac{1}{\chi(H)}$ or $1 - \frac{1}{\chi_{cr}(H)}$, depending on certain divisibility conditions. To state this result, we need to introduce the following definitions.

Given a graph $H$, let $r = \chi(H)$ be the chromatic number of $H$. Let $\mathcal{C}$ be the class of all $r$-vertex-colorings of $H$. 
For $c \in \mathcal{C}$, let $\sigma(c)$ denote the size of the smallest color class in $c$. Let $\sigma(H) = \min_{c \in \mathcal{C}}\sigma(c)$.
The following is the definition of the critical chromatic number:
$$
\chi_{cr}(H) := \frac{(\chi(H)-1)|H|}{|H|-\sigma(H)}.
$$
For each $c\in \mathcal{C}$ with color classes of size $s_1\leq s_2\leq ... \leq s_{r}$, let 
$$
D(c):=\{s_{i+1}-s_i: 1\leq i\leq r-1\}.
$$
Let $D(\mathcal{C})$ be the union of $D(c)$ over all $c\in \mathcal{C}$ and let $hcf_{\chi}(H)$ be the greatest common divisor of the elements in $D(\mathcal{C})$. Let $hcf_c(H)$ denote the largest common divisor of the orders of the connected components of $H$. 
Define a parameter $hcf(H)$ as follows:
If $r\geq 3$, then set $hcf(H) = 1$ if $hcf_\chi(H) = 1$, and if $r=2$, then set $hcf(H)=1$ if $hcf_\chi(H)\leq 2$ and $hcf_c(H) =1$. In all other cases, $hcf(H) \neq 1$. Now define
$$
\chi^*(H) =
\begin{cases}
    \chi(H) & \text{if $hcf(H)\neq 1$,} \\
    \chi_{cr}(H) & \text{otherwise.}
  \end{cases}
$$
Note that $r-1 \leq \chi_{cr}(H) \leq \chi^*(H) \leq r$ for every $H$ with $\chi(H)=r$. Also, if $H$ has only balanced $r$-colorings (i.e. if in every $r$-coloring of $H$, all color-classes have the same size), then $hcf(H) \neq 1,$ hence $\chi^*(H) = \chi(H) = r$.

The aforementioned result of K{\"u}hn and Osthus \cite{existence} states that $1 - 1/\chi^*(H)$ is the minimum degree threshold for the existence of a perfect $H$-factor. More precisely, they prove the following:
\begin{theorem}[\cite{existence}]\label{existence}
For every graph $H$ there exists a constant $C$ such that every graph $G$ whose order $n$ is divisible by $|H|$ with 
$$
\delta(G)\geq (1-1/\chi^*(H))n+C
$$
contains a perfect $H$-factor.
Moreover, for every $m_0$ there exists a graph $J$ of order $m\geq m_0$ such that $m$ is divisible by $H$ with 
$$
\delta(J) = (1-1/\chi^*(H))m-1
$$
such that $J$ does not contain a perfect $H$-factor.
\end{theorem}
We now move on to discrepancy of $H$-factors. 
For a graph $H$, the {\em $H$-factor discrepancy threshold} for $H$, denoted by $\delta^*(H)$, is defined as the infimum $\delta$ which satisfies the following: for every $\eta > 0$ there exists $\gamma > 0$ and $n_0$, such that for every graph $G$ of order $n \geq n_0$ and $\delta(G) \geq (\delta + \eta)n$, with $|H|$ dividing $n$ and for every $2$-edge-coloring $f$ of $G$ there exists a perfect $H$-factor $F$ in $G$ with $|f(F)| \geq \gamma n$. In other words, $\delta^*(H)$ is the (normalized) minimum degree threshold guaranteeing an $H$-factor with linear discrepancy. Trivially, $\delta^*(H)\geq1-1/\chi^*(H)$, because $1-1/\chi^*(H)$ is the minimum degree threshold for the existence of a perfect $H$-factor.

The study of minimum degree discrepancy thresholds for $H$-factors was initiated by Balogh, Csaba, Pluh\'ar and Treglown \cite{balogh2021discrepancy}, who determined $\delta^*(K_r)$.

\begin{theorem}[\cite{balogh2021discrepancy}]\label{kr_discrepancy}
$
\delta^*(K_r) = \max\{3/4,1-1/(r+1)\}
$
\end{theorem}
Balogh et al.~\cite{balogh2021discrepancy} further asked for the discrepancy threshold of other graphs $H$. 
Our main result completely settles this problem, determining the value of $\delta^*(H)$ for every graph $H$. 
We split the statement into three cases: $\chi(H)=2$, $\chi(H)=3$ and $\chi(H)\geq 4$. First, for bipartite $H$, we have the following:
\begin{theorem}\label{bipartite}
For every graph $H$ with $\chi(H)=2$, it holds that
$$
\delta^*(H) =
\begin{cases}
    \frac{3}{4} & \text{if $H$ is regular,} \\
    1/2 & \text{if $H$ is non-regular and there exists $\rho>0$ such that for every connected } \\ & \text{component $U$ of $H$ it holds that $e_H(U)=\rho |U|$,} \\
    1-1/\chi^*(H) & \text{otherwise.}
  \end{cases}
$$
\end{theorem}
To state our results for $r$-chromatic graphs, $r \geq 3$, we first need to introduce some definitions.
Given a graph $G$, a {\em blowup} of $G$ is any graph obtained from $G$ by replacing each vertex $x\in V(G)$ with a vertex-set $V_x$ and replacing edges $xy\in E(G)$ with complete bipartite graphs $(V_x,V_y)$. The {\em $b$-blowup} of $G$ is the blowup where $|V_x|=b$ for every $x \in V(G)$. Given a $2$-edge-coloring $c$ of $G$, a {\em blowup} of $(G,c)$ is a blowup of $G$ whose edges are colored according to $c$, namely, where for $xy\in E(G)$, all the edges in the complete bipartite graph $(V_x,V_y)$ have color $c(xy)$. We denote the coloring of this blowup also by $c$. 
A central strategy of our argument is to find so-called {\em templates}, defined as follows:
\begin{definition}[Template] \label{def:template}
Given graphs $F,H$ and a $2$-edge-coloring $c$ of $F$, we say that $(F,c)$ is a {\em template} for $H$ if there exists a blowup $B$ of $(F,c)$ and two perfect $H$-factors of $B$ with different discrepancies.
\end{definition}
The {\em size} of the template $(F,c)$ is simply $|F|$.
Next, we introduce the following important parameters of a graph $H$.
\begin{definition}[$\mathcal{K}(H)$, $\delta_0(H)$]\label{delta_0_def}
Let $H$ be an $r$-chromatic graph. The set of {\em non-template colorings} of $H$, denoted $\mathcal{K}(H)$, is the set of all $2$-edge-colorings $c$ of $K_r$ such that $(K_r,c)$ is not a template for $H$. 

Let $\delta_0(H)$ be the maximum over all $\delta$ such that there exists a coloring $c\in\mathcal{K}(H)$ and a blowup $B$ of $(K_r,c)$, such that $\delta(B)=\delta \cdot |B|$ and $B$ has a perfect $H$-factor $F$ with $c(F) = 0$ (by the definition of $\mathcal{K}(H)$, this implies that $c(F) = 0$ for every perfect $H$-factor $F$ in $B$). If there exists no such $c \in \mathcal{K}(H)$ then let $\delta_0(H) = 0$.
\end{definition}
Note that $\delta_0(H) \leq 1 - 1/r$ because every $r$-partite graph $B$ has minimum degree at most $(1-1/r)|B|$.
In Section~\ref{sec:Hstar} we show that the maximum in Definition~\ref{delta_0_def} is attained and also provide an algorithm which computes $\delta_0(H).$
Note that if $r=2$ then $\delta_0(H) = 0$, because every blowup of $K_2$ is monochromatic so all its perfect $H$-factors have non-zero discrepancy.

Observe that $\delta^*(H)\geq \delta_0(H)$. Indeed, by the definition of $\delta_0(H)$, there exist $c\in \mathcal{K}(H)$ and a blowup $B$ of $(K_r,c)$ with $\delta(B)=\delta_0(H)|B|$ such that every perfect $H$-factor of $B$ (and there exists one) has discrepancy zero. Then, for every $b\in\mathbb{N}$, the $b$-blowup $B'$ of $(B,c)$ has a perfect $H$-factor with discrepancy zero. Note that $B'$ is also a blowup of $(K_r,c)$ and since $c\in \mathcal{K}(H)$, every perfect $H$-factor of $B'$ must have zero discrepancy. As we can choose $b$ arbitrarily large, we get that $\delta^*(H)\geq \delta(B)/|B| = \delta_0(H)$.

Next we state our result for $3$-chromatic graphs. Here the following graphs, called {\em butterflies}, play an important role. A butterfly is a 2-edge-colored graph $(L,c)$, where $L$ consists of two triangles $u,v_1,w_1$ and $u,v_2,w_2$ intersecting in a single vertex $u$, and the coloring $c$ is ``antisymmetric'' in the sense that $c(uv_1) = -c(uv_2)$, $c(uw_1) = -c(uw_2)$ and $c(v_1w_1) = -c(v_2w_2)$ (see Figure~\ref{fig:butterflies}).

\begin{figure}
    \centering
    \includegraphics{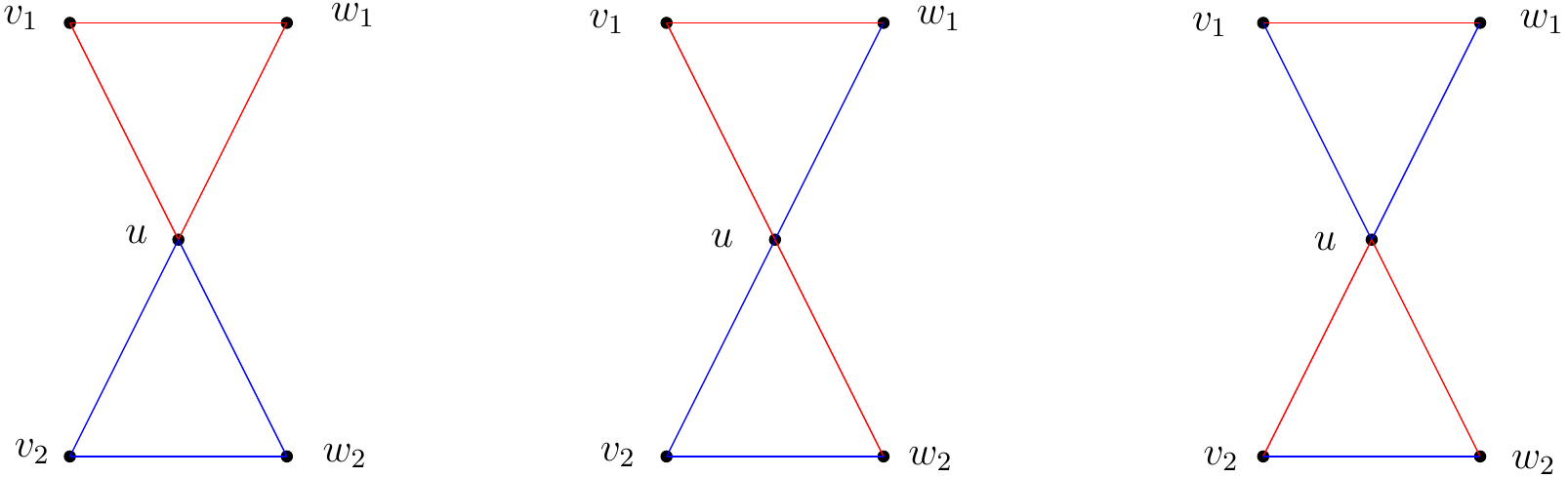}
    \caption{The three different types of butterflies up to isomorphism.}
    \label{fig:butterflies}
\end{figure} 
\begin{theorem}\label{tripartite1}
For every graph $H$ with $\chi(H) = 3$, it holds that
$$
\delta^*(H) =
\begin{cases}
    \frac{3}{4} & \text{if $H$ is regular,} \\
    \max\{1-1/\chi^*(H),\delta_0(H),4/7\} & \text{if $H$ is non-regular and some butterfly is not a template for $H$,}\\
    \max\{1-1/\chi^*(H),\delta_0(H)\} & \text{otherwise.}
  \end{cases}
$$
\end{theorem}
For $k\geq 4$, the following definition, which we call the {\em $k$-wise $C_4$-condition}, plays an essential role in determining the value of $\delta^*(H)$.
\begin{definition}[$C_4$-Condition]\label{def:C4}
For an integer $k \geq 4$, we say that a graph $H$ fulfills the {\em $k$-wise $C_4$-condition} if for every (proper) $k$-vertex-coloring of $H$ with parts $A_1,A_2, \dots, A_k$, we have that
$$e_H(A_1,A_2)+e_H(A_3,A_4)=e_H(A_1,A_3)+e_H(A_2,A_4).$$
\end{definition}
Note that if $H$ satisfies the $k$-wise $C_4$-condition for some $k\geq 5$, then $H$ also satisfies the $(k-1)$-wise $C_4$-condition, as each proper $(k-1)$-coloring of $H$ is also a proper $k$-coloring (by taking the last color-class to be empty). Observe also that if $H$ satisfies the $k$-wise $C_4$-condition then so does every $H$-factor. 
To state our result for $r$-chromatic graphs with $r \geq 4$, it is convenient to state the following two conditions. In the entire paper, we write $a \equiv_m b$ for $a \equiv b \pmod m.$
\begin{condition}\label{cond1}
    $H$ fulfills the $(r+1)$-wise $C_4$-condition and additionally, $r\equiv_4 0$ or $H$ is regular.
\end{condition}
\begin{condition}\label{cond2}
    $H$ fulfills the $r$-wise $C_4$-condition and is regular.
\end{condition}
\noindent
The following result determines $\delta^*(H)$ for $H$ with $\chi(H) \geq 4$.
\begin{theorem}\label{rpartite}
For every graph $H$ with $r = \chi(H) \geq 4$, it holds that
$$
\delta^*(H)=
\begin{cases}
    1-1/(r+1) & $H$ \text{ fulfills Condition~\ref{cond1},} \\
    1-1/r & $H$ \text{ fulfills Condition~\ref{cond2} but not Condition~\ref{cond1},} \\
    \max\{1-1/\chi^*(H),\delta_0(H)\} & $H$ \text{ violates both Conditions~\ref{cond1} and \ref{cond2}.}
  \end{cases}
$$
\end{theorem}
In the next section, we explain the notation that we use throughout the paper. Then, in Section~\ref{sec:overview}, we give a short overview of the proof and explain the main ideas. In Section~\ref{sec: general lemmas} we prove two general lemmas that will play an important role in our proofs. In Section~\ref{sec:Hstar}, we establish properties of the parameter $\delta_0(H)$, and describe a construction of a graph $H^*$ (depending on $H$) that is important in some of our arguments. 
Section~\ref{sec:regularity} is split into two subsections. First, we recall the notions related to Szemer\'edi's regularity lemma and the blowup lemma. And second, we introduce the general setup of how we use the regularity lemma in our proofs. This setup is used throughout the rest of the paper. 
Section~\ref{sec:templates} deals with templates, giving conditions on $H$ that guarantee that various colored graphs are templates for $H$. 
In Section~\ref{sec:lower_bounds} we prove the lower bounds on $\delta^*(H)$ in Theorems~\ref{bipartite},~\ref{tripartite1} and~\ref{rpartite}.
Sections~\ref{sec: K_r same discrepancy},~\ref{sec:no_C4} and~\ref{sec: nonregular} contain key lemmas that are used in the proofs of the main result.  
More specifically, in Section~\ref{sec: K_r same discrepancy} we show how to find perfect $H$-factors of high discrepancy if the coloring of $G$ is unbalanced in a specific way. Section~\ref{sec:no_C4} covers graphs $H$ which violate the $C_4$-condition, and Section~\ref{sec: nonregular} covers graphs $H$ which are non-regular. 
Using the tools from Sections~\ref{sec: K_r same discrepancy}-\ref{sec: nonregular}, we then easily derive the main results (Theorems~\ref{bipartite},~\ref{tripartite1} and~\ref{rpartite}) in Section~\ref{sec:main}. 
To end the paper, in Section~\ref{sec: examples} we give examples of graphs $H$ which fall into different cases of the three main theorems. The purpose of these examples is to show that all cases in Theorems~\ref{bipartite},~\ref{tripartite1} and~\ref{rpartite} are necessary. 

\section{Notation}
Given a graph $G$, let $V(G)$ denote the set of vertices of $G$, $E(G)$ the set of edges and $e(G)=|E(G)|$. Let $|G|$ denote the number of vertices in $G$. Given two sets $U,V\in V(G)$, we write $G[U]$ for the graph induced by $U$ on $G$ and $G[U,V]$ for the graph with edges with one endpoint in $U$ and the other in $V$. Also, let $e_G(U,V)=e(G[U,V])$, and let $e_G(U)$ denote the number of edges in $G[U]$. For $v\in V(G)$, let $d_G(v)$ denote the number of edges incident to $v$ in $G$.

For a 2-edge-coloring $c$ of a graph $G$, we write $G^+,G^-$ for the subgraph of $G$ consisting of the edges of color $+1$ and $-1$, respectively. Note that $c(G) = e(G^+) - e(G^-) = 2e(G^+) - e(G)$. For a color $c \in \{1,-1\}$, we use $G^c$ for $G^+$ if $c=1$ and for $G^-$ if $c=-1$.

Given a blowup $B$ of $G$ and a set of vertices $U\subseteq V(G)$, we write $V_U$ for $\cup_{u\in U}V_u$ and given a vertex $u\in V(B)$, we refer to the vertex in $V(G)$ corresponding to the cluster of $u$ by $V_u^G$. 

Throughout the paper, $H$ is a fixed graph and $r$ is the chromatic number of $H$. An $r$-coloring of $H$ always means a proper $r$-vertex-coloring. We identify an $r$-coloring with its set of color classes, usually denoted $A_1,\dots,A_r$. We think of the parts as ordered, namely permuting them gives a different $r$-coloring. An $r$-coloring $A_1,\dots,A_r$ is called {\em balanced} if $|A_1| = \dots = |A_r|$, and {\em unbalaned} otherwise. 
\section{Proof Overview}\label{sec:overview}
In this section we give a high level overview of our proofs. Some of our arguments apply to any graph $H,$ while some require $H$ to have certain properties. We start by explaining the general setup.

We employ a similar strategy to that used by Balogh, Csaba, Pluh\'{a}r and Treglown~\cite{balogh2021discrepancy} to determine $\delta^*(K_r).$ Given a $2$-edge-coloring of the graph $G$, we apply a colored version of Szemer\'{e}di's regularity lemma and consider the corresponding reduced graph $R$ which, by standard techniques, naturally inherits a $2$-edge coloring $f_R$ from $G$ and has essentially the same minimum degree relative to its number of vertices. A crucial ingredient of our proof is the notion of a template (Definition~\ref{def:template}) which has been introduced in a slightly different form in \cite{balogh2021discrepancy}. The importance of this notion is that if there is a subgraph $F \subseteq R$ of size independent of $n$ such that $(F,f_R)$ is a template for $H$, then, by definition, there is a blowup of $F$ such that there are two $H$-factors of $(F, f_R)$ with different discrepancies. A standard application of the blowup lemma then implies that we can tile a set $U$ of $\Omega(n)$ vertices of $G$ in the clusters of the regular partition corresponding to $V(F)$ with two different $H$-factors whose discrepancies differ by $\Omega(n).$ Taking $U$ to be of small linear size, the graph $G \setminus U$ still has high enough minimum degree to force a perfect $H$-factor by Theorem~\ref{existence}. It is then easy to see that by adding the two $H$-factors of $U$ to this perfect $H$-factor of $G\setminus U,$ we obtain two perfect $H$-factors of $G$ whose discrepancies differ by $\Omega(n)$, hence one of them must have absolute discrepancy $\Omega(n)$, as needed. This shows that finding small templates for $H$ in the reduced graph suffices for finding an $H$-factor of high discrepancy. 

Let us explain another important aspect of the notion of a template. If a certain coloured graph $(F, c)$ is not a template for $H,$ then by definition, for every blowup $B$ of $(F, c),$ all perfect $H$-factors of $B$ have the same discrepancy. If this discrepancy equals $0$ (e.g. this happens if $(F, c)$ is symmetric with respect to the two colours), then this provides us with a lower bound construction for $\delta^*(H).$ The most important special case is when $F = K_r$ (where $r = \chi(H)$)  which leads us to the definition of $\delta_0(H)$. Indeed, $\delta_0(H)$ is the best lower bound on $\delta^*(H)$ that one can obtain by considering blowups of a coloured $K_r$ (apart from the potential divisibility constraints which are encapsulated by the parameter $\chi^*(H)$.)

By the above discussion, we may assume that the reduced graph $R$ has no subgraph on $O(1)$ vertices which is a template for $H$. This can be exploited from two angles. Taking a fixed colored graph $(F, c)$ which is a template for $H,$ we obtain structural information about $R$ as it must be $(F, c)$-free. On the other hand, the fact that a certain coloured graph $F$ is not a template gives us structural information about $H$. More precisely, we obtain that for every $r$-coloring of $H$, the sizes of the color classes and the number of edges between the pairs of them must satisfy certain linear equations. A typical example is Lemma~\ref{template_example}.
(We sometimes also have constraints in terms of $k$-colorings of $H$ for $k = r+1$ or $r+2$. An example is the $C_4$-condition, see Definition~\ref{def:C4}.)

It is possible that there is no small subgraph of $R$ forming a template for $H$, e.g. if all edges in $R$ have color $1$. However, if the coloring of $R$ is so unbalanced, we can find a perfect $H$-factor in $G$ with high discrepancy. 
So, roughly speaking, our strategy is to show that either $R$ has a small template for $H$, or the coloring of $R$ must be in some sense unbalanced, allowing us to find a perfect $H$-factor with high discrepancy by other means. 
A concrete example is Lemma~\ref{all k positive}, which shows that if all $r$-cliques in $R$ have positive discrepancy, then we can indeed find a perfect $H$-factor of high discrepancy, provided we assume additionally that $H$ is not regular. 
So to illustrate our strategy in more detail, let us consider the case when $H$ is not regular, so that Lemma~\ref{all k positive} applies and we may assume that not all $r$-cliques in $R$ have positive discrepancy, and by symmetry not all $r$-cliques have negative discrepancy. Then, since $R$ has minimum degree larger than $1 - 1 / (r-1),$ it is not difficult to show that there are two $r$-cliques $L_1$ and $L_2$ sharing $r-2$ vertices, where one of them has positive discrepancy, while the other has negative discrepancy. If the coloring on $L_1 \cup L_2$ is a template for $H,$ we are done, and otherwise $H$ must have a certain structure. Now, by the minimum degree condition on $R,$ for $v \in L_1 \setminus L_2,$ there are many vertices $u$ such that $L_1 \cup \{u\} \setminus \{v\}$ forms an $r$-clique. We show that, essentially, the edges from $u$ to $L_1 \setminus \{v\}$ must be colored in the same way as those from $v$ to $L_1 \setminus \{v\}$ (or else $R$ contains a template for $H$). Such arguments eventually lead to showing that one of the colors is represented significantly more in $R$ than the other color. 
For example, in one of the cases in the proof of Lemma~\ref{high_min_deg}, we show that there is a set of size more than $3n/4$ which is entirely monochromatic. This allows us
to find a perfect $H$-factor with high discrepancy.

Another ingredient in our proof is the idea of using certain complete $r$-partite graphs (where $r = \chi(H)$). More precisely, in order to find a perfect $H$-factor, we sometimes instead find a perfect $H^*$-factor for a certain complete $r$-partite graph $H^*$, and then tile each copy of $H^*$ with copies of $H$. The advantage of working with complete $r$-partite graphs (rather than with general $r$-partite graphs) is that they consist of $r$-cliques, and our templates for $H$ also consist of $r$-cliques. Thus, assuming that there are no small templates for $H$ allows us to deduce things about the colors of the edges of copies of $H^*$. Typically, we show that $H^*$ is colored as a blowup of $K_r$, namely, that all bipartite graphs between color classes are monochromatic. In order to use this approach, $H^*$ must satisfy certain properties. First, it must contain a perfect $H$-factor. And second, the minimum degree threshold for the existence of an $H^*$-factor must be only slightly larger than that of $H$, so that our minimum degree assumptions guarantee the existence of an $H^*$-factor. Such a graph $H^*$ is constructed in Lemma~\ref{h_star}. 
\section{General Lemmas}\label{sec: general lemmas}
In this section we give two general lemmas which are essential to the proofs of our main results.
The following simple lemma allows us to find a ``chain" of $r$-cliques connecting two given $r$-cliques in a graph of sufficiently high minimum degree. Such a lemma has already appeared in previous works, see e.g. \cite{bradavc2021powers}. For completeness, we include a proof. 
\begin{lemma}\label{connectivity}
Let $k\in\mathbb{N}$, let $J$ be an $m$-vertex graph, and let $L,L'\subseteq J$ be two copies of $K_k$.
\begin{enumerate}
    \item If $\delta(J) > \frac{k-1}{k}m$, then there is a sequence $L_1,L_2,...,L_\ell\subseteq J$ of $k$-cliques with $L_1 = L$, $L_{\ell} = L'$ and $|L_i\cap L_{i+1}| = k-1$ for each $1\leq i< \ell$.
    \item If $\delta(J) > \frac{k-2}{k-1}m$, then there is a sequence $L_1,L_2,...,L_\ell\subseteq J$ of $k$-cliques with $L_1 = L$, $L_{\ell} = L'$ and $|L_i\cap L_{i+1}| \geq k-2$ for each $1\leq i< \ell$.
\end{enumerate}
\end{lemma}
\begin{proof}
We start with the first item. Here we assume that $\delta(J) > \frac{k-1}{k}m$, which implies that every $k$ vertices have a common neighbour. 
It is enough to find a sequence $L = L_1,\dots,L_{\ell} = L'$ with $|L_i\cap L_{i+1}| \geq k-1$ (i.e. we can repeat cliques).  
We prove the claim by reverse induction on $t := |L \cap L'|$. If $t \geq k-1$ then there is nothing to prove. Suppose then that $t \leq k-2$. Write $L \cap L' = \{s_1,\dots,s_t\}$, $L \setminus L' = \{x_1,\dots,x_{k-t}\}$, $L' \setminus L = \{y_1,\dots,y_{k-t}\}$. We define vertices $z_1,\dots,z_{k-t-1}$ inductively as follows. Let $z_1$ be a common neighbour of $s_1,\dots,s_t,x_1,\dots,x_{k-t-1},y_1$. For $2 \leq i \leq k-t-1$, let $z_i$ be a common neighbour of $s_1,\dots,s_t,x_i,\dots,x_{k-t-1},z_1,\dots,z_{i-1},y_1$. Write $M_i = \{s_1,\dots,s_t,x_{i},\dots,x_{k-t-1},z_1,\dots,z_i\}$ for $1 \leq i \leq k-t-1$. Then $M_1,\dots,M_{k-t-1}$ are $k$-cliques, $|M_1 \cap L| \geq k-1$, and $|M_i \cap M_{i+1}| \geq k-1$ for $1 \leq i \leq k-t-2$. Also, $L'' := \{s_1,\dots,s_t,z_1,\dots,z_{k-t-1},y_1\}$ is a $k$-clique, $|L'' \cap M_{k-t-1}| \geq k-1$ and $|L'' \cap L'| \geq t+1$. By the induction hypothesis, there is a chain $L'' = N_1,\dots,N_{\ell}=L'$ with $|N_i \cap N_{i+1}| \geq k-1$ for $1 \leq i \leq \ell-1$. Now $L,M_1,\dots,M_{k-t-1},N_1,\dots,N_{\ell} = L'$ is the required chain. 

Next, we prove Item 2 by reducing to Item 1. Here we assume that $\delta(J) > \frac{k-2}{k-1}m$, which implies that every $k-1$ vertices have a common neighbour, and hence every $(k-1)$-clique is contained in a $k$-clique. Take $M \subseteq L, M' \subseteq L'$ of size $k-1$ each. By Item 1 with parameter $k-1$, there are $(k-1)$-cliques $M = M_1,\dots,M_{\ell} = M'$ with $|M_i \cap M_{i+1}| \geq k-2$ for each $1 \leq i\leq \ell-1$. Let $L_i$ be a $k$-clique containing $M_i$, where $L_1 = L$ and $L_{\ell} = L'$. Then $L_1,\dots,L_{\ell}$ is the required sequence. 
\end{proof}
The next lemma is a key reason why the $C_4$-condition is one of the determining factors for the value of $\delta^*(H)$. 
The lemma allows us to control the discrepancy of subgraphs fulfilling the $C_4$-condition in blowups of regular colorings of $K_k$ (i.e., 2-edge-colorings in which the color-classes form regular graphs). 
We will later apply this lemma to $H$-factors (using that an $H$-factor satisfies the $C_4$-condition if $H$ does), to deduce that a regular coloring of $K_k$ is not a template for $H$.
\begin{lemma}\label{regularC4}
Let $c$ be a $2$-edge-coloring of $K_k$, $k\geq 2$, and suppose that $K_k^+$ is $d$-regular for some $d\in \mathbb{N}$. Let $B$ be a blowup of $(K_k,c)$ and $J$ an arbitrary subgraph of $B$. If $J$ fulfills the $k$-wise $C_4$-condition, then
$$
c\left(J\right) = \frac{2d-k+1}{k-1} e(J).
$$
\end{lemma}
\begin{proof}
First, we estimate the number of edges of $J$ contained in the blowup of a given $2$-factor of $K_k$. Here, by ``$2$-factor'' we mean a disjoint union of cycles covering $V(K_k)$.
\begin{claim}\label{2factor}
Let $C$ be a $2$-factor in $K_k$ and $C'$ the corresponding graph in $B$ (i.e., the blowup of $C$). Then
$$
e(J\cap C') = \frac{2}{k-1}e(J). 
$$
\end{claim}
\begin{proof}
Let $X=x_1x_2,...,x_\ell$ be an arbitrary cycle in $C$ of length $\ell$.
Consider any pair $i < j$ such that the vertices $x_i, x_j$ are not adjacent on $X$. Note that for each $i$ there are $\ell-3$ such $j$. 
Since $J$ satisfies the $k$-wise $C_4$-condition, we have
$$e_J(V_{x_i}, V_{x_{i+1}}) + e_J(V_{x_j}, V_{x_{j+1}}) - e_J(V_{x_i}, V_{x_j}) - e_J(V_{x_{i+1}}, V_{x_{j+1}}) = 0,$$
where indices are taken modulo $\ell$. Summing over all such pairs $i < j,$ we obtain
\begin{equation}\label{eq1}
  (\ell-3)\sum_{1\leq i\leq \ell}e_J(V_{x_i},V_{x_{i+1}})-2\sum_{i,j \; : \; |i-j| \not\equiv \pm 1 \; (\text{mod }\ell)}e_J(V_{x_i},V_{x_j}) = 0.  
\end{equation}
If $X$ is the only cycle in $C$, then $\ell = k$, and by \eqref{eq1},
$$
e(J\cap C') = \sum_{1\leq i\leq \ell}e_J(V_{x_i},V_{x_{i+1}}) = e(J) - \sum_{i,j \; : \; |i-j| \not\equiv \pm 1 \; (\text{mod }\ell)}e_J(V_{x_i},V_{x_j}) =  
\frac{2}{k-1}e(J),
$$
as required. 
Therefore, let us assume that there is a second cycle $Y=y_1,y_,...,y_h$ in $C$. By the $C_4$-condition,  
$$
\sum_{1\leq i \leq \ell}\sum_{1\leq j\leq h} \left[ e_J(V_{x_i},V_{x_{i+1}})+ e_J(V_{y_{j}},V_{y_{j+1}})-e_J(V_{x_i},V_{y_{j}})-e_J(V_{x_{i+1}},V_{y_{j+1}}) \right] = 0.
$$
Reordering the terms in the above equality, we get
\begin{equation}\label{eq2}
h\sum_{1\leq i\leq \ell}e_J(V_{x_i},V_{x_{i+1}})+\ell\sum_{1\leq j\leq h}e_J(V_{y_j},V_{y_{j+1}})-2\sum_{1\leq i\leq \ell}\sum_{1\leq j\leq h}e_J(V_{x_i},V_{y_j}) = 0.
\end{equation}
Now, we take the sum of (\ref{eq1}) over all cycles in $C$ and the sum of (\ref{eq2}) over all pairs of cycles in $C$. 
For each cycle $X$ in $C$, each edge of $X$ is counted $k-|X|$ times when summing \eqref{eq2} over pairs $X,Y$ with $Y \in C \setminus \{X\}$, and is counted $|X|-3$ times in \eqref{eq1}. Therefore, each edge of $C$ is counted exactly $k-3$ times. Also, each edge $e \in K_k \setminus C$ is counted twice (with a negative sign) when summing \eqref{eq1} and \eqref{eq2}; indeed, if $e$ goes between vertices of the same cycle $X$, then $e$ is counted twice in \eqref{eq1}, and if $e$ goes between vertices of two different cycles $X,Y$, then $e$ is counted twice in \eqref{eq2}. All in all, we get that 
\begin{align*}
    (k-3)\sum_{uv\in C}e_J(V_u,V_v)- 2\sum_{uv\in K_k\backslash C}e_J(V_u,V_v) = 0.
\end{align*}
It follows that
$$
e(J\cap C') = \sum_{uv\in C}e_J(V_u,V_v) = \frac{2}{(k-1)}e(J).
$$
\end{proof}
We now complete the proof of Lemma~\ref{regularC4} using Claim~\ref{2factor}.
Note that either $d$ or $k-1-d$ is even because $k$ and $d$ cannot both be odd. Without loss of generality, let us assume that $d$ is even (else apply the same argument to the complement coloring). By Petersen's theorem (see e.g. \cite[Corollary 2.1.5]{Diestel}), the edges of $K_k^+$ can be decomposed into $\frac{d}{2}$ $2$-factors. By applying Claim~\ref{2factor} to each of these $2$-factors, we get that 
$$
e(J^+) = \frac{d}{2} \cdot \frac{2}{(k-1)}e(J) = \frac{d}{(k-1)}e(J).
$$
The result follows since
$$
c(J) = e(J^+)-e(J^-) = 2e(J^+)-e(J).
$$
\end{proof}
\section{The parameters $\delta_0$ and $H^*$}\label{sec:Hstar}
The goal of this section is twofold. First, we consider the parameter $\delta_0(H)$; we show that it is well-defined, i.e. that the maximum in Definition \ref{delta_0_def} is attained, prove some useful properties of $\delta_0(H)$ and give an algorithm that computes $\delta_0(H)$ in finite time. 
And second, we describe a construction of a certain complete $r$-partite graph $H^*$ that will play an important role in our proofs. 

\begin{proposition}\label{prop:compute delta0}
    Let $H$ be a graph. The maximum in the definition of $\delta_0(H)$ (see Definition~\ref{delta_0_def}) is attained and $\delta_0(H) \in \mathbb{Q}$. Moreover, there is an algorithm which, given a graph $H,$ computes $\delta_0(H).$
\end{proposition}
\begin{proof}
    We present an algorithm for computing $\delta_0(H).$ From the algorithm, it will be clear that the maximum in Definition~\ref{delta_0_def} is attained.
    
    The algorithm is as follows. We iterate over all possible $2^{\binom{r}{2}}$ $2$-edge-colorings of $K_r.$ For each coloring $c$ we need to check whether it is a template with respect to $H$ and if it is not, to find the maximum value of $\delta$ such that there is a blowup $B$ of $(K_r, c)$ with $\delta(B) = \delta |B|$ and $B$ has a perfect $H$-factor with $c(F) = 0.$ 
    
    Fix a $2$-edge-coloring $c$ of $K_r$. Let $\mathcal{C} \subseteq [r]^{V(H)}$ denote the set of all proper $r$-vertex-colorings of $H.$ Consider a blowup $B$ of $(K_r, c)$ with parts $A_1, A_2, \dots, A_r$ of sizes $|A_i| = a_i, i \in [r].$ For $f \in \mathcal{C},$ we define $a_i(f) = |\{ v \in V(H) \vert f(v) = i\}|, $ for $i \in [r],$ and $g(f) = \sum_{uv \in E(H)} c(f(u) f(v)),$ where we denote each vertex in $K_r$ by a color of $f$. We think of $f$ as an embedding of $H$ into $B.$ Then, $a_i(f)$ counts the number of vertices embedded into $A_i,$ while $g(f)$ denotes the discrepancy of the embedding. 
    
    Now, checking whether $(K_r, c)$ is a template for $H$ can be done using the following linear program.
    
    \begin{align*}
        \text{maximize } &\sum_{f \in \mathcal{C}} (x_f - y_f) \cdot g(f)\\
        \text{subject to } &\sum_{f \in \mathcal{C}} x_f \cdot a_i(f) = a_i, \forall i \in [r],\\
        &\sum_{f \in \mathcal{C}} y_f \cdot a_i(f) = a_i, \forall i \in [r],\\
        &\sum_{i=1}^r a_i = 1,\\
        &a_i \ge 0, \forall i \in[r]\\
        &x_f, y_f \ge 0, \forall f \in \mathcal{C}.
    \end{align*}
    
    We claim that if the maximum in the above linear program is $0,$ then $(K_r, c)$ is not a template, otherwise it is. Indeed, there exists an optimal feasible solution for which the vectors $x, y$ are fractional $H$-factors of a blowup of $(K_r, c)$ with parts of relative sizes $a_1, \dots, a_r,$ whereas the objective function corresponds to the difference in the discrepancies of the two fractional $H$-factors. Hence, if the maximum is $0,$ no two $H$-factors can have different discrepancies. On the other hand, if the maximum is nonzero, since the optimum is attained by some solution vector with rational entries, we may multiply it by a large number to get a solution with integer entries. It is not difficult to see that this corresponds to a blowup of $(K_r, c)$ and two $H$-factors of it with different discrepancies.
    
    Now, suppose we are given a coloring $c$ such that $(K_r, c)$ is not a template for $H.$ We wish to find a maximum $\delta$ such that there is a blowup $B$ of $(K_r, c)$ with $\delta(B) = \delta|B|$ for which there is an $H$-factor with discrepancy $0$.
    
    This can be found with the following linear program:
    \begin{align*}
        \text{maximize } &1 - a_r\\
        \text{subject to } &0 \le a_1 \le a_2 \le \dots a_r,\\
        &\sum_{i=1}^r a_i = 1,\\
        &\sum_{f \in \mathcal{C}} x_f \cdot g(f) = 0,\\
        &x_f \ge 0, \forall f \in \mathcal{C}.
    \end{align*}
    
    Again, there exists an optimal feasible solution to the above linear program corresponding to a blowup $B'$ of $(K_r,c)$ with relative part sizes $a_1 \le a_2 \le \dots a_r$ and a fractional $H$-factor $x$ of $B'$ with discrepancy $0$. Multiplying this optimal vector with an appropriate integer, we obtain an integral vector which corresponds to a blowup $B$ of $(K_r, c)$ and an $H$-factor with discrepancy $0$ with respect to $c.$ By the ordering of the $a_i$'s it follows that $\delta(B) = (1 - a_r)|B|,$ as needed. Finally, since the above linear program has integer coefficients, it has a rational solution, giving that $\delta_0(H) \in \mathbb{Q}$.
\end{proof}

Next, we prove some useful facts related to $\delta_0(H)$. We begin with the following simple claim, stating that the $b$-blowup of $K_r$ with $b = (r-1)! \cdot |H|$ can be tiled by copies of $H$ in a uniform manner. 
\begin{lemma}\label{lem:balanced blowup}
Let $B$ be the $(r-1)!|H|$-blowup of $K_r$ with parts $B_1,\dots,B_r$. Then, there exists a perfect $H$-factor $F$ in $B$ such that 
$e_F(B_i,B_j) = 2(r-2)!e(H)$ for every pair $1 \leq i < j \leq r$. 
Therefore, if $B$ is colored such that $B$ is the blowup of $(K_r,c)$ for a coloring $c$ of $K_r$, then $c(F) = c(K_r) \cdot 2(r-2)!e(H)$.

\end{lemma}
\begin{proof}
Let $A_1,A_2,\dots,A_r$ be the parts of an $r$-vertex-coloring of $H$. Then, there exists a perfect $H$-factor $F$ of $B$ that contains for every permutation $\sigma:[r]\rightarrow[r]$ a copy of $H$ with the vertices of $A_i$ in cluster $B_{\sigma(i)}$ for every $1\leq i\leq r$. Note that for every $1\leq i,j\leq r$, by the symmetry of $F$ we have that
$$e_{F}(B_{i},B_{j}) = \frac{e(H)|B|}{\binom{r}{2}|H|}.$$
Using that $|B| = r!|H|$, the statement follows. 
\end{proof}
\noindent
Lemma \ref{lem:balanced blowup} implies that if $\mathcal{K}(H)$ contains a coloring $c$ of $K_r$ with $c(K_r) = 0$, then $\delta_0(H) = 1 - 1/r$. Indeed, by taking $B$ to be the $(r-1)!|H|$-blowup of $(K_r,c)$, we get by Lemma~\ref{lem:balanced blowup} that $B$ has a perfect $H$-factor $F$ with $c(F) = c(K_r) \cdot 2(r-2)!e(H) = 0$. 
This implies that $\delta_0(H) \geq \delta(B)/|B| = 1 - 1/r$ by the definition of \nolinebreak $\delta_0$. 

The next lemma gives an important property of $\delta_0(H)$. 

\begin{lemma}\label{delta_0_temp}
For every $2$-edge-coloring $c\in\mathcal{K}(H)$ with $c(K_r) > 0$ the following holds. Let $B$ be a blowup of $(K_r,c)$ with $\delta(B)> \delta_0(H)|B|$. Then for every $H$-factor $F$ of $B$, $c(F) > 0$.
\end{lemma}
\begin{proof}
Fix $c\in\mathcal{K}(H)$ with $c(K_r)>0$.
Let us assume towards a contradiction that there exists a blowup $B$ of $(K_r,c)$ with $\delta(B)> \delta_0(H)|B|$ and a perfect $H$-factor $F$ of $B$ such that $c(F)\leq0$. Note that if $c(F) = 0$ then $\delta_0(H)\geq\delta(B)/|B| > \delta_0(H)$ which is a contradiction.
Therefore, let us assume that $c(F)<0$. 
We use an intermediate-value argument: we will take a ``union'' of $B$ with a balanced blowup of $(K_r,c)$, which has positive discrepancy, choosing the parameters in such a way that this union $B_2$ will have a perfect $H$-factor with discrepancy $0$. However, $B_2$ will also have normalized minimum degree as large as that of $B$, and this would contradict the definition of $\delta_0(H)$. 
The details follow. Let $B_1$ be an $(r-1)!|H|$-blowup of $(K_r,c)$. By Lemma~\ref{lem:balanced blowup}, there exists a perfect $H$-factor $F_1$ in $B_1$ such that 
$$c(F_1)=c(K_r)\cdot 2(r-2)!e(H)>0.$$ 
Let $a_1 \leq a_2 \leq \dots \leq a_r$ be the sizes of the parts of $B$. Let $B_2$ be a blowup of $(K_r,c)$ with parts of sizes $b_1,b_2,\dots b_r$ where
$b_i = c(F_1) \cdot a_i -c(F)(r-1)!|H|$.
Note that $|B_2| = c(F_1)|B| - r\cdot c(F)(r-1)!|H|$, $\delta(B) = |B| - a_r$ and $\delta(B_2) = |B_2| - b_r$. 
It follows that
$\frac{\delta(B_2)}{|B_2|}\geq\frac{\delta(B)}{|B|}$. Indeed, this inequality is equivalent to $a_r \geq |B|/r$, which clearly holds. 
Let $F_2$ be a perfect $H$-factor in $B_2$ consisting of $c(F_1)$ copies of $F$ and $-c(F)$ copies of $F_1$. It follows that 
$$
c(F_2) = c(F_1)c(F)-c(F)c(F_1) = 0.
$$
By the definition of $\delta_0(H)$ and since $c\in\mathcal{K}(H)$, we get that
$$\delta_0(H)\geq \frac{\delta(B_2)}{|B_2|}\geq\frac{\delta(B)}{|B|}>\delta_0(H).$$
As this is a contradiction, the statement must hold.
\end{proof}

We now move to define, for every $r$-chromatic $H$, a certain $r$-partite graph $H^*$ having several useful properties. Later on, when trying to find an $H$-factor with high discrepancy, we often do this by finding an $H^*$-factor and tiling each copy of $H^*$ with copies of $H$. For $r = 2$ we simply set $H^* = H$. The key case is $r \geq 3$, handled by the following lemma. Note that if $r \geq 3$, then $H^*$ is complete $r$-partite. 
\begin{lemma}\label{h_star}
Let $r\geq 3$ and $\eta > 0$. For convenience, put $\alpha(H) := \max\{\delta_0(H),1-1/\chi^*(H)\}$. 
There exists a graph $H^* = H^*(H,\eta)$ such that:
\begin{itemize}
    \item $H^*$ is a complete $r$-partite graph;
    \item $H^*$ has a perfect $H$-factor;
    \item $\alpha(H) \leq 1-1/\chi_{cr}(H^*) \leq \alpha(H)+\eta/4$.
    \item If $\alpha(H) +\eta/4< (r-1)/r$, then $hcf(H^*) = 1$.
    \item If $\delta_0(H)<1-1/r$ then $\delta(H^*)/|H^*|>\delta_0(H)$. 
\end{itemize}
\end{lemma}
\begin{definition}[The graph $H^*$]\label{def:H^*}
Let $H$ be an $r$-chromatic graph and let $\eta > 0$. If $r = 2$ then define $H^* = H$, and else define $H^*$ to be the graph given by Lemma~\ref{h_star}.
\end{definition}

\begin{proof}[Proof of Lemma~\ref{h_star}]
Let $A_1,A_2,\dots,A_r$ be an $r$-coloring of $H$ with $|A_r| = \sigma(H)$.
Let $B_1$ be a blowup of $K_r$ with parts $B_1^1,B_1^2,\dots, B_1^r$, where the first $r-1$ parts have size $|H|-\sigma(H)$, and $B_1^r$ has size $(r-1)\sigma(H)$. 
Let $F_1\subseteq B_1$ be an $H$-factor containing, for each $1\leq i\leq r-1$, a copy of $H$ in which the vertices of $A_j$ are mapped into $B_1^{(j+i) \pmod{r-1}}$ for $1\leq j\leq r-1$, and the vertices of $A_r$ are mapped into $B_1^{r}$. Note that the sizes of $B_1^1,B_1^2,\dots, B_1^r$ are precisely chosen to accommodate these (vertex-disjoint) copies of $H$. 
So $F_1$ is a perfect $H$-factor of $B_1$. 

Let $B_2$ be the $|H|$-blowup of $K_r$ and let $B_2^1,B_2^2,\dots,B_2^r$ be the parts of $B_2$.
Similarly as before, let $F_2$ be a perfect $H$-factor of $B_2$ containing, for every $1\leq i\leq r$, a copy of $H$ in which $A_j$ is mapped into $B_2^{(i+j) \pmod{r}}$ for every $1\leq j\leq r$. Again, the sizes of $B_2^1,B_2^2,\dots,B_2^r$ are precisely chosen to fit these copies of $H$, as $|H| = |A_1| + \dots + |A_r|$.

Note that if $\alpha(H)+\eta/4\geq (r-1)/r$,
then we can take $H^* = B_2$, trivially fulfilling all the necessary conditions. 
Indeed, the first two items in Lemma \ref{h_star} are immediate, the third item holds because $\chi_{cr}(B_2) = r$ and $\alpha(H) \leq 1 - 1/r$, the fourth item holds vacuously, and the fifth item holds because $\delta(H^*)/|H^*| = 1 - 1/r$. 
Let us therefore assume that $\alpha(H) + \eta/4 < (r-1)/r$. Hence, 
$$\min\{1-\delta_0(H),1/\chi^*(H)\}-\eta/4 = 1 - \alpha(H) - \eta/4 >1/r.$$ 
In particular, $\chi^*(H) < r$, which implies that $\chi^*(H) = \chi_{cr}(H)$ and $hcf(H) = 1$ by the definition of $\chi^*(H)$.
Fix a rational number $\beta$ in the range $1 - \alpha(H) - 0.2\eta \leq \beta \leq 1-\alpha(H) - 0.1\eta$,
and note that
$$
1/r < \beta < 1 - \alpha(H) \leq 1/\chi_{cr}(H).
$$
We now define a complete $r$-partite graph $B_3$ with $r-1$ equal parts and an $r$th smaller part, such that $1/\chi_{cr}(B_3) = \beta$. 
Indeed, let $B_3$ be a blowup of $K_r$ with $(r-1)$ parts $B_3^{1},\dots,B_3^{r-1}$ of size $k|H|+\ell(|H|-\sigma(H))$ each, and one (smaller) part $B_3^{r}$ of size $k|H|+\ell(r-1)\sigma(H)$, where $k,\ell\in \mathbb{N}$ are determined later. Note that $B_3$ is essentially a ``linear combination'' of $B_1$ and $B_2$, namely, $B_3$ can be partitioned into $k$ copies of $B_2$ and $\ell$ copies of $B_1$. 
Since $B_1$ and $B_2$ each have a perfect $H$-factor, so does $B_3$. 
Note that 
\begin{equation}\label{eq:chi_cr(B3)}
    1/\chi_{cr}(B_3)= \frac{|B_3| - |B_3^r|}{(r-1)|B_3|} = 
    \frac{k|H|+\ell(|H|-\sigma(H))}{kr|H|+\ell(r-1)|H|}.
\end{equation}
We now show that there exist $k,\ell$ such that $1/\chi_{cr}(B_3) = \beta$.  For this, we need the right-hand side in \eqref{eq:chi_cr(B3)} to equal $\beta$. 
Reordering this equation, we get
$$
    k = \frac{\ell(|H|-\sigma(H)) - \ell(r-1)|H|\beta}{r|H|\beta - |H|}.
$$
Note that the term above is of the form $\ell\cdot q$ for some suitable $q\in\mathbb{Q}$. Therefore, there exists $\ell\in \mathbb{N}$ such that $k\in\mathbb{Z}$. From now on, fix such $k$ and $\ell$. Using 
$1/r< \beta < 1/\chi_{cr}(H)$, we get 
$$
(r-1)|H|\beta<|H|-\sigma(H)
$$
and 
$$
|H|<r|H|\beta.
$$
Therefore, $k>0$.

Our final graph $H^*$ will be obtained by blowing up $B_3$ by a large integer, and then changing the sizes of the parts by a small amount to make sure that $hcf(H^*) = 1$. We now define this small perturbation. 
Recall the set $\mathcal{D}(\mathcal{C})$ defined in the introduction. For each $s\in \mathcal{D}(\mathcal{C})$, let $A_s^1,A_s^2,\dots,A_s^r$ be the parts of an $r$-vertex-coloring of $H$ with $A_s^1-A_s^2 = s$. 
Since $hcf(H) = 1$, it follows by Bézout's Identity (see e.g. Theorem 2.3. in \cite{dujella2021number}) that there exist integers $x_s$ for each $s\in\mathcal{D}(\mathcal{C})$ such that 
$$\sum_{s\in \mathcal{D}(\mathcal{C})}x_ss = 1.$$
Let $B_4$ be a blowup of $K_r$ with parts $B_4^1,B_4^2,\dots,B_4^r$ of size $a_1,a_2,\dots,a_r$ respectively, where
\begin{align*}
a_1 &= \sum_{\substack{s\in\mathcal{D}(\mathcal{C})\\x_s>0}}x_s\cdot A_s^1 -\sum_{\substack{s\in\mathcal{D}(\mathcal{C})\\x_s<0}}x_s\cdot A_s^2,\\
a_2 &= \sum_{\substack{s\in\mathcal{D}(\mathcal{C})\\x_s>0}}x_s\cdot A_s^2 -\sum_{\substack{s\in\mathcal{D}(\mathcal{C})\\x_s<0}}x_s\cdot A_s^1,\\
a_i &= \sum_{s\in\mathcal{D}(\mathcal{C})}|x_s|\cdot A_s^i\text{ for }3\leq i\leq r.
\end{align*}
Note that $a_1-a_2=\sum_{s\in \mathcal{D}(\mathcal{C})}x_ss = 1$. Additionally, there exists a perfect $H$-factor of $B_4$ containing for each $s\in \mathcal{D}(\mathcal{C})$ with $x_s>0$, $x_s$ copies of $H$ with $A_s^i$ on $B_4^i$ for $1\leq i\leq r$; and for each $s \in \mathcal{D}(\mathcal{C})$ with $x_s< 0$, $-x_s$ copies of $H$ with $A_s^1$ on $B_4^2$, $A_s^2$ on $B_4^1$ and $A_s^i$ on $B_4^i$ for $3\leq i\leq r$ if . 
Let $a=\sum_{1\leq i\leq r}a_i$. Fix a large integer $M = M(H,\eta)$, to be chosen later. Let $B_5$ be a blowup of $K_r$ with parts of size $b_1,b_2,\dots,b_r$, where for $1\leq i \leq r$, 
$$b_i = a_i+aM \cdot |B_3^i|.$$
This immediately implies that $hcf(B_5)=1$, since, using $r\geq 3$, $b_1-b_2 = a_1-a_2 = 1$ (recall that $|B_3^1| = \dots = |B_3^{r-1}|$). Note that the vertices of $B_5$ can be partitioned into a copy of $B_4$ and $aM$ copies of $B_3$. As both $B_3$ and $B_4$ have a perfect $H$-factor, so does $B_5$. As $\chi_{cr}(H)<r$, we have $(r-1)\sigma(H)<|H|-\sigma(H)$. This implies that $b_r \leq b_i$ for all $1 \leq i \leq r-1$, and hence
$\sigma(B_5) = b_r = a_r + aM|B_3^r|$. Also, $|B_5| = a + aM|B_3|$.
Now we get that
\begin{align*}   
    \frac{1}{\chi_{cr}(B_5)} &= \frac{a+a M|B_3| - a_r - aM|B_3^r|}{(r-1)(a+aM|B_3|)} = 
    \frac{|B_3| - |B_3^r|}{(r-1)|B_3|} - 
    \frac{a|B_3^r| - a_r|B_3|}{(r-1)|B_3|(a+aM|B_3|)},
\end{align*}
where the second equality above is a simple calculation. 
Recall that $1/\chi_{cr}(B_3)=\frac{|B_3| - |B_3^r|}{(r-1)|B_3|}$. Choose $M$ large enough so that the second term above is at most $0.05\eta$ in absolute value. Then we have
\begin{align*}   
    \left|1/\chi_{cr}(B_3)-1/\chi_{cr}(B_5)\right| \leq 
     0.05\eta. 
\end{align*}
Recalling that $1/\chi_{cr}(B_3) = \beta$ and using $1 - \alpha(H) - 0.2\eta \leq \beta \leq 1-\alpha(H) - 0.1\eta$, we get 
$$1 - \alpha(H) - \eta/4 \leq 1/\chi_{cr}(B_5) \leq 1 - \alpha(H).$$ This proves the third item in the lemma. 
It remains to prove the last item. 
Note that the largest part of $B_5$ has size 
$$\max_{1 \leq i \leq r-1}b_i \leq a + aM \cdot (|B_3| - |B_3^r|)/(r-1) = a + aM \cdot \beta |B_3|,$$ where the two equalities use that $|B_3^1| = \dots = |B_3^{r-1}|$ and that $\beta = 1/\chi_{cr}(B_3)$. The minimum degree of $B_5$ is $|B_5| - \max_{i}b_i$. Hence,
\begin{align*}
\delta(B_5)/|B_5|\geq
1 - \frac{a + aM \cdot \beta |B_3|}{a + aM|B_3|} = \frac{aM(1-\beta)|B_3|}{a+aM|B_3|} = 1-\beta - \frac{(1-\beta)a}{a+aM|B_3|} &> 1-\beta - 0.1\eta \\ &\geq \alpha(H) \geq \delta_0(H),
\end{align*}
where the strict inequality holds for large enough $M$. 
We see that $H^* = B_5$ fulfills all necessary conditions.
\end{proof}

We end this section with the following important property of $H^*$, 
allowing us to control the discrepancy of $H$-factors of $H^*$ under certain assumptions. 
\begin{lemma}\label{delta_0}
Let $c\in\mathcal{K}(H)$ with $c(K_r) > 0$. Let $J$ be a colored copy of $H^*$ and suppose that $J$ is a blowup of $(K_r,c)$. Then for every perfect $H$-factor $F$ of $J$ it holds that $c(F) > 0$.
\end{lemma}
\begin{proof}
The statement holds trivially if $r=2$ as then, $K_2$ must be monochromatic. Therefore, let us assume that $r\geq 3$.
Recall that $H^*$ is a complete $r$-partite graph. By Lemma~\ref{delta_0_temp}, we may assume that $\delta(H^*)\leq \delta_0(H)|H^*|$. Then, by the definition of $H^*$ in Lemma \ref{h_star}, we have $\delta_0(H) = \delta(H^*)/|H^*| = 1-1/r$.
So $H^*$ is a balanced $r$-partite graph. 
Now, let $h\in\mathbb{N}$ and $J$ a colored copy of $H^*$ such that $J$ is an $h$-blowup of $(K_r,c)$ for some $c\in\mathcal{K}(H)$ with $c(K_r)>0$. Let $F$ be a perfect $H$-factor of $J$ and let us assume towards a contradiction that $c(F)\leq 0$. 

Let $B$ be the $(r-1)!|H|$-blowup of $(K_r,c)$.
By Lemma~\ref{lem:balanced blowup}, there exists a perfect $H$-factor $F'$ in $B$ with 
$$c(F') = c(K_r)\cdot 2(r-2)!e(H)>0.$$
Note that since $c\in\mathcal{K}(H)$, we get that every perfect $H$-factor of $J$ must have discrepancy $c(F)$ and every perfect $H$-factor of $B$ must have discrepancy $c(F')$.

Next, consider an $h\cdot (r-1)!|H|$-blowup $B'$ of $(K_r,c)$. Clearly, $B'$ has a perfect $J$-factor $F_{J}$ containing $(r-1)!|H|$ copies of $J$ and a perfect $B$-factor $F_B$ containing $h$ copies of $B$. Let $F_1$ be a perfect $H$-factor of $B'$ obtained by taking the union of a perfect $H$-factor of each copy of $J$ in $F_{J}$ and $F_2$ similarly by taking the union of a perfect $H$-factor of each copy of $B$ in $F_B$. As we showed above, we get $c(F_1) = (r-1)!|H|\cdot c(F)\leq 0$ and $c(F_2)=h\cdot c(F')>0$. It follows that $c(F_1)\neq c(F_2)$ and therefore, $(K_r,c)$ is a template for $H$. This contradicts the assumption $c\in \mathcal{K}(H)$. 
\end{proof}

\section{Regularity and its application}\label{sec:regularity}
The goal of this section is twofold. First, we recall the well-known Szemer\'edi's regularity lemma and the blowup lemma, which play a key role in our proofs. And second, we introduce the general setup in which we shall apply these tools. This setup will be used throughout the rest of the paper. 
\subsection{Background on regularity}\label{sec_regularity}
Let us recall the basic definitions and notation related to the regularity lemma. 
Given a bipartite graph $G$ with vertex-classes $A,B\subseteq V(G)$, the density of $G$ is defined as
$$
d_G(A,B) = \frac{e_G(A,B)}{|A||B|}.
$$
Given $\varepsilon,d>0$, we say that $G$ is $(\varepsilon,d)$-regular if 
$$
d_G(A,B)\geq d
$$
and for every $X\subseteq A$ with $|X|\geq \varepsilon |A|$ and $Y\subseteq B$ with $|Y|\geq \varepsilon |B|$, we have that 
$$
|d_G(A,B)-d_G(X,Y)|< \varepsilon.
$$
Additionally, we say that $G$ is $(\varepsilon,d)$-superregular if $d_G(a)> d |B|$ for every $a\in A$, $d_G(b)> d |A|$  for every $b\in B$, and for every $X\subseteq A$ and $Y\subseteq B$ with $|X|\geq \varepsilon|A|$ and $|Y|\geq \varepsilon|B|$ we have
$$
d_G(X,Y)> d.
$$ 

\noindent
We shall use the following two color version of the regularity lemma.
\begin{lemma}[\cite{komlos1996szemeredi}]\label{regularity}
For every $\varepsilon>0$ and $\ell_0\in\mathbb{N}$ there exists $L_0=L_0(\varepsilon,\ell_0)$ so that the following holds. Let $d\in[0,1]$ and let $G$ be a graph on $n\geq L_0$ vertices with $2$-edge-coloring $f$. Then there exists a partition $V_0,V_1,...V_\ell$ of $V(G)$ and a spanning subgraph $G'$ of $G$, such that the following conditions hold:
\begin{enumerate}
    \item $\ell_0\leq \ell\leq L_0$;
    \item $d_{G'}(x)\geq d_G(x)-(2d+\varepsilon)n$ for every $x\in V(G)$;
    \item the subgraph $G'[V_i]$ is empty for all $1\leq i\leq \ell$;
    \item $|V_0|\leq \varepsilon n$;
    \item $|V_1|=|V_2|=...=|V_\ell|$;
    \item for all $1\leq i <j \leq \ell$, $G'[V_i,V_j]^+$ is either an $(\varepsilon, d)$-regular pair or empty.
    \item for all $1\leq i <j \leq \ell$, $G'[V_i,V_j]^-$ is either an $(\varepsilon, d)$-regular pair or empty.
\end{enumerate}
\end{lemma}
We call $G'$ the pure graph of $G$ (for the parameters $\varepsilon, \ell_0,d$).
Given a graph $G$ with $2$-edge-coloring $f$ and a pure graph $G'$ of $G$, the {\em reduced graph} $R$ is defined as the graph on vertices $V_1,V_2,...V_\ell$, where $V_i$ and $V_j$ are connected if at least one of $G'[V_i,V_j]_+$ or $G'[V_i,V_j]^-$ is non-empty. Additionally, we associate with $R$ a $2$-edge-coloring $f_R$, where for $V_iV_j\in R$, $f_R(V_i V_j)=1$ if $G'[V_i,V_j]^+$ is non-empty and $f_R(V_i V_j)=-1$ otherwise.
The following is a useful, well-known fact about the reduced graph.
\begin{lemma}\label{R_degree}
Given $c>0$, let $G$ be a graph on $n$ vertices with $\delta(G)\geq cn$ and let $R$ be the reduced graph obtained by applying Lemma~\ref{regularity} to it with some parameters $\varepsilon, d, \ell_0$. Then, $\delta(R)\geq (c-2d-2\varepsilon)|R|$.
\end{lemma}
\noindent
We will also use the well-known Blow-up lemma of Koml\'os, S\'ark\"ozy and Szemer\'edi \cite{komlos1997blow}.

\begin{lemma}[Blow-up lemma\cite{komlos1997blow}]\label{blowup_lemma}
Given a graph $K$ on vertices ${1,...,k}$ and $d,\Delta>0$, there exists $\varepsilon_0=\varepsilon(d,\Delta,k)>0$ such that the following holds. Given $L_1,...,L_k\in \mathbb{N}$ and $\varepsilon<\varepsilon_0$, let $F^*$ be the graph obtained from $K$ by replacing each vertex $i\in F$ with a set $V_i$ of $L_i$ new vertices and joining all vertices in $V_i$ to all vertices in $V_j$ whenever $ij$ is an edge in $K$. Let $G$ be a spanning subgraph of $F^*$ such that for every edge $ij\in F$ the pair $(V_i,V_j)_G$ is $(\varepsilon,d)$-superregular. Then, for every spanning subgraph $F$ of $F^*$ with $\Delta(H)\leq \Delta$, $G$ contains a copy of $F$ in which the vertices playing the role of $V_i$ are mapped to $V_i$. 
\end{lemma}

\noindent
To apply the Blow-up lemma, we will need the following simple lemma. 

\begin{lemma}\label{lem:superregular cleaning}
Let $1 \leq b_1 \leq \dots \leq b_t =: b$ be integers.
Let $G$ be a graph and let $f$ be a $2$-edge-coloring of $E(G)$. 
Let $W_1,\dots,W_t \subseteq V(G)$ be pairwise-disjoint with $|W_1| = \dots = |W_t|=: m$. Let $d,\varepsilon > 0$, and suppose that 
$m \gg 1/\varepsilon \gg b, t, 1/d$.
Then there 
are subsets $W'_i \subseteq W_i$, $i = 1,\dots,t$, such that:
\begin{enumerate}
    \item $|W'_i| = b_i s$ for each $i = 1,\dots,t$, where 
    $s := \lfloor m/b \rfloor - \lfloor 2t\varepsilon m \rfloor$.
    \item For every $1 \leq i < j \leq t$ and color $c \in \{\pm 1\}$, if $G[W_i,W_j]^c$ is $(\varepsilon,d)$-regular then $G[W'_i,W'_j]^c$ is $(2b\varepsilon/b_1,d/2)$-superregular. 
\end{enumerate}
\end{lemma}
The second item of Lemma~\ref{lem:superregular cleaning} uses the standard fact that regular pairs can be made superregular by deleting a small number of vertices. The goal of the first item is to make it possible to tile $W'_1,\dots,W'_t$ with a graph having a $t$-coloring with color-classes of size $b_1,\dots,b_t$.

\begin{proof}[Proof of Lemma~\ref{lem:superregular cleaning}]
First, for each $i \in [t]$, take an arbitrary $U_i \subseteq W_i$ of size $b_i \lfloor m/b \rfloor \leq m$.  
Now fix any $1 \leq i \leq t$, and let $I_i$ be the set of pairs $(j,c)$ such that $G[W_i,W_j]^c$ is $(\varepsilon,d)$-regular (where $j \in [t] \setminus \{i\}$ and $c = \pm 1$). 
For $(j,c) \in I_i$, let $B_{i,j}^c$ be the set of vertices $u_i \in U_i$ which have less than 
$(d - \varepsilon) |U_j|$ neighbours in color $c$ in $U_j$. 
The $(\varepsilon,d)$-regularity of $G[W_i,W_j]^c$ and the fact that $|U_j| \geq \frac{m}{2b} \geq \varepsilon |W_j|$ imply that $|B_{i,j}^c| \leq \varepsilon m$.
%\leq 2b \varepsilon |U_i|$. 
Therefore, $B_i := \bigcup_{(j,c) \in I_i}B_{i,j}^c$ satisfies 
% $|B_i| \leq 2bt\varepsilon |U_i|$.
$|B_i| \leq 2(t-1)\varepsilon m$.
Hence, 
$
|U_i \setminus B_i| \geq b_i \cdot \lfloor m/b \rfloor - 2(t-1)\varepsilon m \geq b_is.
$
So take
$W'_i \subseteq U_i \setminus B_i$ of size $b_i s$ (for $i \in [t]$). Note that $|W'_i| \geq \frac{b_1 m}{2b}$, using that $m \gg 1/\varepsilon \gg b,t$. Also,
$|U_i \setminus W'_i| = b_i \cdot (\lfloor m/b \rfloor - s) \leq 2tb\varepsilon m$.

Now let $1 \leq i < j \leq t$ and $c \in \{\pm 1\}$ such that $G[W_i,W_j]^c$ is $(\varepsilon,d)$-regular. As $W'_i \subseteq U_i \setminus B_i$, all vertices in $W'_i$ have at least $(d-\varepsilon)|U_j|$ color-$c$ neighbours in $U_j$, and hence at least 
$$(d-\varepsilon)|U_j| - |U_j \setminus W'_j| \geq 
(d-\varepsilon)|U_j| - 2tb\varepsilon m > d/2 \cdot |U_j|$$ 
neighbours in $W'_j$, where the last inequality uses that 
$|U_j| \geq \frac{m}{2b}$ and $1/\varepsilon \gg b,t,1/d$.  

Also, for $X \subseteq W'_i, Y \subseteq W'_j$ with $|X| \geq \frac{2b\varepsilon}{b_1} |W'_i|$, $|Y| \geq \frac{2b\varepsilon}{b_1} |W'_j|$, we have $|X|,|Y| \geq \varepsilon m$, and therefore, by the $(\varepsilon,d)$-regularity of $G[W_i,W_j]^c$, the color-$c$ density between $X,Y$ is at least $d - \varepsilon > d/2$. This shows that $G[W'_i,W'_j]^c$ is $(2b\varepsilon/b_1,d/2)$-superregular, as required. 
\end{proof}

\subsection{Applying the regularity lemma: the general setup}\label{sec_app_regularity}
Here we explain the general setup which we use throughout the rest of the paper.
Let $\eta>0$ be fixed and let $\gamma \ll \eta$ be small enough (depending on $\eta$). 
To prove Theorems \ref{bipartite}, \ref{tripartite1} and \ref{rpartite}, we shall consider a graph $G$ with $n \geq n_0$ vertices, $n$ divisible by $|H|$, and minimum degree $\delta(G) \geq (\delta + \eta)n$, for $\delta$ corresponding to the particular case of the above theorems that we are proving, and show that in every $2$-edge-coloring of $G$, there must exist a perfect $H$-factor with discrepancy at least $\gamma n$ in absolute value. 
We say that such an $H$-factor has {\em high discrepancy}. Proving this statement (for all $\eta$) would establish that $\delta^*(H) \leq \delta$. 
Note that we may always assume that $\eta$ is small enough with respect to $H$. From now on, fix $n_0,\varepsilon,\ell_0,d_0,L_0>0$ such that $$1/n_0\ll\gamma\ll 1/L_0\leq1/\ell_0\ll\varepsilon\ll d_0\ll\eta \ll 1/|H|,$$ where $L_0$ is the constant obtained by applying Lemma~\ref{regularity} with $\varepsilon,\ell_0$. 

Recall that $\delta_0$ and $1-1/\chi^*(H)$ are (natural) lower-bounds for $\delta^*(H)$ for every graph $H$. Therefore, throughout this paper, we shall always assume that our target parameter $\delta$ satisfies 
\begin{equation}\label{ineq: lowerbound of delta}
\delta\geq\max\{\delta_0(H),1-1/\chi^*(H)\}.
\end{equation}

So let $G$ be a graph with $n \geq n_0$ vertices, where $n$ is divisible by $|H|$, and with $\delta(G) \geq (\delta + \eta)n$. Our strategy for finding a perfect $H$-factor of high discrepancy sometimes requires us to first find
a perfect $H^*$-factor $F^*$ (and then tile each $H^*$-copy with copies of $H$). To this end, we need that $|G|$ is divisible by $H^*$. Therefore, we shall put aside a small number of vertices to make the number of remaining vertices divisible by $|H^*|$. Indeed, let $F_{Rest}$ be a collection of vertex-disjoint copies of $H$ in $G$, such that $|V(F_{Rest})| < |H^*|$ and $n-|V(F_{Rest})|$ is divisible by $|H^*|$. 
Such a collection exists because $G$ even has a perfect $H$-factor, by Theorem~\ref{existence} and by \eqref{ineq: lowerbound of delta}. 
Set $G^*:=G[V(G)\backslash V(F_{Rest})]$. Recall that $H^*$ depends only on $H$ and $\eta$, so $|H^*| \ll \gamma n$. Therefore, it will suffice to find a perfect $H$-factor of $G^*$ with high discrepancy (as this will give a perfect $H$-factor of $G$ with absolute value discrepancy at least $\frac{\gamma}{2}n$, say). Hence, we concentrate on $G^*$ from now. With a slight abuse of notation, we shall use $n$ to denote $|G^*|$. Note that 
$$
\delta(G^*)\geq(\delta+3\eta/4)n.
$$ 

Fix an arbitrary $2$-edge-coloring $f$ of $G^*$. 
Let $G'$ be the pure graph obtained by applying Lemma~\ref{regularity} with 
parameters $\varepsilon,\ell_0,d_0$ to $G^*$ and $f$. Let $R$ be the corresponding reduced graph with $2$-edge-coloring $f_R$. 
Using $\eta\gg d_0,\varepsilon$ and Lemma~\ref{R_degree}, we get that
$$
\delta(G')/n, \; \delta(R)/|R| \geq\delta(G^*)/n-\eta/4\geq \delta+\eta/2,
$$
and by \eqref{ineq: lowerbound of delta}, we get that
\begin{equation}\label{ineq: mindeg R}
\delta(G')/n, \; \delta(R)/|R| \geq \max\{\delta_0(H),1-1/\chi^*(H)\}+\eta/2.
\end{equation}
We shall assume throughout the paper that \eqref{ineq: mindeg R} holds. For a vertex $u \in V(G') \setminus V_0$, we denote by $V^R_u \in V(R)$ the vertex of $R$ corresponding to the part of the regular partition containing $u$.

We now show that $G'$ contains a perfect $H^*$-factor. 
\begin{lemma}\label{lem:H* factor}
$G'$ has a perfect $H^*$-factor, and hence a perfect $H$-factor. 
\end{lemma}
\begin{proof}
    For convenience, set $\alpha(H) := \max\{\delta_0(H), 1-1/\chi^*(H)\}$. In this notation, we have 
    $$\delta(G') \geq (\alpha(H) + \eta/2)n.$$
    We claim that
    $$\delta(G') \geq (1 - 1/\chi^*(H^*) + \eta/4)n.$$ If $\alpha(H) \geq 1 - 1/r - \eta/4$ then $\delta(G') \geq (1 - 1/r + \eta/4)$, which suffices as $\chi^*(H^*) \leq r$. And if $\alpha(H) \leq 1 - 1/r - \eta/4$, then Lemma \ref{h_star} guarantees that $hcf(H^*) = 1$ and hence 
    $$1 - 1/\chi^*(H^*) = 1 - 1/\chi_{cr}(H^*) \leq \alpha(H) + \eta/4 \leq \delta(G')/n - \eta/4,$$
    as claimed. Now, Theorem~\ref{existence} guarantees that $G'$ has a perfect $H^*$-factor. This also implies that $G'$ has a perfect $H$-factor, because $H^*$ has a perfect $H$-factor by Lemma~\ref{h_star}.
\end{proof}

The next lemma allows us to assume that for each pair of clusters $V_i,V_j$ in the regular partition, all edges in $G'[V_i,V_j]$ have the same color (namely the color $f_R(V_iV_j)$).
Equivalently, for every edge $xy \in E(G')$ with $x,y \notin V_0$, it holds that 
\begin{equation}\label{eq:f_R}
f_R(V_x^RV_y^R) = f(xy).
\end{equation}
We shall assume this throughout the rest of the paper.
\begin{lemma}\label{no_double}
If there exist $1\leq i<j\leq \ell = |R|$ such that both $G'[V_i,V_j]^+$ and $G'[V_i,V_j]^-$ are $(\varepsilon, d_0)$-regular, then there exists a perfect $H$-factor in $G'$ with high discrepancy.
\end{lemma}
\begin{proof}
For convenience, put $U := V_i, V := V_j$, and assume that $G'[U,V]^+$ and $G'[U,V]^-$ are both $(\varepsilon, d_0)$-regular. Then $UV\in R$ and $f_R(UV)=1$ by the definition of $R$. By \eqref{ineq: mindeg R}, 
$$\delta(R)>(1-1/\chi^*(H))|R|\geq
\frac{r-2}{r-1}|R|$$ 
and therefore, there exists a copy $L\subseteq R$ of $K_r$ containing the edge $UV$. Let $W_1,W_2,\dots,W_r$ be the clusters of $L$ with $W_1 = U$ and $W_2 = V$ and let $m=(n-|V_0|)/|R|$ denote the size of each cluster $W_i$. 
By definition, all pairs $G'[W_i,W_j]^{f_R(W_i,W_j)}$ are $(\varepsilon,d_0)$-regular. 
By Lemma~\ref{lem:superregular cleaning} with $b_1 = \dots = b_t = (r-1)!|H|$, there are $W_i' \subseteq W_i$, $i = 1,\dots,r$, with $|W_1'|=|W_2'|=\dots=|W_r'| \geq 0.9m$ and $|W_i'|$ divisible by $(r-1)!|H|$, such that $G'[W'_i,W'_j]^{f_R(W_i,W_j)}$ is $(2\varepsilon,d_0/2)$-superregular for each $1\leq i<j\leq r$, and $G'[W_1',W_2']^{-}$ is also $(2\varepsilon,d_0/2)$-superregular. 
Let $G_1\subseteq G'$ be the graph on $\bigcup_{1\leq i\leq r}W_i'$ with edges $\bigcup_{1\leq i<j\leq r}G'[W_i',W_j']^{f_R(W_iW_j)}$. Let $G_2\subseteq G'$ be the same graph but with $G'[W_1',W_2']^{-}$ in place of $G'[W_1',W_2']^{+}$. Note that 
$|W'_i| \leq |W_i| \leq \frac{n}{\ell_0}$ and $|W'_i| \geq 0.9|W_i| \geq 0.9 \frac{n - \varepsilon n}{L_0} \geq \frac{n}{2L_0}$ (for each $1 \leq i \leq r$). 
Therefore,
$$
\frac{rn}{2L_0}\leq|G_1|=|G_2|\leq\frac{rn}{\ell_0}.
$$

Let $B$ be the $(r-1)!|H|$-blowup of $K_r$ with clusters $B_1,B_2,\dots,B_r$. By Lemma~\ref{lem:balanced blowup}, $B$ has a perfect $H$-factor $F_B$ such that $e_{F_B}(B_i,B_j) = 2(r-2)!e(H)$ for every pair $1 \leq i < j \leq r$.
As $|W'_1| = \dots = |W'_r|$ and each $|W_i'|$ is divisible by $(r-1)!|H|$, we can apply Lemma~\ref{blowup_lemma} to deduce that there exist perfect $B$-factors $F''_1,F''_2$ of $G_1$ and $G_2$ respectively. 
Taking the $H$-factor $F_B$ of every copy of $B$ in $F''_1,F''_2$ gives us perfect $H$-factors $F_1',F_2'$ of $G_1,G_2$, respectively. Moreover, for every $1\leq i,j\leq r$, we have
$$
e_{F_1'}(W_i',W_j') =e_{F_2'}(W_i',W_j') = \frac{|G_1|}{|B|}\cdot 2(r-2)!e(H).
$$
It follows that
$$
f(F_1')-f(F_2') = e_{F_1'}(W_1',W_2') - e_{F_2'}(W_1',W_2')= \frac{|G_1|}{|B|}\cdot 4(r-2)!e(H).
$$
Let $G_0 = G' - \bigcup_{1\leq i\leq r}W_i'$,
and note that
$$\delta(G_0)\geq (1 - 1/\chi^*(H) + \eta/4)n,$$ 
using \eqref{ineq: mindeg R} and that $|G_1|\leq\frac{rn}{\ell_0}$ and $1/\ell_0 \ll \eta$. 
Also, $|V(G_0)|$ is divisible by $|H|$ because $|V(G')|$ and $\sum_{i=1}^{t}|W'_i|$ are.
Thus, $G_0$ contains a perfect $H$-factor $F'$ by Theorem~\ref{existence}. 
Let $F_i := F'_i \cup F'$, $i = 1,2$. 
Then $F_1,F_2$ are perfect $H$-factors of $G'$, 
and 
$$
f(F_1)-f(F_2) = f(F_1')-f(F_2')=\frac{|G_1|}{|B|}\cdot 4(r-2)!e(H) \geq 2\gamma n,
$$
where the last inequality uses that $|B|=r!|H|$, $|G_1| \geq \frac{rn}{2L_0}$ and $\gamma\ll \frac{1}{|H|}, \frac{1}{L_0}$.
Therefore, at least one of $F_1,F_2$ has absolute discrepancy at least $\gamma n$, as required. 
\end{proof}

From now on, we shall work under the above setup and repeatedly use \eqref{ineq: mindeg R} and \eqref{eq:f_R}. 
Recall that our ultimate goal is to find a perfect $H$-factor of $G'$ with high discrepancy. Very roughly, our argument works by showing that either $R$ contains a template for $H$, or else $R$ is colored by $f_R$ in such an imbalanced way that we can directly find an $H$-factor in $G'$ with high discrepancy. 
The next lemma handles the case that $R$ contains a template for $H$, showing that in that case $G'$ indeed contains a perfect $H$-factor with high discrepancy.  
The proof of this lemma is similar to that of \cite[Claim 6.2]{balogh2021discrepancy} (and also uses some ideas from the proof of Lemma~\ref{no_double}). 
\begin{lemma}\label{template}
Let $t \in \mathbb{N}$ depending only on $H$, and let $T\subseteq R$ be a subgraph of $R$ of order $t$. 
If $(T,f_R)$ is a template for $H$, then there exists a perfect $H$-factor in $G'$ with high discrepancy.
\end{lemma}
\begin{proof}
Since $(T,f_R)$ is a template for $H$, there exists a blowup $B$ of $(T,f_R)$ and two perfect $H$-factors $F_1,F_2$ of $B$ with $f_R(F_1)> f_R(F_2)$. 
Write $V(T)=\{w_1,w_2,\dots, w_t\}$, let $B_i$ be the part of $B$ corresponding to $w_i$, and put $b_i = |B_i|$. Suppose that $b_1 \leq \dots \leq b_t =: b$. 
Also, let $W_i$ be the cluster of $G'$ corresponding to $w_i \in V(R)$, and recall that $|W_i| = (n-|V_0|)/|R| =: m$ for $i = 1,\dots,t$. 
Also, recall that if for some $1\leq i<j\leq t$, $w_iw_j\in R$, then $G'[W_i,W_j]^{f_R(w_iw_j)}$ is $(\varepsilon,d_0)$-regular.
By Lemma~\ref{lem:superregular cleaning}, there is an integer $s \geq 0.9m/b$ and subsets $W_i'\subseteq W_i$ ($i = 1,\dots,t$) such that $|W'_i| = b_i s$, and such that for every pair $1 \leq i < j \leq t$, if $w_iw_j\in R$ then $G'[W_i',W_j']^{f_R(w_iw_j)}$ is $(\frac{2b_t}{b_1}\varepsilon,d_0/2)$-superregular.

Let $B'$ be a blowup of $T$ with parts $B'_1,\dots,B'_t$, where $B'_i$ is the cluster corresponding to $w_i$, and $|B'_i| = |W_i'| = b_is$. 
So $B'$ is the $s$-blowup of $B$. 
Note that we may apply Lemma~\ref{blowup_lemma} with $B'$ in the role of $F^*$ and with $\bigcup_{w_iw_j\in T}G'[W_i',W_j']^{f_R(w_iw_j)}$ in the role of $G$. 

Clearly, $B'$ has a perfect $B$-factor $F_B$ consisting of $s$ copies of $B$ (where each copy of $B$ places the part $B_i$ of $B$ inside the part $B'_i$ of $B'$).
Let $F_1',F_2'$ be the perfect $H$-factors of $B'$, obtained by taking the perfect $H$-factor $F_1$ or $F_2$ respectively of each copy of $B$ in $F_B$. By Lemma~\ref{blowup_lemma}, there exist copies $F_1'',F_2''$ of $F_1',F_2'$ (respectively) in $\bigcup_{w_iw_j\in T}G'[W_i',W_j']^{f_R(w_iw_j)}$, with all the vertices on the same corresponding parts (i.e., with $B'_i$ embedded to $W'_i$ for $i = 1,\dots,t$). 
Now, using that $f_R(F_1)-f_R(F_2) \geq 1$, we get that
$$
f(F_1'')-f(F_2'') \geq s.
$$
Now let $G_0 = G' - \bigcup_{1\leq i\leq t}W_i'$, and note that $\delta(G_0)\geq (1 - 1/\chi^*(H) + \eta/4)n,$  
using \eqref{ineq: mindeg R} and 
$$
\sum_{i=1}^{t}|W'_i| \leq \sum_{i=1}^{t}|W_i| \leq tm \leq tn/\ell_0 \ll \eta n,
$$
as $\frac{1}{\ell_0} \ll \frac{1}{|H|},\eta$, and $t$ depends only on $H$.
Also, $|V(G_0)|$ is divisible by $|H|$ because $|V(G')|$ and $\sum_{i=1}^{t}|W'_i|$ are.
Now, by Theorem~\ref{existence}, $G_0$ has a perfect $H$-factor $F'$. Thus, both $F'\cup F_1''$ and $F'\cup F_2''$ are perfect $H$-factors of $G'$, and
$$
f(F'\cup F_1'')-f(F'\cup F_2'') = f(F_1'')-f(F_2'') \geq s \geq \frac{0.9m}{b} \geq \frac{0.9(1-\varepsilon)n}{L_0b} \geq \gamma n,
$$
using that $m = (n-|V_0|)/|R| \geq (1-\varepsilon)n/L_0$, and that $\gamma \ll \varepsilon,1/L_0 \ll 1/b$, as $b$ depends only on $H$. 
Now we see that $F'\cup F_1''$ or $F'\cup F_2''$ has discrepancy at least $\gamma n$, as required.  
\end{proof}

\section{Templates}\label{sec:templates}
Lemma \ref{template} states that in order to find an $H$-factor with high discrepancy, it suffices to show that the reduced graph $R$ contains a template for $H$. The present section is thus dedicated to various constructions of such templates. These constructions form a substantial and crucial part of our proofs. In most cases, the templates consist of either a single $K_r$ or two copies of $K_r$ sharing $r-1$ or $r-2$ vertices. Our results typically state that either a certain colored graph is a template for $H$, or else $H$ has a certain ``uniformity property'', e.g. $H$ is regular, only has balanced $r$-colorings, or has the same number of edges between any two parts in any $r$-coloring. Each of the following subsections deals with one such uniformity property and gives a suitable template for graphs violating the property. 
The basic idea for proving that a certain configuration $L$ is a template for $H$ is usually the same: we define a (carefully chosen) perfect $H$-factor $F_1$ of a (carefully chosen) blowup of $L$, and then modify $F_1$ by moving some vertices to different blowup-clusters, thus obtaining a second perfect $H$-factor $F_2$. We then observe that if this modification did not change the discrepancy (i.e. if $F_1,F_2$ have the same discrepancy), then $H$ must have the relevant uniformity property. 
The implementation of this rough idea for the various uniformity properties can be quite involved.  

\subsection{Disconnected bipartite graphs}
In this section, we consider disconnected bipartite graphs $H$. 
We show that if $H$ has two connected components with different average degrees then the colored graph consisting of two disjoint edges of different color is a template for $H$.
\begin{lemma}\label{bipartite template}
Suppose that $H$ is bipartite and there exist two connected components $U,W\subseteq V(H)$ of $H$ such that $e_H(U)/|U|\neq e_H(W)/|W|$. 
Let $e_1,e_2$ be two disjoint edges and $c$ be a $2$-coloring of $e_1,e_2$ with $c(e_1)\neq c(e_2)$. 
Then $(e_1\cup e_2,c)$ is a template for $H$.
\end{lemma}
\begin{proof}
Write $e_1 = x_1y_1 ,e_2 = x_2y_2$. without loss of generality, suppose that $c(e_1) = 1$, $c(e_2) = -1$.
Fix a $2$-coloring $A_1,A_2$ of $H$, and let $U_i = U \cap A_i$, $W_i = W \cap A_i$, $i = 1,2$. 
Let $B$ be a blowup of $(e_1 \cup e_2 ,c)$ with 
\begin{align*}
    |V_{x_1}|,|V_{y_1}| &= 2|U||W|,\\
    |V_{x_2}|,|V_{y_2}| &= 2|U||H|+2|W|(|H|-|U|).
\end{align*}

Our goal is to find two perfect $H$-factors $F_1,F_2$ of $B$ with different discrepancies. The idea is simple: $F_1$ will contain copies of $H$ in which $U = U_1 \cup U_2$ is mapped onto $V_{x_1},V_{y_1}$ and $H - U$ is mapped onto $V_{x_2},V_{y_2}$, while $F_2$ will contain copies of $H$ in which $W = W_1 \cup W_2$ is mapped onto $V_{x_1},V_{y_1}$ and $H - W$ is mapped onto $V_{x_2},V_{y_2}$. The fact that $H[U]$ and $H[W]$ have different average degrees will imply that $F_1$ and $F_2$ have a different number of edges of color $1$ (since all edges between $V_{x_1},V_{y_1}$ have color $1$, and all edges between $V_{x_2},V_{y_2}$ have color $-1$). This will imply that $c(F_1) \neq c(F_2)$. To make this scheme work, we need two additional ideas. First, $F_1$ will have $H$-copies mapping $U_1$ to $V_{x_1}$ and $U_2$ to $V_{y_1}$, as well as $H$-copies mapping $U_2$ to $V_{x_1}$ and $U_1$ to $V_{y_1}$; there will be the same number of copies of the two types. This ``symmetrization'' allows us to take $V_{x_1},V_{y_1}$ to be of the same size, as each $H$-copy adds $|U|/2$ vertices on average to $V_{x_1}$ and $V_{y_1}$. When describing this construction, we will say that we take copies of $H$ with each permutation of $U_1,U_2$ on $V_{x_1},V_{y_1}$. (This language is also used later on in this section.)
We will do the same for the $H$-copies in $F_2$ with respect to $W_1,W_2$. 

The above scheme tiles $V_{x_1} \cup V_{y_1}$ with 
$(|V_{x_1}| + |V_{y_1}|)/|U|$ copies of $H$ in $F_1$, and with $(|V_{x_1}| + |V_{y_1}|)/|W|$ copies of $H$ in $F_2$. 
A problem that might occur is that $F_1,F_2$ use a different number of vertices in $V_{x_2} \cup V_{y_2}$. This must be avoided because $F_1,F_2$ need to be perfect $H$-factors of $B$. To remedy this, we add additional $H$-copies to $F_1,F_2$ which only use vertices from $V_{x_2} \cup V_{y_2}$. By appropriately choosing the number of these copies, as well as the size of $V_{x_2},V_{y_2}$, we can make sure that $F_1,F_2$ tile $B$. 
The details follow.

Define two perfect $H$-factors $F_1,F_2$ of $B$, each containing $4(|U|+|W|)$ copies of $H$, as follows:
\begin{itemize}
    \item $F_1$ contains $|W|$ copies of $H$ for each permutation of $U_1,U_2$ on $V_{x_1},V_{y_1}$ and each permutation of $A_1\backslash U_1,A_2\backslash U_2$ on $V_{x_2},V_{y_2}$. Additionally, $F_1$ contains $2|U|$ copies of $H$ for each permutation of $A_1,A_2$ on $V_{x_2},V_{y_2}$.
    \item $F_2$ contains $|U|$ copies of $H$ for each permutation of $W_1,W_2$ on $V_{x_1},V_{y_1}$ and each permutation of $A_1\backslash W_1,A_2\backslash W_2$ on $V_{x_2},V_{y_2}$. Additionally, $F_2$ contains $2|W|$ copies of $H$ for each permutation of $A_1,A_2$ on $V_{x_2},V_{y_2}$.
\end{itemize}
Note that the choice of $|V_{x_1}|,|V_{y_1}|,|V_{x_2}|,|V_{y_2}|$ exactly corresponds to the definition of $F_1,F_2$. 
For example, the number of vertices of $F_1$ in $V_{x_2}$ is exactly $2 \cdot |W| \cdot (|A_1 \setminus U_1| + |A_2 \setminus U_2|) + 2|U| \cdot (|A_1| + |A_2|) = 2|W| (|H| - |U|) + 2|U||H| = |V_{x_2}|$, and similarly for $F_2$ and the other three clusters $V_{y_2},V_{x_1},V_{y_1}$.

Next, observe that 
$$
e(F_1^+) = e_{F_1}(V_{x_1},V_{y_1}) = 4|W|e_H(U)
$$
and 
$$
e(F_2^+) = e_{F_2}(V_{x_1},V_{y_1}) = 4|U|e_H(W).
$$
As $|W|e_H(U) \neq |U|e_H(W)$, we have that $e(F_1^+) \neq e(F_2^+)$. Also, $e(F_1) = e(F_2)$ because $F_1,F_2$ are both perfect $H$-factors of $B$. 
Hence,
$$
c(F_1) = 2e(F_1^+) - e(F_1) \neq 2e(F_2^+) - e(F_2) = c(F_2).
$$
This implies that $(e_1 \cup e_2,c)$ is a template for $H$, as required. 
\end{proof}

\subsection{Non-regular graphs}

In this section we consider non-regular graphs $H$. As always, $r$ denotes the chromatic number of $H$. 
\begin{lemma}\label{sharing_much_template}
Let $L_1,L_2$ be two copies of $K_r$ sharing $r-1$ vertices with $2$-edge-coloring $c$ such that $c(L_1)\neq c(L_2)$. If $H$ is non-regular then $(L_1\cup L_2,c)$ is a template for $H$.
\end{lemma}
\begin{proof}
Let $u,v\in V(H)$ be two vertices with $d_H(u)\neq d_H(v)$. Fix an $r$-vertex-coloring of $H$ with parts $A_1,A_2,\dots,A_r$. 
Let $A_{i_u}$ be the part containing $u$ and $A_{i_v}$ the one containing $v$ (possibly $i_u = i_v$). 
Write $L_1 \cap L_2 = \{q_2,q_3,\dots,q_{r}\}$, $L_1\backslash L_2 = \{s\}$ and $L_2\backslash L_1 = \{t\}$. Let $B$ be a blowup of $(L_1\cup L_2,c)$ with
\begin{itemize}
    \item $|V_s| = (r-1)!$,
    \item $|V_t| = (r-1)!(|A_{i_u}|+|A_{i_v}|-1)$ and
    \item $|V_{q_i}|=(r-2)!(2|H|-|A_{i_u}|-|A_{i_v}|)$ for $2\leq i\leq r$.
\end{itemize}
We now define certain $H$-copies in $B$. For each permutation $\sigma :  \{2,\dots,r\} \rightarrow [r] \setminus \{i_v\}$, let $X_{\sigma}$ be a copy of $H$ in which $A_{i_v}$ is embedded into $V_t$ and $A_{\sigma(i)}$ is embedded into $V_{q_i}$ for all $2 \leq i \leq r$. Let $X'_{\sigma}$ be the copy of $H$ obtained from $X_{\sigma}$ by moving $v$ from $V_t$ to $V_s$. Similarly, for each permutation $\tau : \{2,\dots,r\} \rightarrow [r] \setminus \{i_u\}$, let $Y_{\tau}$ be a copy of $H$ in which $A_{i_u}$ is embedded into $V_t$ and $A_{\tau(i)}$ is embedded into $V_{q_i}$ for all $2 \leq i \leq r$. Let $Y'_{\tau}$ be the copy of $H$ obtained from $Y_{\tau}$ by moving $u$ from $V_t$ to $V_s$. We define all these $H$-copies such that the copies in each of the sets $F_1 := \{ X_{\sigma}, Y'_{\tau}\}_{\sigma,\tau}$ and $F_2 := \{ Y_{\tau}, X'_{\sigma}\}_{\sigma,\tau}$ are pairwise-disjoint and partition $V(B)$; note that the sizes of $V_s,V_t,V_{q_2},\dots,V_{q_r}$ are precisely chosen to allow this. In other words, $F_1,F_2$ are perfect $H$-factors of $B$. 
We now calculate $c(F_1)-c(F_2)$. 
By definition,
$$
c(X_{\sigma}) - c(X'_{\sigma}) = 
\sum_{i = 2}^r(c(tq_i) - c(sq_i)) \cdot e_H(v,A_{\sigma(i)}).
$$
For each $j \in [r] \setminus \{i_v\}$ and $2 \leq i \leq r$, there are exactly $(r-2)!$ permutations $\sigma$ with $\sigma(i) = j$. Hence, summing over all $\sigma$, we get
\begin{align*}
\sum_{\sigma}(c(X_{\sigma}) - c(X'_{\sigma})) &= 
\sum_{i = 2}^r(c(tq_i) - c(sq_i)) \sum_{j \in [r] \setminus \{i_v\}}(r-2)! \cdot e_H(v,A_j) \\ &= \sum_{i = 2}^r(c(tq_i) - c(sq_i)) \cdot (r-2)! \cdot d_H(v) = 
(c(L_2) - c(L_1)) \cdot (r-2)! \cdot d_H(v).
\end{align*}
Similarly,
$$
\sum_{\tau}(c(Y_{\tau}) - c(Y'_{\tau})) = (c(L_2) - c(L_1)) \cdot (r-2)! \cdot d_H(v).
$$
We get that
$$
c(F_1) - c(F_2) = 
\sum_{\sigma}(c(X_{\sigma}) - c(X'_{\sigma})) - 
\sum_{\tau}(c(Y_{\tau}) - c(Y'_{\tau})) = 
(c(L_2) - c(L_1)) \cdot (r-2)! \cdot (d_H(v) - d_H(u)) \neq 0,
$$
using that $c(L_1) \neq c(L_2)$ by assumption and $d_H(u) \neq d_H(v)$. 
The fact that $c(F_1) \neq c(F_2)$ means that $(L_1\cup L_2,c)$ is a template with respect to $H$.
\end{proof}
\noindent
We now use Lemma \ref{sharing_much_template} to show that certain colorings of $K_{r+1}$ are templates for $H$. 
\begin{corollary}\label{non_reg_clique}
Let $L$ be a copy of $K_{r+1}$ with $2$-edge-coloring $c$ such that $L^+$ is non-regular. If $H$ is non-regular then $(L,c)$ is a template for $H$.
\end{corollary}
\begin{proof}
Since $L^+$ is non-regular, there exist $u,v\in V(L)$ such that $c(u,L)\neq c(v,L)$. Let $U = V(L)\backslash\{v\}$ and $V = V(L)\backslash\{u\}$. It is not hard to see that $|U\cap V| = r-1$ and 
$$c(L[U])- c(L[V]) = c(u,L)-c(uv) - (c(v,L)-c(uv))\neq 0.$$
The statement follows by Lemma~\ref{sharing_much_template}.
\end{proof}
\subsection{Different degrees to different parts}
In this section we consider another ``uniformity property'' defined in terms of vertex-degrees. Here, we assume that there is an $r$-coloring of $H$ such that some vertex has different degrees to two color-classes. We show that in this case, a certain configuration is a template for $H$. 

\begin{lemma}\label{like blowup2}
Let $L_1,L_2$ be two copies of $K_r$ sharing $r-1$ vertices with $\{x\} = L_1\backslash L_2$ and $\{y\}=L_2\backslash L_1$. Let $c$ be a $2$-edge-coloring of $L_1\cup L_2$ with $c(L_1)=c(L_2)$ such that there exists a vertex $z\in L_1\cap L_2$ with $c(xz)=-c(yz)=1$. If $H$ has an $r$-coloring $A_1,A_2,\dots, A_r$ and a vertex $a_1\in A_1$ such that $d_H(a_1,A_2)\neq d_H(a_1,A_3)$, then $(L_1\cup L_2,c)$ is a template for $H$.
\end{lemma}
\begin{proof}
We have $$0 = c(L_1)-c(L_2) = 
\sum_{w \in L_1 \cap L_2}(c(xw) - c(yw)).$$ Since $z \in L_1 \cap L_2$ satisfies $c(yz)=-c(xz)=1$, there must exist $w \in (L_1 \cap L_2) \setminus \{z\}$ such that $c(yw)=-c(xw)=-1$. Note that this implies that $r\geq 3$.
Write $L_1 \cap L_2 = \{z,w,v_4,\dots, v_r\}$. 
Let $A_1,A_2,\dots,A_r$ be an $r$-vertex-coloring of $H$ such that there exists a vertex $a_1\in A_1$ with $d_H(a_1,A_2)\neq d_H(a_1,A_3)$. Let $B$ be a blowup of $(L_1\cup L_2,c)$ with 
\begin{align*}
&|V_{x}| = 1,\\ 
&|V_y| = 2|A_1|-1,\\ 
&|V_{z}|=|V_{w}| =|A_2|+|A_3|,\\
&|V_{v_i}| = 2|A_i|\ \text{for }4\leq i\leq r.
\end{align*}
We now define certain $H$-copies in $B$. For a permutation $\sigma : \{4,\dots,r\} \rightarrow \{4,\dots,r\}$, let $X_{\sigma}$ be a copy of $H$ in which $A_1$ is embedded into $V_y$, $A_2$ is embedded into $V_z$, $A_3$ is embedded into $V_w$ and $A_{\sigma(i)}$ is embedded into $V_{v_i}$ for every $4 \leq i \leq r$. Let $Y_{\sigma}$ be the $H$-copy obtained from $X_{\sigma}$ by swapping $A_2$ and $A_3$, i.e. embedding $A_2$ into $A_w$ and $A_3$ into $A_z$. Let $X'_{\sigma}$ (resp. $Y'_{\sigma}$) be the $H$-copy obtained from $X_{\sigma}$ (resp. $Y_{\sigma}$) by moving $a_1$ from $V_y$ to $V_x$.  
We define all these $H$-copies such that the copies in each of the sets $F_1 := \{ X_{\sigma}, Y'_{\sigma}\}_{\sigma}$ and $F_2 := \{ Y_{\sigma}, X'_{\sigma}\}_{\sigma}$ are pairwise-disjoint and partition $V(B)$; note that the sizes of $V_x,V_y,V_z,V_w,V_{v_4},\dots,V_{v_r}$ are precisely chosen to allow this. In other words, $F_1,F_2$ are perfect $H$-factors of $B$. 
We now calculate $c(F_1)-c(F_2)$. By definition,
$$
c(X_{\sigma}) - c(X'_{\sigma}) = (c(yz) - c(xz)) \cdot d_H(a_1,A_2) + 
(c(yw) - c(xw)) \cdot d_H(a_1,A_3) + 
\sum_{i = 4}^r (c(yv_i) - c(xv_i)) \cdot d_H(a_1,A_{\sigma(i)}).
$$
For every two indices $4 \leq i,j \leq r$, there are exactly $(r-4)!$ permutations $\sigma$ with $\sigma(i) = j$. 
Hence, $\sum_{\sigma}d_H(a_1,A_{\sigma(i)}) = (r-4)! \cdot d_H(a_1,A_4\cup \dots \cup A_r)$.
We get that
\begin{align*}
\sum_{\sigma}(c(X_{\sigma}) - c(X'_{\sigma})) = \; &(r-3)! \cdot (c(yz) - c(xz)) \cdot d_H(a_1,A_2) + 
(r-3)! \cdot (c(yw) - c(xw)) \cdot d_H(a_1,A_3) \\&+ (r-4)! \cdot \sum_{i = 4}^r (c(yv_i) - c(xv_i)) \cdot d_H(a_1,A_4 \cup \dots \cup A_r).
\end{align*}
By the same argument,
\begin{align*}
\sum_{\sigma}(c(Y_{\sigma}) - c(Y'_{\sigma})) = \; &(r-3)! \cdot (c(yz) - c(xz)) \cdot d_H(a_1,A_3) + 
(r-3)! \cdot (c(yw) - c(xw)) \cdot d_H(a_1,A_2) \\&+ (r-4)! \cdot \sum_{i = 4}^r (c(yv_i) - c(xv_i)) \cdot d_H(a_1,A_4 \cup \dots \cup A_r).
\end{align*}
Hence,
\begin{align*}
c(F_1) - c(F_2) &= 
\sum_{\sigma}(c(X_{\sigma}) - c(X'_{\sigma})) - 
\sum_{\sigma}(c(Y_{\sigma}) - c(Y'_{\sigma})) \\ &= 
(r-3)! \cdot \big(d_H(a_1,A_2) - d_H(a_1,A_3) \big) \cdot \big( c(yz) - c(xz) - c(yw) + c(xw) \big) \neq 0,
\end{align*}
using that $d_H(a_1,A_2) \neq d_H(a_1,A_3)$ by assumption and that $c(yz) - c(xz) = 2$ and $c(yw) - c(xw) = -2$. As $c(F_1) \neq c(F_2)$, it follows that $(L_1\cup L_2, c)$ is a template for $H$. 
\end{proof}
\subsection{Unbalanced $r$-colorings}
In this section we consider graphs $H$ having an unbalanced $r$-coloring (recall that an $r$-coloring $A_1,\dots,A_r$ is called balanced if $|A_1| = \dots = |A_r|$).
We shall prove the following.
\begin{lemma}\label{non-balanced-uniform template2}
Suppose that $r \geq 4$. Let $L_1,L_2$ be two copies of $K_r$ sharing $r-2$ vertices, and let $c$ be a $2$-edge-coloring of $L_1\cup L_2$ such that $L_1^+$ is $d$-regular and  $L_2^+$ is $d'$-regular for some $d\neq d'$. If $H$ fulfills the $r$-wise $C_4$-condition and has an unbalanced $r$-coloring, then $(L_1\cup L_2,c)$ is a template for $H$.
\end{lemma} 
Let us give an overview of the proof of Lemma~\ref{non-balanced-uniform template2}. 
Assuming that $H$ satisfies the $r$-wise $C_4$-condition allows us to control the discrepancy of $H$-factors via Lemma~\ref{regularC4}. Indeed, this lemma implies that the discrepancy of an $H$-factor in any blowup of $K_r$ with a $2$-edge-coloring for which $K_r^+$ is $d$-regular, is simply determined by $d$ and the number of copies of $H$ in the factor. We will make use of this observation by considering two different $H$-factors of a carefully chosen blowup of $L_1 \cup L_2$. Each $H$-copy in both $H$-factors will be contained in the blowup of $L_1$ or $L_2$, and the two $H$-factors will differ on the number of $H$-copies of each type (here we will use that $H$ has an unbalanced $r$-coloring). This will guarantee that the two $H$-factors have different discrepancies. The detailed proof follows.  
\begin{proof}[Proof of Lemma~\ref{non-balanced-uniform template2}]
Let $A_1,A_2,\dots,A_r$ be the parts of an unbalanced $r$-coloring of $H$ with $|A_1|\leq|A_2|\leq\dots\leq|A_r|$ and $|A_1|<|A_r|$. 
Write $L_1 \cap L_2 = \{v_3,v_4,\dots,v_r\}$ and 
$\{x_i,y_i\} = L_i \setminus L_{3-i}$ for $i = 1,2$. 
Let $B$ be the blowup of $(L_1\cup L_2,c)$ with 
\begin{itemize}
    \item $|V_{x_1}|=|V_{x_2}|=|V_{y_1}|=|V_{y_2}|=(|A_1|+|A_2|)(|A_{r-1}|+|A_r|)$ and 
\item for each $1\leq i\leq r-2$, $|V_{v_i}|=2\big(|A_1|+|A_2|\big)|A_{i+2}|+2\big(|A_{r-1}|+|A_r|\big)|A_{i}|.
$
\end{itemize}

Let $F_1,F_2$ be the following two perfect $H$-factors of $B$.
\begin{itemize}
    \item $F_1$ contains $(|A_{r-1}|+|A_r|)$ copies of $H$ with each permutation of $A_1,A_2$ on $V_{x_1},V_{y_1}$ and for each $3\leq i\leq r$, $A_i$ on $V_{v_i}$. Additionally, $F_1$ contains $(|A_{1}|+|A_2|)$ copies of $H$ with each permutation of $A_{r-1},A_r$ on $V_{x_2},V_{y_2}$ and for each $1\leq i\leq r-2$, $A_i$ on $V_{v_{i+2}}$.
    \item $F_2$ contains $(|A_{1}|+|A_2|)$ copies of $H$ with each permutation of $A_{r-1},A_r$ on $V_{x_1},V_{y_1}$ and for each $1\leq i\leq r-2$, $A_i$ on $V_{v_{i+2}}$. Additionally, $F_2$ contains $(|A_{r-1}|+|A_r|)$ copies of $H$ with each permutation of $A_{1},A_2$ on $V_{x_2},V_{y_2}$ and for each $3\leq i\leq r$, $A_i$ on $V_{v_i}$.
\end{itemize}
By definition, each copy of $H$ in $F$ is contained either in the blowup of $L_1$ or in the blowup of $L_2$.
Let $F_{1|L_1}, F_{1|L_2}$ denote the copies of $H$ in $F_1$ which are contained in the blowup of $L_1,L_2$, respectively, and define $F_{2|L_1}, F_{2|L_2}$ similarly.
Since $F_{1|L_1}$ satisfies the $r$-wise $C_4$-condition, we can apply Lemma~\ref{regularC4} to the blowup $B[V(F_{1|L_1})]$ of the $r$-clique $(L_1,c)$. As $L_1^+$ is $d$-regular, Lemma~\ref{regularC4} gives
$$
c(F_{1|L_1})=\frac{2d-r+1}{r-1}e(F_{1|L_1}).
$$
Similarly, we get
\begin{align*}
c(F_{1|L_2})&=\frac{2d'-r+1}{r-1}e(F_{1|L_2}),\\
c(F_{2|L_1})&=\frac{2d-r+1}{r-1}e(F_{2|L_1}),\\
c(F_{2|L_2})&=\frac{2d'-r+1}{r-1}e(F_{2|L_2}).
\end{align*}
Additionally, by the definition of $F_1,F_2$ we have that $e(F_{1|L_1})=e(F_{2|L_2})=2(|A_{r-1}|+|A_r|)e(H)$ and $e(F_{1|L_2})=e(F_{2|L_1})=2(|A_1|+|A_2|)e(H)$. We get that
$$
c(F_1)-c(F_2) = 
c(F_{1|L_1}) - c(F_{2|L_2}) + c(F_{1|L_2}) - c(F_{2|L_1}) = 
\frac{4(d-d')}{r-1}(|A_{r-1}|+|A_r|-|A_1|-|A_2|)e(H) \neq 0,
$$
using that $d\neq d'$ and that $|A_{r-1}|+|A_r|-|A_1|-|A_2|>0$. 
so we see that $c(F_1)-c(F_2)\neq 0$, implying that $(L_1\cup L_2,c)$ is a template for $H$.
\end{proof}
\subsection{Non-uniform $r$-colorings}
An $r$-coloring of $H$ is called {\em uniform} if the number of edges between any two color-classes is the same. 
We will say that $H$ is uniform if it only has uniform colorings:
% \begin{definition}
% We say that a graph $H$ is {\em uniform} if for any $r$-vertex-coloring of $H$ with parts $A_1,A_2,...,A_r$ it holds that $e_H(A_i,A_j)= e(H)/\binom{r}{2}$ for all $1\leq i<j\leq r$.
% \end{definition}
\begin{definition}\label{def:uniform}
A graph $H$ is called {\em uniform} if for every $r$-coloring of $H$ with parts $A_1,A_2,...,A_r$, it holds that $e_H(A_i,A_j)= e(H)/\binom{r}{2}$ for all $1\leq i<j\leq r$.
\end{definition}
\noindent
The following lemma gives a template for non-uniform graphs $H$ under some additional conditions. 

\begin{lemma}\label{sharing_little_template}
Let $L_1, L_2$ be two copies of $K_r$ sharing $r-2$ or $r-1$ vertices, let $V\subseteq L_1 \cap L_2$ of size $r-2$ and let $e_i := L_i \setminus V$.
Let $c$ be a $2$-edge-coloring of $L_1\cup L_2$ such that
$$
c(L_1) - c(e_1) - c(L_2) + c(e_2)\notin\{-4(r-2), -2(r-2), 0, 2(r-2), 4(r-2)\}.
$$
If $H$ fulfills the $r$-wise $C_4$-condition and is non-uniform, then $(L_1\cup L_2,c)$ is a template for $H$.
\end{lemma}

Before proving Lemma \ref{sharing_little_template}, let us note some simple facts. 
Clearly, a non-uniform coloring must have at least three parts, so $r \geq 3$. We also need the following easy claim.
\begin{claim}\label{claim:non-uniform}
if $H$ is non-uniform, then there exists an $r$-coloring of $H$ with parts $A_1,A_2\dots,A_r$ such that $e_H(A_1,A_2)\neq e_H(A_1,A_3)$. 
\end{claim}
\begin{proof}
Fix any non-uniform $r$-coloring $B_1,\dots,B_r$ of $H$, and let $1\leq i<j\leq r$ and $1\leq k<\ell\leq r$ such that $e_H(B_i,B_j)\neq e_H(B_k,B_\ell)$.
Then $i=k$ or $e_H(B_i,B_j)\neq e_H(B_i,B_k)$ or $e_H(B_i,B_k)\neq e_H(B_k,B_\ell)$, and by renaming the parts we get the desired $r$-coloring $A_1,\dots,A_r$ with $e_H(A_1,A_2)\neq e_H(A_1,A_3)$.    
\end{proof}

\begin{proof}[Proof of Lemma~\ref{sharing_little_template}]
% Let $L_1,L_2,V,c$ be as in the statement. 
Write $e_i = x_iy_i$, $i = 1,2$. Note that $\{x_1,y_1\}$ may intersect $\{x_2,y_2\}$ (namely, if $|L_1 \cap \nolinebreak L_2| = r-1$). Without loss of generality, let us assume that if these sets intersect then $x_1=x_2$, so that $y_1\neq y_2$ always holds.
Let $S$ be the bipartite graph between $\{x_1,y_1\}$ and $V$ (so $S \subseteq L_1$), and let $T$ be the bipartite graph between $\{x_2,y_2\}$ and $V$ (so $T \subseteq L_2$). 
Observe that $$c(S) - c(T) = c(L_1) - c(e_1) - c(L_2) + c(e_2),$$ so 
\begin{equation}\label{eq:c(S)-c(T)}
c(S) - c(T) \notin\{-4(r-2), -2(r-2), 0, 2(r-2), 4(r-2)\}.
\end{equation}
by assumption. 
Also, as $S$ and $T$ contain $2(r-2)$ edges each, we have $|c(S) - c(T)| \leq 4(r-2)$. Thus, \eqref{eq:c(S)-c(T)} implies that 
$c(S)-c(T)$ is not a multiple of $2(r-2)$.
For $v \in V$, define $g(v) := c(vx_1) + c(vy_1) - c(vx_2) - c(vy_2)$. We claim that there are $u,v \in V$ with $g(u) \neq g(v)$. Indeed, suppose otherwise. 
Then there exists $g\in\mathbb{N}$ such that each vertex $v\in V$ has $g(v) = g$. Then
\begin{align*}
    c(S) - c(T) &= \sum_{v\in V}g(v)
                    = g(r-2).
\end{align*}
Note also that $g$ is even since $g(v)$ is an even number for every $v\in V$. But now we see that $c(S) - c(T)$ is a multiple of $2(r-2)$, a contradiction. 
We conclude that there exist vertices $u,v\in V$ with $g(u) \neq g(v)$. Note that this also implies that $r\geq 4$. Write $V = \{u,v,v_5,v_6,\dots,v_r\}$. By Claim~\ref{claim:non-uniform}, there is an $r$-coloring $A_1,A_2,\dots,A_r$ of $H$ with $e_H(A_1,A_2)\neq e_H(A_1,A_3)$. Let $B$ be a blowup of $(L_1\cup L_2,c)$ with
\begin{itemize}
    \item $|V_{x_1}| = |V_{x_2}| = |A_1|$,
    \item $|V_{y_1}| = |V_{y_2}| = |A_4|$,
    \item $|V_{u}| = |V_{v}| = |A_2|+|A_3|$,
    \item $|V_{v_i}| = 2|A_i|$ for every $5\leq i\leq r$. 
\end{itemize}
Here, $V_{x_1},V_{x_2}$ are distinct parts even if $x_1 = x_2$. Note that $B$ is a blowup of $H$ also in the case $x_1 = x_2$, as then $V_{x_1} \cup V_{x_2}$ is the part corresponding to the vertex $x_1 = x_2$.

We now define four $H$-copies in $B$, denoted $F_{1,1},F_{1,2},F_{2,1},F_{2,2}$. All four copies embed $A_i$ into $V_{v_i}$ for every $5 \leq i \leq r$. Also, $F_{1,1}$ and $F_{2,1}$ embed $A_1$ into $V_{x_1}$ and $A_4$ into $V_{y_1}$, while $F_{1,2}$ and $F_{2,2}$ embed $A_1$ into $V_{x_2}$ and $A_4$ into $V_{y_2}$. Finally, $F_{1,1}$ and $F_{2,2}$ embed $A_2$ into $V_u$ and $A_3$ into $V_v$, while $F_{1,2}$ and $F_{2,1}$ embed $A_2$ into $V_v$ and $A_3$ into $V_u$. We define these copies such that $F_i := \{F_{i,1},F_{i,2}\}$ forms a perfect $H$-factor of $B$ for every $i=1,2$; this is possible due to our choice of the cluster sizes of $B$. 

We now show that $c(F_1)-c(F_2) \neq 0$. 
First observe that for an edge $zw \in E(L_1 \cup L_2)$, if $zw \notin \{x_1,y_1,x_2,y_2\} \times \{u,v\}$ then $e_{F_1}(V_z,V_w) = e_{F_2}(V_z,V_w)$. Indeed, if $z,w \in \{v_5,\dots,v_r\}$ then this is immediate. If $z \in \{x_1,y_1,x_2,y_2\}$ and $w \in \{v_5,\dots,v_r\}$ then $e_{F_{1,1}}(V_z,V_w) = e_{F_{2,1}}(V_z,V_w)$ and 
$e_{F_{1,2}}(V_z,V_w) = e_{F_{2,2}}(V_z,V_w)$, so the assertion holds. If $z \in \{u,v\}$ and $w \in \{v_5,\dots,v_r\}$ then $e_{F_{1,1}}(V_z,V_w) = e_{F_{2,2}}(V_z,V_w)$ and 
$e_{F_{1,2}}(V_z,V_w) = e_{F_{2,1}}(V_z,V_w)$, so again the assertion holds. The case $\{z,w\} = \{u,v\}$ is also immediate. We see that the bipartite graph $(V_z,V_w)$ does not contribute to $c(F_1)-c(F_2)$ unless $zw \in \{x_1,y_1,x_2,y_2\} \times \{u,v\}$.
By definition, we have
\begin{align*}
c(F_1[V_{x_1}\cup V_{x_2} \cup V_{y_1}\cup V_{y_2},V_u\cup V_v]) =& (c(ux_1)+c(vx_2)) \cdot e_H(A_1,A_2)+(c(uy_1)+c(vy_2)) \cdot e_H(A_4,A_2)\\
&+(c(vx_1)+c(ux_2)) \cdot e_H(A_1,A_3)+(c(vy_1)+c(uy_2)) \cdot e_H(A_4,A_3),
\end{align*}
and similarly,
\begin{align*}
c(F_2[V_{x_1}\cup V_{x_2} \cup V_{y_1}\cup V_{y_2},V_u\cup V_v]) =& (c(vx_1)+c(ux_2)) \cdot e_H(A_1,A_2)+(c(vy_1)+c(uy_2)) \cdot e_H(A_4,A_2)\\
&+(c(ux_1)+c(vx_2)) \cdot e_H(A_1,A_3)+(c(uy_1)+c(vy_2)) \cdot e_H(A_4,A_3).
\end{align*}
It follows that
\begin{align*}
    c(F_1)-c(F_2) =& (c(ux_1)+c(vx_2)-c(vx_1)-c(ux_2)) \cdot (e_H(A_1,A_2)-e_H(A_1,A_3))\\
    &+(c(uy_1)+c(vy_2)-c(vy_1)-c(uy_2)) \cdot (e_H(A_4,A_2)-e_H(A_4,A_3)).
\end{align*}
By the $C_4$-condition, we also have that
$$
e_H(A_1,A_2)+e_H(A_4,A_3)-e_H(A_1,A_3)-e_H(A_4,A_2) = 0.
$$
Adding this multiplied by $(c(uy_1)+c(vy_2)-c(vy_1)-c(uy_2))$ to the previous equation, we get
\begin{align*}
   c(F_1)-&c(F_2) = \\&(c(ux_1)+c(vx_2)-c(vx_1)-c(ux_2)+c(uy_1)+c(vy_2)-c(vy_1)-c(uy_2))\cdot (e_H(A_1,A_2)-e_H(A_1,A_3))\\
   =& \; (g(u) - g(v)) \cdot (e_H(A_1,A_2)-e_H(A_1,A_3)) \neq 0,
   %(d_S(u)-d_T(u)-(d_S(v)-d_T(v)))(e_H(A_1,A_3)-e_H(A_1,A_4)).
\end{align*}
using that $g(u) \neq g(v)$ and that $e_H(A_1,A_2)\neq e_H(A_1,A_3)$ by assumption. So $c(F_1) \neq c(F_2)$ and hence $(L_1\cup L_2,c)$ is a template for $H$.
\end{proof}

\subsection{Graphs violating the $C_4$-condition}

Here we show that if $H$ violates the $k$-wise $C_4$-condition and $c$ is a coloring of $K_k$, then $(K_k,c)$ is a template for $H$ unless $c$ has some very specific structure. This is given by the following definition and lemma.
\begin{definition}\label{def:star}
For some $k\geq1$, $K_{k,+}$ is a copy of $(K_k,c)$ where $c$ is a $2$-edge-coloring of $K_k$ with all edges of color $1$. The \emph{$(K_k,+)$-star} is the copy of $(K_k,c)$ where $c$ is a $2$-edge-coloring of $K_k$ such that the edges of color $1$ induce a star with $k-1$ leaves. We call the root of this star the \emph{head} of the $(K_k,+)$-star. Define $K_{k,-}$ and the $(K_k,-)$-star analogously.
\end{definition}
\begin{lemma}\label{no_c4_template}
Let $k\geq 4$ and let $c$ be a $2$-edge-coloring of $K_k$ such that $(K_k,c)$ is neither monochromatic nor a star. If $H$ does not fulfill the $k$-wise $C_4$-condition, then $(K_k,c)$ is a template for $H$.
\end{lemma}
\begin{proof}
We need the following claim, which appears implicitly in \cite[proof of Claim 6.4]{balogh2021discrepancy}. For completeness, we give a proof.
\begin{claim}
There exist vertices $a_1,a_2,a_3,a_4\in V(K_k)$ such that
$$
c(a_1a_2)+c(a_3a_4) \neq c(a_1a_3)+ c(a_2a_4).
$$
\end{claim}
\begin{proof}
Let us assume towards a contradiction that for every $a_1,a_2,a_3,a_4\in V(K_k)$ it holds that 
\begin{equation}\label{eq_c4}
  c(a_1a_2)+c(a_3a_4) = c(a_1a_3)+ c(a_2a_4).  
\end{equation}
Fix an arbitrary vertex $a\in V(K_k)$.
If $d_{K_k^+}(a)\geq 3$, let $b,c,d\in N_{K_k^+}(a)$ be three arbitrary vertices. By (\ref{eq_c4}), it follows that $c(bc)=c(cd)$. As this holds for arbitrary $b,c,d\in N_{K_k^+}(a)$, we get that either all the edges in $K_k[N_{K_k^+}(a)]$ have color $1$ or they all have color $-1$. Therefore, $N_{K_k^-}(a)$ is not empty as otherwise, $K_k$ colored by $c$ is monochromatic or a star. By symmetry, if $d_{K_k^-}(a)\geq 3$ then $N_{K_k^+}(a)$ is non-empty. 

We have $d_{K_k^+}(a) + d_{K_k^-}(a) = k-1 \geq 3$.
Without loss of generality, let us assume that $d_{K_k^+}(a)\geq 2$. Let $b,c\in N_{K_k^+}(a)$ and $d\in N_{K_k^-}(a)$. By (\ref{eq_c4}), we get that $c(bc)=1$ and $c(bd)=c(cd)-1$. Since this holds for arbitrary $b,c\in N_{K_k^+}(a)$ and $d\in N_{K_k^-}(a)$, we get that all the edges in $K_k[N_{K_k^+}(a)]$ have color $1$ and all the edges in $K_k[N_{K_k^+}(a),N_{K_k^-}(a)]$ have color $-1$. Thus, if $d_{K_k^-}(a)=1$ then $K_k$ colored by $c$ is a copy of a $(K_k,-)$-star (whose head is the unique vertex in $N_{K_k^-}(a)$). 
So $d_{K_k^-}(a)\geq 2$. Now, by a symmetrical argument to the above, we get that all the edges in $K_k[N_{K_k^+}(a),N_{K_k^-}(a)]$ have color $1$, which is a contradiction.
\end{proof}
Let $V(K_k) = \{a_1,a_2,\dots, a_{k}\}$ with $a_1,a_2,a_3,a_4$ as in the statement of the above claim. As $H$ violates the $k$-wise $C_4$-condition, there exists a $k$-coloring $A_1,A_2,\dots, A_{k}$ of $H$ with
$$
e_H(A_1A_2)+e_H(A_3A_4)\neq e_H(A_1A_3)+e_H(A_2A_4).
$$
Let $B$ be a blowup of $(K_k,c)$ with 
\begin{itemize}
    \item $|V_{a_1}|=|V_{a_4}| = |A_1|+|A_4|$,
    \item $|V_{a_2}|=|V_{a_3}| = |A_2|+|A_3|$, and
    \item $|V_{a_i}| = 2|A_i|$ for $5\leq i\leq k$.
\end{itemize}

We now define two perfect $H$-factors $F_1$ and $F_2$ of $H$, each consisting of two copies of $H$. Each $H$-copy in $F_1$ and $F_2$ embeds $A_i$ into $V_{a_i}$ for every $5 \leq i \leq k$. 
\begin{itemize}
    \item $F_1$ contains a copy of $H$ with $A_1$ on $V_{a_1}$, $A_2$ on $V_{a_2}$, $A_3$ on $V_{a_3}$ and $A_4$ on $V_{a_4}$ and a second copy of $H$ with $A_1$ on $V_{a_4}$, $A_2$ on $V_{a_3}$, $A_3$ on $V_{a_2}$ and $A_4$ on $V_{a_1}$.
    \item $F_2$ contains a copy of $H$ with $A_1$ on $V_{a_1}$, $A_2$ on $V_{a_3}$, $A_3$ on $V_{a_2}$ and $A_4$ on $V_{a_4}$ and a second copy of $H$ with $A_1$ on $V_{a_4}$, $A_2$ on $V_{a_2}$, $A_3$ on $V_{a_3}$ and $A_4$ on $V_{a_1}$.
\end{itemize}
We now calculate $c(F_1) - c(F_2)$. Note that for each $5\leq i<j\leq k$, we have that 
$$
e_{F_1}(V_{a_i},V_{a_j}) =2e_H(A_i,A_j) = e_{F_2}(V_{a_i},V_{a_j}).
$$
For $5\leq i\leq k$ and $j\in \{1,4\}$, we get
$$
e_{F_1}(V_{a_i},V_{a_j}) =e_H(A_i,A_1)+e_H(A_i,A_4) = e_{F_2}(V_{a_i},V_{a_j}),
$$
and for $j\in \{2,3\}$
$$
e_{F_1}(V_{a_i},V_{a_j}) =e_H(A_i,A_2)+e_H(A_i,A_3) = e_{F_2}(V_{a_i},V_{a_j}).
$$
Additionally, we have
$$
e_{F_1}(V_{a_1},V_{a_4}) =2e_H(A_1,A_4) = e_{F_2}(V_{a_1},V_{a_4}),
$$
$$
e_{F_1}(V_{a_2},V_{a_3}) =2e_H(A_2,A_3) = e_{F_2}(V_{a_2},V_{a_3}),
$$
$$
e_{F_1}(V_{a_1},V_{a_2}) = e_{F_1}(V_{a_3},V_{a_4}) =e_{F_2}(V_{a_1},V_{a_3}) = e_{F_2}(V_{a_2},V_{a_4})= e_H(A_1,A_2) + e_H(A_3,A_4),
$$
and
$$
e_{F_2}(V_{a_1},V_{a_2}) = e_{F_2}(V_{a_3},V_{a_4}) =e_{F_1}(V_{a_1},V_{a_3}) = e_{F_1}(V_{a_2},V_{a_4}) = e_H(A_1,A_3) + e_H(A_2,A_4).
$$
Combining all of the above, we get that
\begin{align*}
&c(F_1)-c(F_2) =
% c(a_1a_2)(e_{F_1}(V_{a_1},V_{a_2})-e_{F_2}(V_{a_1},V_{a_2}))+c(a_3a_4)(e_{F_1}(V_{a_3},V_{a_4})-e_{F_2}(V_{a_3},V_{a_4}))\\
% &+c(a_1a_3)(e_{F_1}(V_{a_1},V_{a_3})-e_{F_2}(V_{a_1},V_{a_3}))+c(a_2a_4)(e_{F_1}(V_{a_2},V_{a_4})-e_{F_2}(V_{a_2},V_{a_4}))\\
\sum_{i \in \{1,4\}}\sum_{j \in \{2,3\}}c(a_ia_j) \cdot (e_{F_1}(V_{a_i},V_{a_j})-e_{F_2}(V_{a_i},V_{a_j})) = \\
&(c(a_1a_2)+c(a_3a_4)-c(a_1a_3)-c(a_2a_4)) \cdot (e_H(A_1,A_2)+e_H(A_3,A_4)-e_H(A_1,A_3)-e_H(A_2,A_4))
\neq 0.
\end{align*}
Therefore, $(K_k,c)$ is a template for $H$.
\end{proof}
\subsection{Balanced $H$-tilings with non-uniform edge distribution}
In the previous section, we saw that if $H$ violates the $r$-wise $C_4$-condition, then $(K_r,c)$ is a template for $H$ for every coloring $c$ of $K_r$ except for some very special cases. In this section we continue investigating when a given coloring of $K_r$ is a template for $H$, this time assuming that $H$ does satisfy the $r$-wise $C_4$-condition.  
We shall see that if $c$ is a coloring of $K_r$ such that $K_r^+$ is not regular, then $(K_r,c)$ is a template for $H$ unless every $H$-factor, i.e. disjoint union of copies of $H$, has a certain uniformity property. The precise statement is given by the following definition and lemma. Recall that by ``$H$-factor'' we simply mean a graph consisting of vertex-disjoint copies of $H$. 

\begin{definition}\label{def:balanced-uniform}
We say that an $r$-chromatic graph $F$ is {\em balanced-uniform} if every balanced $r$-coloring of $F$ with parts $A_1,A_2,...,A_r$ satisfies that $e_F(A_i,A_j)= e(F)/\binom{r}{2}$ for all $1\leq i<j\leq r$.
\end{definition}
\begin{lemma}\label{non-balanced-uniform template}
Suppose that $r\geq 4$ and $r \neq 5$, and assume that $H$ satisfies the $r$-wise $C_4$-condition. Let $c$ be a $2$-edge-coloring of $K_r$ such that $K_r^+$ is non-regular. If there exists a non-balanced-uniform $H$-factor, then $(K_r,c)$ is a template for $H$.
\end{lemma}
The condition that $K_r^+$ is non-regular is necessary; indeed, it follows from Lemma~\ref{regularC4} that $(K_r,c)$ is not a template for $H$ if $K_r^+$ is regular and $H$ fulfills the $r$-wise $C_4$-condition. 
Also, the assumption that there exists a non-balanced-uniform $H$-factor is necessary for the proof method. 
Indeed, in the proof of Lemma~\ref{non-balanced-uniform template}
we show that the {\em balanced} blowup $B$ of $(K_r,c)$ has two $H$-factors with different discrepancies, (this showing that $(K_r,c)$ is a template). However, if $H$-factor were balanced-regular, then every perfect $H$-factor $F$ of $B$ would have exactly $e(F)/\binom{r}{2}$ edges between any two parts of $B$ (as $B$ is balanced), and so $c(F) = c(K_r) \cdot e(F)/\binom{r}{2}$, meaning that every perfect $H$-factor of $B$ would have the same discrepancy. 

For the proof of Lemma~\ref{non-balanced-uniform template} we need the following simple claim. 
\begin{claim}\label{claim:non-balanced-uniform}
Suppose that $r\geq 4$ and $H$ satisfies the $r$-wise $C_4$-condition. Let $F$ be a non-balanced-uniform $H$-factor. 
Then there exists a balanced $r$-coloring $A_1,\dots,A_r$ of $F$ such that $e_F(A_1,A_2)\neq e_F(A_3,A_4)$.
\end{claim}
\begin{proof}
By definition, there is a balanced $r$-coloring $A_1,\dots,A_r$ of $F$ such that $e_F(A_i,A_j)$ are not all equal. By renaming the parts, we may assume that $e_F(A_1,A_2)>e(F)/\binom{r}{2}$. Then there are $1\leq i<j\leq r$ such that $e_F(A_i,A_j)<e(F)/\binom{r}{2}$. Now, if $i,j \notin \{1,2\}$ then we are done. Else, suppose without loss of generality that $i=1$. Let $k \in [r] \setminus \{1,2,j\}$. We may assume that $e_F(A_1,A_2) = e_F(A_j,A_k)$ and $e_F(A_1,A_j) = e_F(A_2,A_k)$, as otherwise we are done. Now we get that $$e_F(A_1,A_2) + e_F(A_j,A_k)-e_F(A_1,A_j) -e_F(A_2,A_k) = 2e_F(A_1,A_2) - 2e_F(A_i,A_j) > 0,$$ contradicting the assumption that $H$ fulfills the $r$-wise $C_4$-condition.
\end{proof}
\begin{proof}[Proof of Lemma~\ref{non-balanced-uniform template}]
Let $c$ be as in the statement. 
For a vertex $u \in V(K_r)$, denote
$$c(u,K_r) := \sum_{v \in V(K_r) \setminus \{u\}}c(uv).$$
Let $a,b\in V(K_r)$ be the two vertices maximizing $c(a,K_r)+c(b,K_r)$ and $d,e\in V(K_r)$ the vertices minimizing $c(d,K_r)+c(e,K_r)$. In other words, $a,b$ are the two vertices with highest degree in $K_r^+$, and $d,e$ are the two vertices with lowest degree. 
As $r\geq 4$, we may assume that $a,b,d,e$ are distinct. Since $K_r^+$ is non-regular, 
\begin{equation}\label{eq:abde}
    c(a,K_r)+c(b,K_r)>c(d,K_r)+c(e,K_r).
\end{equation}
Let $F$ be a non-balanced-uniform $H$-factor. By Claim~\ref{claim:non-balanced-uniform}, there is a balanced $r$-coloring $A_1,A_2,\dots,A_r$ of $F$ such that $e_F(A_1,A_2)\neq e_F(A_3,A_4)$.
Let $B$ be a $4(r-4)!\frac{|F|}{r}$-blowup of $(K_r,c)$. Let $F_1$ and $F_2$ be the following two perfect $F$-factors of $B$. 
\begin{itemize}
    \item $F_1$ contains each copy of $F$ with each permutation of $A_1,A_2$ on $V_a,V_b$, each permutation of $A_3, A_4$ on $V_d, V_e$, and each permutation of $A_5,A_6,...,A_r$ on the remaining clusters of $B$.
    \item $F_2$ contains each copy of $F$ with each permutation of $A_1,A_2$ on $V_d,V_e$, each permutation of $A_3,A_4$ on $V_a,V_b$, and each permutation of $A_5,A_6,...,A_r$ on the remaining clusters of $B$.
\end{itemize}
Note that $F_1,F_2$ each have $4(r-4)!$ copies of $F$. As $|A_1| = \dots = |A_r| = \frac{|F|}{r}$, the size of $B$ exactly matches these $F$-factors. 
Clearly, every $F$-factor is also an $H$-factor, as $F$ is a union of disjoint copies of $H$. Hence $F_1,F_2$ are perfect $H$-factors of $B$. 
We now show that $c(F_1)-c(F_2) \neq 0$. Put $X = V(K_r)\backslash\{a,b,d,e\}.$
By symmetry, we have for all $v \in X$:
\begin{itemize}
    \item $e_{F_1}(V_a,V_v)=e_{F_1}(V_b,V_v)=e_{F_2}(V_d,V_v)=e_{F_2}(V_e,V_v) \eqqcolon y_1,$
    \item $e_{F_2}(V_a,V_v)=e_{F_2}(V_b,V_v)=e_{F_1}(V_d,V_v)=e_{F_1}(V_e,V_v) \eqqcolon y_2.$
\end{itemize}
If $r=4$ then $X$ is empty, and we set $y_1=y_2=0$. 
We also have
$$e_{F_1}(V_u,V_v) = e_{F_2}(V_u,V_v)$$
for all $u,v \in X$ 
and all $u \in \{a,b\}, v \in \{d,e\}$.
Finally, observe that 
\begin{equation*}
e_{F_1}(V_a,V_b) =  e_{F_2}(V_d,V_e) = 4(r-4)! \cdot e_F(A_1,A_2), \; \; \; 
e_{F_1}(V_d,V_e) = e_{F_2}(V_a,V_b) = 4(r-4)! \cdot e_F(A_3,A_4).
\end{equation*}
Therefore, $e_{F_1}(V_a,V_b) \neq
e_{F_1}(V_d,V_e)$ holds by our choice of $A_1,A_2,A_3,A_4$.
For convenience, let us denote 
$$h(u,v) := c(uv) \cdot ( e_{F_1}(V_u,V_v) - e_{F_2}(V_u,V_v))$$ for $uv \in E(K_r)$. From the above, we see that $h(u,v)$ is non-zero only if $uv \in \{ab,cd\}$ or $v \in X, u \in \{a,b,c,d\}$. Also, in the latter case we have 
$h(u,v) = y_1-y_2$ if $u \in \{a,b\}$ and $h(u,v) = y_2 - y_1$ if $u \in \{c,d\}$. 
Finally, $h(a,b) = -h(d,e) = 
(c(ab) - c(de)) \cdot (e_{F_1}(V_a,V_b) - e_{F_1}(V_d,V_e)).$
Using these equations, we get 
\begin{equation}\label{eq: uniform f1-f2 prelim}
    \begin{split}
    &c(F_1)-c(F_2)= \sum_{uv \in E(K_r)}h(u,v) =  
   \\ &(c(ab) - c(de)) \cdot (e_{F_1}(V_a,V_b) - e_{F_1}(V_d,V_e)) + (y_1-y_2)\sum_{v \in X}(c(av) + c(bv) - c(dv) - c(ev)) 
    \end{split}
\end{equation}
Observe that 
\begin{align*}
    \sum_{v \in X}(c(av) + c(bv) - c(dv) - c(ev)) = c(a,K_r)+c(b,K_r)-2c(ab)-c(d,K_r)-c(e,K_r)+2c(de). 
\end{align*}
Plugging this into \eqref{eq: uniform f1-f2 prelim} and rearranging, we get
\begin{equation}\label{eq: uniform f1-f2}
   \begin{split}
        c(F_1)-c(F_2) = (c(ab)-c(de)) &\cdot (e_{F_1}(V_a,V_b)-e_{F_1}(V_d,V_e)+2y_2-2y_1) \\ &+ (y_1-y_2) \cdot (c(a,K_r)+c(b,K_r)-c(d,K_r)-c(e,K_r)).
   \end{split}
\end{equation}
Suppose first that $r = 4$. Then $y_1 = y_2 = 0$. Also, $c(ab)\neq c(de)$, because otherwise we would have $c(a,K_r)+c(b,K_r)=c(K_r)=c(d,K_r)+c(e,K_r)$, contradicting \eqref{eq:abde}. Now, using \eqref{eq: uniform f1-f2}, we have
$$c(F_1)-c(F_2) = (c(ab)-c(de)) \cdot (e_{F_1}(V_a,V_b)-e_{F_1}(V_d,V_e)) \neq 0,$$ 
as required. 

As $r \neq 5$, we may assume from now on that $r\geq 6$. 
Since $H$ satisfies the $r$-wise $C_4$-condition, so do $F_1,F_2$ (as $F_1,F_2$ are $H$-factors).
Fix arbitrary $u,v\in X$.
By the $C_4$-condition, we have
$$
e_{F_1}(V_a,V_b)+e_{F_1}(V_u,V_v)-e_{F_1}(V_a,V_u)-e_{F_1}(V_b,V_v)=0
$$
and
$$
-e_{F_1}(V_d,V_e)-e_{F_1}(V_u,V_v)+e_{F_1}(V_d,V_u)+e_{F_1}(V_e,V_v)=0.
$$
Recall that $e_{F_1}(V_a,V_u)=e_{F_1}(V_b,V_v)=y_1$ and $e_{F_1}(V_d,V_u)=e_{F_1}(V_e,V_v)=y_2$. Adding the two above equations, we get
\begin{align*}
    0 =e_{F_1}(V_a,V_b)+e_{F_1}(V_u,V_v)-2y_1-e_{F_1}(V_d,V_e)-e_{F_1}(V_u,V_v)+2y_2
    = e_{F_1}(V_a,V_b)-2y_1-e_{F_1}(V_d,V_e)+2y_2.
\end{align*}
As $e_{F_1}(V_a,V_b) \neq e_{F_1}(V_d,V_e)$, we have $y_1 \neq y_2$. Finally, plugging this into \eqref{eq: uniform f1-f2}, we get
\begin{align*}
   c(F_1)-c(F_2) = (y_1-y_2) \cdot (c(a,K_r)+c(b,K_r)-c(d,K_r)-c(e,K_r)) \neq 0,
\end{align*}
using \eqref{eq:abde}. Therefore, $(K_r,c)$ is a template for $H$.
\end{proof}

We end this section with the following lemma, showing that if all $H$-factors are balanced-uniform, then no coloring of $K_r$ is a template for $H$.  
\begin{lemma}\label{lem:if balanced-uniform then not a template}
    If every $H$-factor is balanced-uniform, then for every $2$-edge-coloring $c$ of $K_r$ it holds that $c\in\mathcal{K}(H)$.
\end{lemma}
\begin{proof}
    Assume towards a contradiction that there exists a $2$-edge-coloring $c$ of $K_r$ such that $(K_r,c)$ is a template for $H$. Then, by definition, there exists a blowup $B$ of $(K_r,c)$ and two perfect $H$-factors $F_1,F_2$ of $B$ such that $c(F_1)\neq c(F_2)$. Let 
    $v_1,v_2,\dots, v_r$ be the vertices of $K_r$ and let $B_i$ be the part of $B$ corresponding to $v_i$. Take a balanced blowup $B'$ of $(K_r,c)$ with parts $B'_1,\dots,B'_r$ of size $|B| = |B_1| + \dots + |B_r|$ each. For $0 \leq j \leq r-1$, let $J_j$ be a copy of $B$ in $B'$ in which $B_i$ is mapped to $B'_{i+j \pmod{r}}$ for each $i = 1,\dots,r$. Choose $J_1,\dots,J_r$ to be vertex-disjoint and to partition $V(B')$; this is possible as each part of $B'$ has size $|B|$. 
    
    For each $1 \leq j \leq r$, let $F^{j}_1,F^{j}_2$ be the $H$-factors playing the roles of $F_1,F_2$, respectively, in the copy $J_j$ of $B$. 
    Note that in $J_1$, $B_i$ is mapped to $B'_i$ for every $1 \leq i \leq r$. This means that $J \subseteq B'$ is colored the same way as $B$ (i.e., $J$ is a blowup of $(K_r,c)$). Therefore, $c(F^{1}_i) = c(F_i)$ for $i = 1,2$ (with a slight abuse of notation, we denote by $c$ both the coloring of $B$ and that of $B'$). 
    Let $F'_1 := \bigcup_{j=1}^r F^{j}_1$, and let $F'_2$ be obtained from $F'_1$ by replacing $F^{1}_1$ with $F^{1}_2$, i.e. 
    $F'_2 := F^{1}_2 \cup \bigcup_{j=2}^r F^{j}_1$. Both $F'_1,F'_2$ are perfect $H$-factors of $B'$. 
    Also,
    $$c(F_1')-c(F_2')=c(F_1)-c(F_2)\neq 0.$$
    Note, however, that $F'_1,F'_2$ are both $H$-factors with a balanced $r$-coloring $B'_1,\dots,B'_r$. As $F'_1,F'_2$ are balanced-uniform by assumption, it follows that $e_{F'_1}(B'_i,B'_j) = e_{F'_2}(B'_i,B'_j) = e(F'_1)/\binom{r}{2}$ for all $1 \leq i < j \leq r$. However, this implies that 
    $$
    c(F'_1) = c(F'_2) = c(K_r) \cdot \frac{e(F'_1)}{\binom{r}{2}},
    $$
    a contradiction. 
\end{proof}

\begin{remark}
    Given $H,$ it is possible to computationally check whether there exists a non-balanced-uniform $H$-factor. Like in the proof of Proposition~\ref{prop:compute delta0}, for every $r$-coloring $f$ of $H,$ one defines vectors $x_f$ and $y_f$ corresponding to the sizes of the color classes and the number of edges between color classes, respectively. Then, the task reduces to     a certain linear-algebraic statement regarding these vectors. 
    % We omit the details. 
    % However, this is not needed to determine $\delta^*(H)$ so we omit the details.
\end{remark}
    
\subsection{$(s,t)$-structured graphs}
The templates considered in this section consist of two $r$-cliques $L_1,L_2$ sharing $r-2$ vertices. 
We will describe the structure of graphs $H$ for which a particular coloring of $L_1 \cup L_2$ is not a template. This structure is more involved than in previous sections. The precise definition is as follows.
\begin{definition}\label{structured}
    For $s, t, \rho \in \mathbb{R}$, we say that $H$ is $(s,t)$-\emph{structured} with parameter $\rho$ if for every $r$-vertex-coloring of $H$ with parts $A_1, \dots, A_r$ and for all $1 \leq i < j \leq r$, it holds that
    \begin{equation}\label{eq:structured}
        \rho(|A_i| + |A_j|) = s \cdot e_H(A_i \cup A_j, V(H) \setminus (A_i \cup A_j)) + t \cdot e_H(A_i, A_j).
    \end{equation}
\end{definition}
We say that $(s,t)$-structured to mean that it is $(s,t)$-structured with some parameter $\rho$.
Note that if $H$ is $(s,t)$-structured then it is also $(\alpha\cdot s,\alpha\cdot t)$-structured for every $\alpha\in \mathbb{R}$. 
Another important fact is that if $H$ is $(s,t)$-structured with parameter $\rho$, then so is every $H$-factor.

The following lemma is the main result of this subsection. 

\begin{lemma}\label{template_example}
Let $L_1, L_2$ be two copies of $K_r$ sharing $r-2$ vertices, and let $e_1 = x_1y_1 = L_1 \setminus L_2$ and $e_2 = x_2y_2 = L_2 \setminus L_1$. Let $c$ be a $2$-edge-coloring of $L_1\cup L_2$. Then, either $(L_1\cup L_2,c)$ is a template for $H$ or $H$ is $(s,t)$-structued for $s = \frac{c(L_1)-c(e_1)- c(L_2)+c(e_2)}{2(r-2)}$ and $t = c(e_1)-c(e_2)$.
\end{lemma}
Let us give an overview of the proof of Lemma~\ref{template_example}.
We shall take a blowup of $(L_1\cup L_2,c)$ with $|V_{x_1}|=|V_{y_1}|=|V_{x_2}|=|V_{y_2}|$ and consider a (carefully chosen) perfect $H$-factor $F$ of $B$ such that each copy of $H$ in $F_1$ is either on the clusters of $L_1$ or the clusters of $L_2$. Since we have chosen the clusters of $x_1,y_1,x_2,y_2$ to have the same size, we can then find a second $H$-factor $F_2$ by swapping the vertices in $V_{x_1} \cup V_{y_1}$ with the vertices in $V_{x_2} \cup V_{y_2}$ in each of the copies of $H$ in $F_1$.
We will show that the only case in which the discrepancy does not change under this operation is that $H$ is $(s,t)$-structured (with $s,t$ as in the statement of the lemma). We now proceed with the detailed proof. 
\begin{proof}[Proof of Lemma~\ref{template_example}]
For convenience, set 
$$q := c(L_1) - c(e_1) - c(L_2) + c(e_2) = 
\sum_{i=3}^r[c(x_1v_i) + c(y_1v_i) - c(x_2v_i) - c(y_2v_i)].
$$
Let $A_1,\dots,A_r$ be an $r$-coloring of $H$. To ease the notation, put 
$f_H(A_1,A_2) := e_H(A_1 \cup A_2, V(H) \setminus (A_1 \cup A_2))$ and 
$$g(A_1,\dots,A_r) := 
\frac{\frac{q}{2(r-2)} \cdot f_H(A_1,A_2) + (c(e_1) - c(e_2)) \cdot e_H(A_1, A_2)}{|A_1| + |A_2|} \; .
$$
Our goal is to show that if $(L_1 \cup L_2,c)$ is not a template for $H$, then $g(A_1,\dots,A_r)$ is the same for every $r$-coloring $A_1,\dots,A_r$ of $H$; we then take this common value to be $\rho$ (see Definition \ref{structured}). 
So let $A_1,A_2,\dots, A_r$ and $B_1,B_2,\dots, B_r$ be two arbitrary $r$-vertex-colorings of $H$. We will show that $g(A_1,\dots,A_r) = g(B_1,\dots,B_r)$.
Write $V := L_1 \cap L_2 = \{v_3,\dots,v_r\}$. 
Consider the following blowup $B$ of $(L_1\cup L_2,c)$:
\begin{itemize}
    \item $|V_{x_1}|,|V_{y_1}|,|V_{x_2}|,|V_{y_2}| = (r-2)!(|A_1|+|A_2|)(|B_1|+|B_2|)$,
    \item for $3\leq i\leq r$, $|V_{v_i}| = 2(r-3)!\sum_{3\leq i\leq r}[ (|B_1|+|B_2|)|A_i|+(|A_1|+|A_2|)|B_i| ]$. 
\end{itemize}
To calculate $c(F_1)-c(F_2)$, 
first observe that for every pair $3\leq i<j\leq r$, 
$$
e_{F_1}(V_{v_i},V_{v_j}) = e_{F_2}(V_{v_i},V_{v_j}) = 
4(r-4)! \cdot \left( (|B_1|+|B_2|)\sum_{3\leq i<j\leq r}e_H(A_i,A_j)+(|A_1|+|A_2|)\sum_{3\leq i<j\leq r}e_H(B_i,B_j)\right).
$$
Therefore, $e_{F_1}(V_{v_i},V_{v_j})-e_{F_2}(V_{v_i},V_{v_j})=0$. 
Next, we claim that for each $3 \leq i \leq r$, it holds that 
\begin{equation}\label{eq:structured eq1}
e_{F_1}(V_{x_1},V_{v_i}) = e_{F_1}(V_{y_1},V_{v_i}) = 
e_{F_2}(V_{x_2},V_{v_i}) = e_{F_2}(V_{y_2},V_{v_i}) = 
(r-3)!\cdot (|B_1|+|B_2|) \cdot f_H(A_1,A_2)
\end{equation}
and
\begin{equation}\label{eq:structured eq2}
e_{F_1}(V_{x_2},V_{v_i}) = e_{F_1}(V_{y_2},V_{v_i}) = 
e_{F_2}(V_{x_1},V_{v_i}) = e_{F_2}(V_{y_1},V_{v_i}) = 
(r-3)! \cdot (|A_1|+|A_2|) \cdot f_H(B_1,B_2).
\end{equation}
Let us prove this for $e_{F_1}(V_{x_1},V_{v_i})$; all other cases are similar. For every $3 \leq j \leq r$, there are $(r-3)!$ permutations that embed $A_j$ into $V_{v_i}$ and $A_1$ (resp. $A_2$) into $V_{x_1}$, and each such permutation contributes $e_H(A_1,A_j)$ (resp. $e_H(A_2,A_j)$) to 
$e_{F_1}(V_{x_1},V_{v_i})$. Also, each such permutation gives rise to $(|B_1| + |B_2|)$ copies of $H$ in $F_1$. Summing over all $3 \leq j \leq r$ and all permutations, we get 
$$
e_{F_1}(V_{x_1},V_{v_i}) = (r-3)! \cdot (|B_1| + |B_2|) \cdot \sum_{k = 1}^2\sum_{j = 3}^r e_H(A_k,A_j) = (r-3)!\cdot (|B_1|+|B_2|) \cdot f_H(A_1,A_2),
$$
as required. 

Lastly, from the definition of $F_1,F_2$ it follows that
% let us consider $e_{F_1}(V_{x_1},V_{y_1})-e_{F_2}(V_{x_1},V_{y_1})$. Summing over all copies of $H$, we get
\begin{equation}\label{eq:structured eq3}
    \begin{split}
        e_{F_1}(V_{x_1},V_{y_1})-e_{F_2}(V_{x_1},V_{y_1}) &= 
- (e_{F_1}(V_{x_2},V_{y_2})-e_{F_2}(V_{x_2},V_{y_2})) \\ &= 
2(r-2)! \cdot [ (|B_1|+|B_2|) \cdot e_H(A_1,A_2)-(|A_1|+|A_2|) \cdot e_H(B_1,B_2) ].
    \end{split}
\end{equation}
We now combine all of the above to calculate $c(F_1) - c(F_2)$. First, we can write
\begin{align}\label{eq:structured eq4}
    & \; \; \; \; \; \; \; \; c(F_1)-c(F_2) \notag =\\&
\sum_{i=3}^r \sum_{z \in \{x_1,y_1,x_2,y_2\}}   c(zv_i) \cdot (e_{F_1}(V_{z},V_{v_i})-e_{F_2}(V_{z},V_{v_i})) +\sum_{i=1}^2 c(x_iy_i) \cdot (e_{F_1}(V_{x_i},V_{y_i})-e_{F_2}(V_{x_i},V_{y_i}))
\end{align}
By \eqref{eq:structured eq3}, the second term in \eqref{eq:structured eq4} equals 
\begin{equation}\label{eq:structured eq5}
    2(r-2)! \cdot (c(e_1)-c(e_2)) \cdot 
[|B_1|+|B_2|) \cdot e_H(A_1,A_2)-(|A_1|+|A_2|) \cdot e_H(B_1,B_2) ].
\end{equation}
By \eqref{eq:structured eq1} and \eqref{eq:structured eq2}, the first term in \eqref{eq:structured eq4} equals 
\begin{align}\label{eq:structured eq6}
    \notag &\sum_{i=3}^r [c(x_1v_i) + c(y_1v_i) - c(x_2v_i) - c(y_2v_i) ] \cdot (r-3)! \cdot [(|B_1| + |B_2|) \cdot f_H(A_1,A_2) - (|A_1| + |A_2|) \cdot f_H(B_1,B_2)] \\ &=
    (r-3)! \cdot q \cdot 
     [(|B_1| + |B_2|) \cdot f_H(A_1,A_2) - (|A_1| + |A_2|) \cdot f_H(B_1,B_2)].
\end{align}
If $c(F_1) - c(F_2) \neq 0$ then $(L_1 \cup L_2,c)$ is a template for $H$ and we are done. so suppose that $c(F_1) - c(F_2) = 0$. Then, by plugging \eqref{eq:structured eq5} and \eqref{eq:structured eq6} into \eqref{eq:structured eq4}, dividing by $(r-3)!$ and rearranging, we get 
$g(A_1,\dots,A_r) = g(B_1,\dots,B_r)$, as required. This completes the proof. 
\end{proof}

We end this subsection with several properties of $(s,t)$-structured graphs. The following simple lemma expresses $e(H)$ in terms of $s,t,\rho,|H|$.
\begin{lemma}\label{lem:(s,t)-structured edge count}
    If $H$ is $(s,t)$-structured with parameter $\rho$, then 
    $$
    e(H) = \rho\frac{r-1}{(2r-4)s+t}|H|.
    $$
\end{lemma}
\begin{proof}
    Fix an arbitrary $r$-coloring $A_1,\dots,A_r$ of $H$. By summing \eqref{eq:structured} over all pairs $1 \leq i < j \leq r$, we get
    \begin{align*}
    ((2r-4)s + t) \cdot e(H) &= 
    \sum_{1\leq i<j\leq r}\Big[s\cdot e_H(A_i\cup A_j, V(H)\backslash(A_i\cup A_j)) + t \cdot e_H(A_i,A_j) \Big] \\ &= \rho\sum_{1\leq i<j\leq r}(|A_i| + |A_j|) = \rho(r-1)\sum_{i=1}^r|A_i| = \rho(r-1)|H|.
    \end{align*}
\end{proof}
\noindent
In what follows, we will need the following trivial claim.
\begin{claim}\label{claim:all equal}
    Let $r \geq 3$ and let $a_1,\dots,a_r \in \mathbb{R}$. If there is $c$ such that $a_i + a_j = c$ for all $1 \leq i < j \leq r$, then $a_1 = \dots = a_r$.
\end{claim}
\noindent

In the next lemma, we show that if $H$ is $(s,t)$-structured for two different choices of $(s,t)$, in the sense that the ratio $s/t$ is different, then every $r$-vertex-coloring $A_1,\dots,A_r$ of $H$ is balanced, i.e. satisfies $|A_1| = |A_2|= \dots = |A_r|$. By normalizing, we may and will assume that one of the $t$'s equals $1$. 
\begin{lemma}\label{no_two}
Let $r\geq 3$ and $s,s',t\in \mathbb{R}$ with $s'-t\cdot s\neq 0$.
If $H$ is $(s,1)$- and $(s',t)$-structured, then every $r$-coloring of $H$ is balanced.
\end{lemma}
\begin{proof}
Fix any $r$-coloring of $H$ with parts $A_1,\dots,A_r$. 
By definition (see Definition \ref{structured}), there exist $\rho,\rho'$ such that
$$
\rho(|A_1|+|A_2|) = s\cdot e_H(A_1\cup A_2, V(H)\backslash(A_1\cup A_2))+e_H(A_1,A_2),
$$
and 
$$
\rho'(|A_1|+|A_2|) = s'\cdot e_H(A_1\cup A_2, V(H)\backslash(A_1\cup A_2))+t\cdot e_H(A_1,A_2).
$$
% Let $A_1,A_2,\dots, A_r$ be the parts of an arbitrary $r$-vertex-coloring of $H$.
Combining these two equations, we get
\begin{align*}
    (s'\cdot\rho-s\cdot\rho')(|A_1|+|A_2|) &= (s'-t \cdot s)\cdot e_H(A_1,A_2),
\end{align*}
and 
$$
(\rho'-t\cdot\rho)(|A_1|+|A_2|) = (s'-t\cdot s)\cdot e_H(A_1\cup A_2, V(H)\backslash(A_1\cup A_2)).
$$
Note that $s'-t s\neq 0$ by assumption.
Setting $c = (s'\rho - s\rho')/(s'-ts)$ and $c' = (\rho'-t\rho)/(s'-ts)$, we have
\begin{align}
    c(|A_1|+|A_2|) &= e_H(A_1,A_2) \notag,\\
    c'(|A_1|+|A_2|) &= e_H(A_1\cup A_2, V(H)\backslash(A_1\cup A_2))\label{eq: c'|A_1|+|A_2|}.
\end{align}
Note that $c\neq 0$ because $e_H(A_1,A_2)\neq 0$, as $\chi(H) = r$.
By permuting the parts $A_1,\dots,A_r$, we obtain the analogous equations for every pair of parts $A_i,A_j$. In particular, for all $1 \leq i < j \leq r$,
\begin{align*}
    c(|A_i|+|A_j|) &= e_H(A_i,A_j).
\end{align*}
We now get
\begin{align*}
e_H(A_1\cup A_2, V(H)\backslash(A_1\cup A_2)) &= \sum_{3\leq i\leq r}\left( e_H(A_1,A_i)+e_H(A_2,A_i) \right) = 
c\sum_{3\leq i\leq r}(|A_1| + |A_2| + 2|A_i|)
\\
&=c(r-4)(|A_1|+|A_2|)+2c|H|.
\end{align*}
Combining this with \eqref{eq: c'|A_1|+|A_2|}, we get
$$
(c'-c(r-4))(|A_1|+|A_2|) = 2c|H|\neq 0.
$$
Applying this argument for $A_i,A_j$ in place of $A_1,A_2$, we see that $|A_i| + |A_j| = \frac{2c|H|}{c'-c(r-4)}$ for every $1 \leq i < j \leq r$. Using that $r\geq 3$, we get $|A_1| = \dots = |A_r|$ by Claim~\ref{claim:all equal}.
\end{proof}

Next, we show that if $H$ is $(s,t)$-structured with parameter $\rho$, then for every $r$-coloring $A_1,\dots,A_r$ of $H$, one can express $e_H(A_1,A_2)$ as a function of $|A_1|+|A_2|$. For $s = 0$ this is trivial (recall Definition \ref{structured}), so we assume $s \neq 0$. As it turns out, the cases $s/t \neq\frac{1}{2}$ and $s/t = \frac{1}{2}$ need to be handled separately, and in the latter case we need to additionally assume that $H$ satisfies the $r$-wise $C_4$-condition. To avoid repetitions, we handle both cases together. For convenience, we assume that $s=1$.
\begin{lemma}\label{get_c}
Let $t,\rho \in \mathbb{R}$ and suppose that $H$ is $(1,t)$-structured with parameter $\rho$.
Assume that $t \neq 2$, or $t=2$ and $H$ satisfies the $r$-wise $C_4$-condition. 
Then for every $r$-coloring $A_1,\dots,A_r$ of $H$ it holds that
$$
(r-4+t)(2r-4+t) \cdot e_H(A_1,A_2) = \rho(2r-4+t) \cdot (|A_1|+|A_2|)-2\rho|H|.
$$
\end{lemma}
\begin{proof}
Let $A_1,A_2,\dots A_r$ be the parts of a $r$-coloring of $H$. 
By definition, for all $1 \leq i < j \leq r$,
\begin{align}\label{eq: base property}
\rho(|A_i|+|A_j|) = 
e_H(A_i\cup A_j, V(H)\backslash(A_i\cup A_j))+t\cdot e_H(A_i,A_j).
\end{align}
Summing over all pairs $1 \leq i < j \leq r$, we get
\begin{align}\label{eq_H_edges}
(r-1)\rho|H| =\sum_{1\leq i<j\leq r} \rho(|A_i|+|A_j|)\notag &= 
\sum_{1\leq i<j\leq r} e_H(A_i \cup A_j,V(H)\backslash (A_i\cup A_j))+t\cdot e_H(A_i,A_j)
\\ &=(2r-4+t)e(H).
\end{align}
Let $A = \bigcup_{3\leq i\leq r}A_i$. 
Summing \eqref{eq: base property} over $3 \leq i < j \leq r$, we get
\begin{align}\label{eq_middlestep}
\rho(r-3)(|H|-|A_1|-|A_2|) &= \sum_{3 \leq i < j \leq r}{\rho(|A_i| + |A_j|)} \notag \\ &=
\sum_{3\leq i<j\leq r}e_H(A_i \cup A_j,V(H)\backslash (A_i\cup A_j))+t\cdot e_H(A_i,A_j)\notag \\
&= 
(2r-8+t) \cdot e_H(A)+(r-3) \cdot e_H(A_1\cup A_2, V(H)\backslash(A_1\cup A_2)).
\end{align}
Now multiply \eqref{eq: base property} (for $i = 1,j=2$) by $r-5+t$ and add to \eqref{eq_middlestep}, obtaining:
\begin{align}\label{eq:one-to-last-step}
    & \; \; \rho(r-3)|H| + \rho(t-2)(|A_1| + |A_2|) \notag
    \\ &=(2r-8+t) \cdot e_H(A)+(2r-8+t) \cdot e_H(A_1\cup A_2, V(H)\backslash(A_1\cup A_2)) + t(r-5+t) \cdot e_H(A_1,A_2) \notag \\&=
    (2r-8+t) \cdot e(H) + [t(r-5+t) - (2r-8+t)] \cdot e_H(A_1,A_2) \notag \\&=
    (2r-8+t) \cdot e(H) + (t-2)(r-4+t) \cdot e_H(A_1,A_2),
\end{align}
where the second equality uses $e(H) = e_H(A_1,A_2) + e_H(A) + e_H(A_1\cup A_2, V(H)\backslash(A_1\cup A_2))$. 
Next, we cancel the term $e(H)$ in \eqref{eq:one-to-last-step}. To this end, multiply \eqref{eq_H_edges} by $2r-8+t$ and subtract this from \eqref{eq:one-to-last-step} multiplied by $2r-4+t$, to get:
\begin{align*}
    & \; \; \rho((r-3)(2r-4+t) - (r-1)(2r-8+t)) \cdot |H| + \rho(t-2)(2r-4+t)(|A_1| + |A_2|) \\ &= (t-2)(r-4+t)(2r-4+t) \cdot e_H(A_1,A_2).
\end{align*}
Note that $(r-3)(2r-4+t) - (r-1)(2r-8+t) = -2(t-2)$. If $t \neq 2$, then dividing through by $t-2 \neq 0$ completes the proof. Suppose from now on that $t=2$. We then need another relation coming from the $C_4$-condition. For  every pair $3 \leq i \neq j \leq r$, 
$e_H(A_1,A_2)+e_H(A_i,A_j)-e_H(A_1,A_i)-e_H(A_2,A_j) = 0$. 
Summing this over all (ordered) pairs $i,j$, we get
\begin{align*}
0&= \sum_{3\leq i \neq j\leq r}e_H(A_1,A_2)+e_H(A_i,A_j)-e_H(A_1,A_i)-e_H(A_2,A_j)\\
&=(r-2)(r-3) \cdot e_H(A_1,A_2) + 2e_H(A) - (r-3) \cdot e_H(A_1\cup A_2, V(H)\backslash(A_1\cup A_2)).
\end{align*}
Adding the above equation to (\ref{eq_middlestep}), we get 
\begin{equation}\label{eq:structured t=2}
\rho(r-3)(|H|-|A_1|-|A_2|) = 
(r-2)(r-3) \cdot e_H(A_1,A_2)+ (2r-4) \cdot e_H(A).
\end{equation}
We now continue as before: multiply \eqref{eq: base property} by $2r-4$ and add this to \eqref{eq:structured t=2} to get
\begin{align}\label{eq:one-to-last-step t=2}
    & \; \; 
    \rho(r-3)|H| + \rho (r-1)(|A_1| + |A_2|) =
    \notag \\ &= (2r-4) \cdot e_H(A) + (2r-4) \cdot e_H(A_1\cup A_2, V(H)\backslash(A_1\cup A_2)) + ((r-2)(r-3) + 2(2r-4)) \cdot e_H(A_1,A_2) \notag
    \\ &= (2r-4) \cdot e(H) + (r-1)(r-2) \cdot e_H(A_1,A_2).
\end{align}
In the second equality we used $e(H) = e_H(A_1,A_2) + e_H(A) + e_H(A_1\cup A_2, V(H)\backslash(A_1\cup A_2))$. Finally, multiply \eqref{eq_H_edges} by $2r-4$ and subtract this from \eqref{eq:one-to-last-step t=2} multiplied by $2r-2$, to get
$$
-2\rho(r-1)|H| + \rho(r-1)(2r-2) \cdot (|A_1| + |A_2|) = (2r-2)(r-1)(r-2) \cdot e_H(A_1,A_2).
$$
Dividing through by $r-1$ completes the proof. 
\end{proof}

\section{Lower bounds}\label{sec:lower_bounds}
In this section, we describe some constructions that are used to prove the lower bounds in Theorems \ref{bipartite}, \ref{tripartite1} and \ref{rpartite}. We start with an observation about regular graphs.

\begin{lemma}\label{H_regular}
If $H$ is regular then $\delta^*(H)\geq3/4$.
\end{lemma}
\begin{proof}
Let $c$ be a $2$-edge-coloring of $K_4$ such that $K_4^+$ is isomorphic to $K_{1,3}$. Let $v_0$ be the vertex whose incident edges all have color $+1$. Let $m\in\mathbb{N}$ be divisible by $4$ and $|H|$. Let $B$ be an $m/4$-blowup of $(K_4,c)$. If $B$ has no perfect $H$-factor then $\delta^*(H)\geq \delta(B)/|B|=3/4$.
Let us assume that this is not the case and let $F$ be a perfect $H$-factor of $B$. Let $d\in\mathbb{N}$ be such that $H$ is $d$-regular and note that every $H$-factor is also $d$-regular. Therefore, $F$ has $\frac{d}{2}m$ edges. Also, the number of edges of color $1$ is exactly $d|V_{v_0}| = dm/4$.
Thus, exactly half of the edges have color $1$, meaning that $F$ has zero discrepancy. Hence, $\delta^*(H)\geq \frac{\delta(B)}{|B|}=3/4$.
\end{proof}

\begin{lemma}\label{claim:bipartite_rho}
If there exists $\rho$ such that for every connected component $U$ of $H$ it holds that 
$$
e_H(U)=\rho|U|,
$$
then $\delta^*(H) \geq 1/2$.
\end{lemma}
\begin{proof}
If $\chi^*(H) \geq 2$, then this holds trivially, because 
$\delta^*(H) \geq 1 - 1/\chi^*(H)$, as $1-1/\chi^*(H)$ is the threshold for the existence of a perfect $H$-factor by Theorem~\ref{existence}.
So let us assume that $\chi^*(H)<2$.
Let $m\in\mathbb{N}$ be divisible by $2$ and $|H|$, and let $J$ be an $m$-vertex graph which is the disjoint union of two-cliques $A,B$ of size $m/2$ each (there are no edges between $A$ and $B$). Note that $\delta(J) = m/2-1$. Let $c$ be a $2$-edge-coloring of $J$ with $c(e)=1$ for $e\in J[A]$ and $c(e) = -1$ for $e\in J[B]$. Let $F$ be a perfect $H$-factor of $J$; if no such $F$ exists then $\delta^*(H)\geq 1/2$ immediately holds. 
For each copy $H_0 \in F$ of $H$ and for each connected component $U$ of $H_0$, we have by assumption
$$
e_{F}(U)=\rho|U|.
$$
Also, $U \subseteq A$ or $U \subseteq B$, because there are no edges between $A$ and $B$. 
So $c(F[U]) = \rho|U|$ if $U \subseteq A$ and $c(F[U]) = -\rho|U|$ if $U \subseteq B$. Summing over all $H_0 \in F$ and all components $U$ of $H$, we get
$c(F)=\rho(|A|-|B|)=0$. As this holds for any perfect $H$-factor of $J$, we get $\delta^*(H)\geq 1/2$.
\end{proof}
\noindent
Recall the definition of butterfly graphs in the paragraph above Theorem \ref{tripartite1}. In the next Lemma we use the symmetry of butterflies with respect to the coloring to show that if a certain butterfly is not a template for a given graph $H$, then all $H$-factors of an appropriate blowup of this butterfly have discrepancy $0$, giving a lower bound on $\delta^*(H)$. The construction is depicted on Figure~\ref{fig:butterfly-lower-bound}.

\begin{figure}
    \centering
    \includegraphics{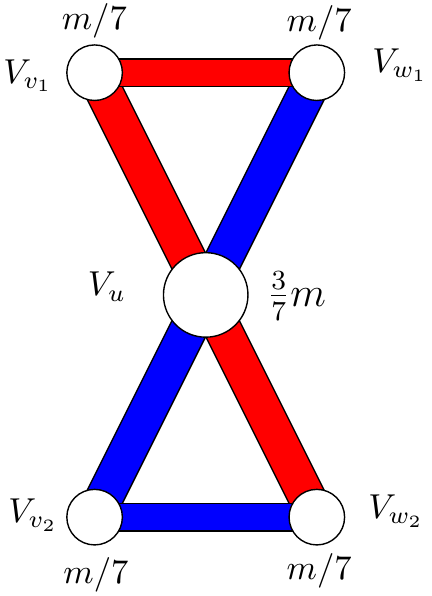}
    \caption{A blowup of a butterfly showing that $\delta^*(H) \ge 4/7$ if the butterfly is not a template for $H$.} 
    \label{fig:butterfly-lower-bound}
\end{figure}
% it is an extremal example for graphs for which it is not a template.
\begin{lemma}\label{butterfly}
If there exists a butterfly which is not a template for $H$, then $\delta^*(H)\geq 4/7$.
\end{lemma}
\begin{proof}
Let $(L,c)$ be a butterfly with triangles $\{u,v_1,w_1\}$ and $\{u,v_2,w_2\}$, and suppose that $(L,c)$ is not a template for $H$. By definition, we have 
\begin{equation}\label{eq:butterfly}
    c(uv_1) = -c(uv_2), \; c(uw_1) = -c(uw_2) \; c(v_1w_1) = -c(v_2w_2).
\end{equation}
% and $u\in V(L)$ the vertex shared by both triangles in $L$.
Let $B$ be a blowup of $(L,c)$ of size $m$, divisible by $7$ and $|H|$, with 
\begin{itemize}
    \item $|V_u|=3m/7$ and 
    \item$|V_v|=m/7$ for $v\in \{v_1,w_1,v_2,w_2\}$.
\end{itemize} 
We claim that every perfect $H$-factor of $B$ satisfies $c(F) = 0$. Indeed, consider the automorphism of $L$ which swaps between $v_1,v_2$ and between $w_1,w_2$, and let $F'$ be the image of $F$ under this automorphism. Then by \eqref{eq:butterfly}, we have $c(F') = -c(F)$. On the other hand, as $(L,c)$ is not a template for $H$, we must have $c(F) = c(F')$. It follows that $c(F) = 0$. This proves that $\delta^*(H) \geq \delta(B)/|B| = 4/7$.
\end{proof}

Next, we prove the lower bound on $\delta^*(H)$ in the first two cases of Theorem~\ref{rpartite}. The constructions use Lemma~\ref{regularC4} and are depicted in Figure~\ref{fig:lower-bounds}.

\begin{figure}[!ht]
    \centering
    \includegraphics[scale=0.6]{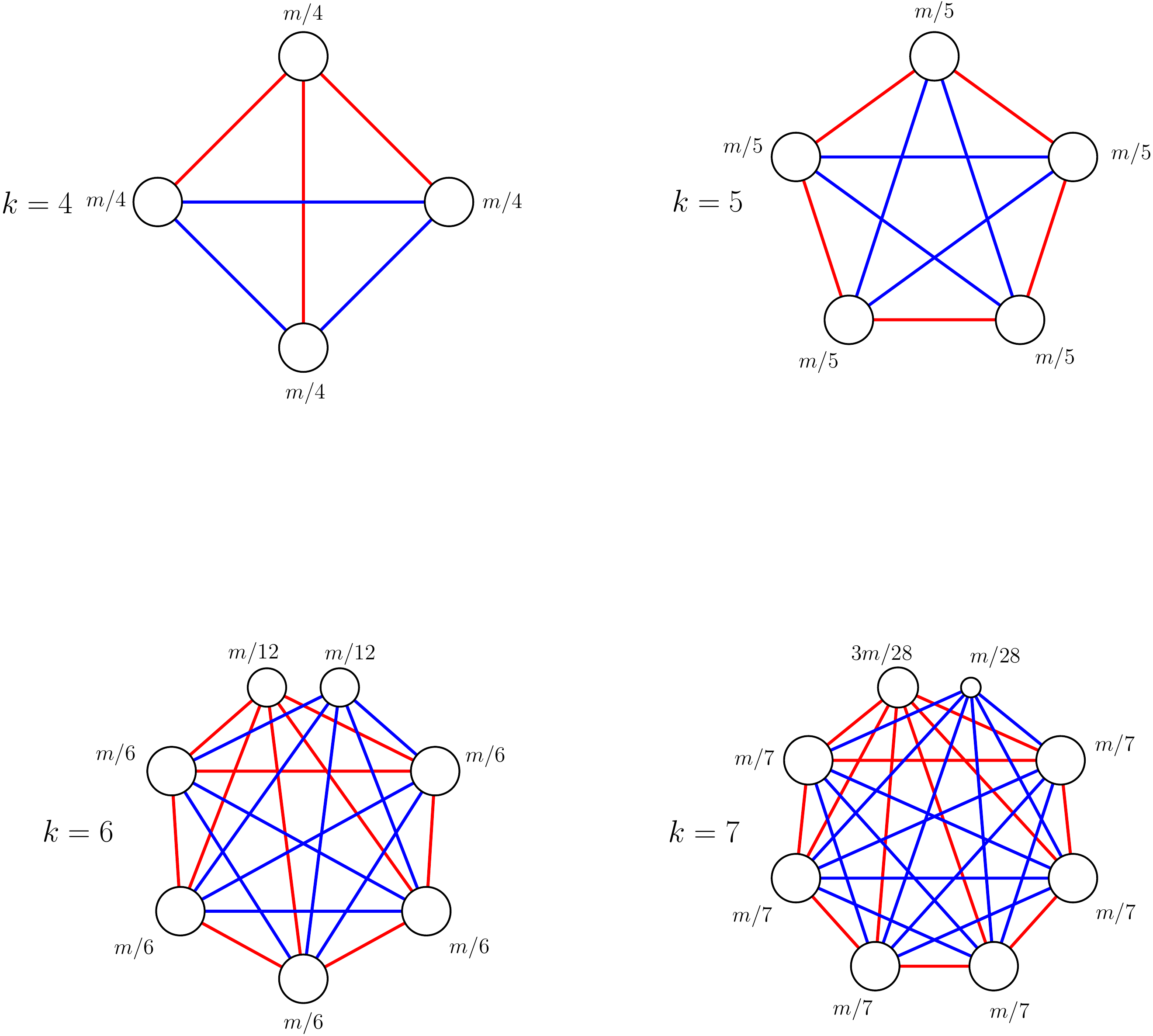}
    \caption{The constructions given in Lemmas~\ref{construction}~and~\ref{construction2} for $4 \le k \le 7.$}
    \label{fig:lower-bounds}
\end{figure}

\begin{lemma}\label{construction}
If $H$ fulfills the $k$-wise $C_4$-condition for some $k\equiv_4 1$, then $\delta^*(H)\geq \frac{k-1}{k}$. Additionally, if $k=\chi(H)$ then $\delta_0(H) = \frac{k-1}{k}$.
\end{lemma}
\begin{proof}
Let $c$ be a $2$-edge-coloring of $K_k$ such that $K_k^+$ is $(k-1)/2$-regular; such a coloring exists because $k\equiv_4 1$. Since $H$ satisfies the $k$-wise $C_4$-condition, so does every $H$-factor. Hence, 
Lemma~\ref{regularC4} with $d = (k-1)/2$ implies that for every blowup $B$ of $(K_k,c)$ and for every perfect $H$-factor $J$ of $B$, it holds that 
$$
c(J) = \frac{2\cdot(k-1)/2-k+1}{k-1} \cdot e(J)=0.
$$
This implies that $c \in \mathcal{K}(H)$. Also, taking $B$ to be the
$m/k$-blowup of $(K_k,c)$, we get that $\delta^*(H)\geq \delta(B)/|B|=\frac{k-1}{k}$. Finally, if $k = \chi(H)$ and we take $m$ to be divisible by $(k-1)!|H|$, then $B$ has a perfect $H$-factor by Lemma~\ref{lem:balanced blowup}, and we get that $\delta_0(H) = \frac{k-1}{k}$ by the definition of $\delta_0(H)$.
\end{proof}
\begin{lemma}\label{construction2}
If $H$ is regular and fulfills the $k$-wise $C_4$-condition for some $k\not\equiv_4 1$, then
$\delta^*(H)\geq \frac{k-1}{k}$.
\end{lemma}
\begin{proof}
We give three different constructions depending on the residue of $k$ modulo $4$, but the general idea is the same for all three.
Let $U\subseteq V(K_k)$ with $|U| = k-1$ and $\{v\} = V(K_k)\backslash U$. Let $c$ be any 2-edge-coloring of $K_k$ such that $K_k[U]^+$ is $\ell$-regular, where $\ell$ is an even integer to be determined later; such a coloring exists because $\ell$ is even. Let $B$ be an $m/k$-blowup of $(K_k,c)$ for some $m$ with $m$ divisible by $4k$. Let $d\in \mathbb{N}$ be such that $H$ is $d$-regular and let $F$ be an arbitrary perfect $H$-factor of $B$. Note that $F$ is also $d$-regular and satisfies the $k$-wise $C_4$-condition. So $e(F) = dm/2$ and $V_v$ is incident to $d|V_v|=\frac{dm}{k}$ edges in $F$. Hence, $e_F(V_U) = \frac{k-2}{2k}dm$. Additionally, $F[V_U]$ satisfies the $(k-1)$-wise $C_4$-condition. To see this, observe that every $(k-1)$-coloring of $F[V_U]$ can be extended to 
 a $k$-coloring of $F$ by adding $V(F) \cap V_v$ as an additional color-class. Thus, by applying Lemma~\ref{regularC4} with $J = F[V_U]$ (and with $k-1$ in place of $k$), we get
$$
c(F[V_U]) = \frac{2\ell-k+2}{k-2}\cdot e_F(V_U) = \frac{2\ell-k+2}{k-2}\cdot\frac{k-2}{2k}dm = \frac{2\ell-k+2}{2k}dm.
$$
We now define a $2$-edge-coloring $c'$ of $B$ as follows. First, if $e$ is contained in $V_U$, then set $c'(e) = c(e)$. Second, to color the edges incident to $V_v$, split $V_v$ into two sets $V_+$ and $V_-$ (whose sizes will be determined later), and set $c'(e) = 1$ for all edges incident to $V_+$ and $c'(e) = -1$ for all edges incident to $V_-$. 
Then, 
\begin{equation}\label{eq3}
c'(F) = c(F[V_U])+e_F(V_+,V_U)-e_F(V_-,V_U) = \frac{2\ell-k+2}{2k}dm + d(|V_+|-|V_-|),
\end{equation}
where the last equality uses that $F$ is $d$-regular. 
We now consider the different cases of $k$ modulo $4$ and choose $\ell$, $|V_+|$ and $|V_-|$ so that $c'(F)=0$.
\begin{itemize}
    \item If $k\equiv_4 0$, take $\ell = k/2$, $|V_+| = 0$ and $|V_-| = m/k$.
    \item If $k\equiv_4 2$, take $\ell = (k-2)/2$ and $|V_+|=|V_-|=\frac{m}{2k}$.
    \item If $k\equiv_4 3$, take $\ell = (k-3)/2$, $|V_+| = \frac{3m}{4k}$ and $|V_-| = \frac{m}{4k}$.
\end{itemize}
It is easy to check that $\ell$ is even and we have $c'(F) = 0$ by (\ref{eq3}) in all three cases.  
\end{proof}
The final two lemmas of this section provides us with a lower bound on $\max\{\delta_0(H),1-1/\chi^*(H)\}$ (and therefore also on $\delta^*(H)$) for $(s,t)$-structured graphs (recall Definition~\ref{structured}). 

\begin{lemma}\label{lower_delta_r=3}
Suppose that $r=3$ and $H$ is $(1,2)$-structured. Then
$
\max\{\delta_0(H),1-1/\chi^*(H)\} \geq 5/8.
$
\end{lemma}
\begin{proof}
By the definition of structuredness (see Definition~\ref{structured}), there is $\rho \in \mathbb{R}$ such that 
\begin{equation}\label{eq:rho 5/8 general}
\rho(|A_i| + |A_j|) = 
e_H(A_i\cup A_j, V(H)\backslash(A_i\cup A_j))+2e_H(A_i,A_j)= 
e(H) + e_H(A_i,A_j).
\end{equation}
for every $3$-coloring $A_1,A_2,A_3$ of $H$ and all $1 \leq i < j \leq 3$. 
By summing the above equation over all $1 \leq i < j \leq 3$, we get
$$
2\rho|H| = \sum_{1 \leq i < j \leq r}\rho(|A_i| + |A_j|) = 
3e(H) + \sum_{1 \leq i < j \leq r}e_H(A_i,A_j) = 4e(H).
$$
So $\rho |H| = 2e(H)$.
Subtracting from this the equation $\rho(|A_2|+|A_3|) = e(H) + e_H(A_2,A_3) $, we get
\begin{equation}\label{eq:rho 5/8}
\rho|A_1| = e_H(A_1,A_2) + e_H(A_1,A_3) .
\end{equation}

Now consider a triangle $T$ with vertices $u_1,u_2,u_3$ and let $c$ be the $2$-edge-coloring of $T$ where $c(u_1u_2) = c(u_1u_3) = 1$, $c(u_2u_3) = -1$. We first show that $c\in\mathcal{K}(H)$.
So let $B$ be an arbitrary blowup of $(T,c)$, and let $U_1,U_2,U_3$ be the clusters of $B$ (where $U_i$ corresponds to $u_i$). We need to show that all perfect $H$-factors of $B$ have the same discrepancy. So let $F$ be a perfect $H$-factor of $B$. 
First note that 
$$e(F) = e(H) \cdot |B|/|H| = \rho|B|/2.$$  
For each $H$-copy $H_0 \in F$, consider the $3$-coloring $A_1,A_2,A_3$ of $H_0$ given by $A_i = V(H_0) \cap U_i$. The number of edges of color 1 in $H_0$ is precisely 
$e_{H_0}(A_1,A_2)+e_{H_0}(A_1,A_3)$. By \eqref{eq:rho 5/8}, this number is $\rho|A_1|$. Summing over all $H$-copies $H_0 \in F$, we see that the number of edges of color $1$ in $F$ is precisely $\rho |U_1|.$
Therefore, 
$$
c(F) = 2e(F^+) - e(F) = \rho(2|U_1| - |B|/2),
$$
which is independent of $F$. This shows that indeed $c\in\mathcal{K}(H)$.

Let us now consider the specific case where $|U_1|=2m$ and $|U_2|=|U_3|=3m$, for an arbitrary integer $m$.
Then $|B| = 8m$ and $\delta(B)/|B| = (|U_1| + |U_2|)/|B| = 5/8$.
If $B$ has no perfect $H$-factor for any choice of $m$, then $1-1/\chi^*(H) \geq \delta(B)/|B| = 5/8$. 
And otherwise, fix an $m$ for which $B$ has perfect $H$-factors. Every perfect $H$-factor $F$ of $B$ satisfies
$$c(F) = \rho(2|U_1|-|B|/2) = 0.$$ 
By the definition of $\delta_0(H)$, this implies that $\delta_0(H)\geq \delta(B)/|B|=5/8$, as required.  
\end{proof}

\begin{lemma}\label{lower_delta}
Suppose that $r \geq 6$ and $r \equiv_4 2,3$. Assume that $H$ is $(0,1)$-structured, or that $H$ is $(1,t)$-structured for some $t\in\{-2,-1,0,1,2\}$ and satisfies the $r$-wise $C_4$-condition. 
Then
$$\max\{\delta_0(H),1-1/\chi^*(H)\}\geq \frac{3r-5}{3r-2}.$$
\end{lemma}

The bound of $\frac{3r-5}{3r-2}$ in Lemma~\ref{lower_delta} is particularly interesting because a blowup on $m$ vertices of two copies of $K_r$ sharing $r-2$ vertices has minimum degree at most $\frac{3r-5}{3r-2}m$. The lower bound given by Lemma~\ref{lower_delta} will allow us to assume later on that $\delta(R)/|R| > \frac{3r-5}{3r-2}$ (when proving upper bounds on $\delta^*(H)$). We will then (implicitly) use the fact that $R$ is not a blowup of two $r$-cliques sharing $r-2$ vertices. See Lemma~\ref{high_min_deg}.

\begin{proof}[Proof of Lemma~\ref{lower_delta}]
We assume that $H$ is $(s,t)$-structured, where $(s,t) = (0,1)$ or $s = 1$ and $t \in \{-2,-1,0,1,2\}$. In particular, $s\in \{0,1\}$ and $(s,t) \neq (0,0)$.
First we show that there exist $\tau,C$ such that for every $r$-coloring $A_1,\dots A_r$ of $H$,
\begin{align}\label{eq: tau,C}
e_H(A_i,A_j) = \tau(|A_i|+|A_j|)+\tau\cdot C|H|.
\end{align}
It is enough to show this for $i,j = 1,2$.
Let $\rho$ be such that $H$ is $(s,t)$-structured with parameter $\rho$ (recall Definition~\ref{structured}).
If $(s,t) = (0,1)$ then \eqref{eq: tau,C} is trivially satisfied for $\tau=\rho$ and $C = 0$. Otherwise, we have that $s=1$. Then, by Lemma~\ref{get_c}, we get that
\begin{align}\label{eq: one edge equal to sum minus total}
(r-4+t)(2r-4+t)\cdot e_H(A_1,A_2) = \rho(2r-4+t)\cdot(|A_1|+|A_2|)-2\rho|H|.
\end{align}
Observe that in all possible choices of $(s,t)$, we have $r + t \geq 4$, and equality holds only if $r = 6$, $s = 1$ and $t = -2$. We consider this case first. Then \eqref{eq: one edge equal to sum minus total} becomes
\begin{align}\label{eq: t=-2, r=6}
2\rho|H| = 6\rho (|A_1|+|A_2|).
\end{align}
We claim that $\rho \neq 0$. To see this, we use that $H$ is $(1,-2)$-structured, summing \eqref{eq:structured} over all pairs $1 \leq i < j \leq r$. This gives:
\begin{align*}
(r-1)\rho|H| = \sum_{1\leq i<j\leq r}\rho(|A_i|+|A_j|)&=\sum_{1\leq i<j\leq r}[e_H(A_i\cup A_j, V(H)\setminus (A_i\cup A_j)) - 2e_H(A_i,A_j)]\\
&= (2r-6) \cdot e(H) = 6e(H) >0,
\end{align*}
Now divide both sides of \eqref{eq: t=-2, r=6} by $\rho$ to get 
$|H|/3=|A_1|+|A_2|.$ 
Since this holds for every $r$-coloring of $H$, we get that every $r$-coloring of $H$ must be balanced by Claim~\ref{claim:all equal}. This implies that $\chi^*(H) = r$ by the definition of $\chi^*$. Now, $1 - 1/\chi^*(H) = \frac{r-1}{r} \geq \frac{3r-5}{3r-2}$.

% Note that for none of the considered cases we have $r+t<4$ and thus, let us from now on assume that $r+t>4$. 
Suppose from now on that $r+t>4$.
Then $(r-4+t)(2r-4+t)\neq 0$. Thus, dividing both sides of \eqref{eq: one edge equal to sum minus total} by $(r-4+t)(2r-4+t)$, we get that \eqref{eq: tau,C} is satisfied for $\tau = \frac{\rho}{r-4+t}$ and $C = \frac{-2}{2r-4+t}.$ 
This proves \eqref{eq: tau,C}. 
Note also that either $s = 0$ and $C=0$, or $s = 1$ and $C = \frac{-2}{2r-4+t}.$
Since $t \ge -2$, we get in either case that 
\begin{align}\label{eq: bounds on C}\frac{-1}{r-3}\leq C \leq 0.\end{align}

Next, we use \eqref{eq: tau,C} to complete the proof of the lemma. 
First, we claim that for every $2$-coloring $c$ of $K_r$ it holds that $c \in \mathcal{K}(H)$. Indeed, let $B$ be an arbitrary blowup of $(K_r,c)$, and let $F$ be any perfect $H$-factor of $B$. 
Let $B_1,\dots,B_r$ be the parts of $B$. 
By summing \eqref{eq: tau,C} over all copies of $H$ in $F$, we get that  
\begin{equation*}\label{eq: edges between parts in factor}
    e_F(B_i,B_j) = \tau (|B_i|+|B_j|)+\frac{|F|}{|H|}\cdot \tau C|H|=\tau (|B_i|+|B_j|)+\tau C|F|.
\end{equation*}
for all $1 \leq i \leq j \leq r$. 
This implies that 
\begin{equation}\label{eq:c(F) proof of Lemma 8.7}
c(F) = \sum_{1 \leq i < j \leq r}c(ij) \cdot e_F(B_i,B_j) = 
\tau \sum_{1 \leq i < j \leq r}c(ij) \cdot (|B_i|+|B_j|+C|F|),
\end{equation}
which is independent of $F$, since $|F|=|B|$. This means that all perfect $H$-factors of $B$ have the same discrepancy. 
It now follows, by the definition of $\mathcal{K}(H)$, that $c\in\mathcal{K}(H)$.

Next, we define a certain $2$-edge-coloring $c$ of $K_r$, as follows. Write $V(K_r) = \{v_1,\dots,v_r\}$. We claim that there exists a coloring $c$ such that $c(e) = 1$ for every $e \in E(K_r)$ incident to $v_1$, and $c(K_r) = 1$. Indeed, there are $r-1$ edges incident to $v$, and $r-1 < \binom{r}{2}/2$ for $r \geq 6$. Hence, there is a coloring $c$ such that $c(v_1v_j) = 1$ for all $2 \leq j \leq r$, and $c$ colors exactly $\lceil\binom{r}{2}/2\rceil = (\binom{r}{2} + 1)/2$ edges with color $1$. Here we use that $\binom{r}{2}$ is odd because $r \equiv_4 2,3$. 
So the number of edges of color $-1$ is $(\binom{r}{2} - 1)/2$, and hence $c(K_r) = 1$, as required. We fix such a coloring $c$ from now on. 

Fix $C' \in \mathbb{N}$ such that $C\cdot C'\in \mathbb{Z}$, where $C$ is the constant from \eqref{eq: tau,C}. 
Fix an integer $m$ divisible by $(r-2)(r-1)(r+1)C'|H|$, and let $B$ 
be the blowup of $(K_r,c)$ with $|B| = m$ and with parts $B_1,\dots,B_r$ (where $B_i$ corresponds to $v_i$), such that $|B_1| = x$ and $|B_i| = \frac{m-x}{r-1} \eqqcolon y$ for all $2 \leq i \leq r$, where
$$
x = \frac{(r-3)-(r-1)C}{(r-2)(r+1)}m.
$$
By our choice of $m,$ both $x$ and $y$ are integers. 
Next, we show that 
\begin{align}\label{eq: upper and lower bound for x}
\frac{m}{3r-2} \leq x \le \frac{m}{r}.
\end{align}
As $C \leq 0$, to prove the lower bound in \eqref{eq: upper and lower bound for x} it suffices to check that 
$
(r-2)(r+1) \leq (3r-2)(r-3),
$
which holds for $r \geq 4$.
For the upper bound,
we use that $C \geq \frac{-1}{r-3}$ and therefore $x \leq \frac{(r-3)^2 + (r-1)}{(r-2)(r+1)(r-3)}m$. Hence, it suffices to verify that 
$
r((r-3)^2+(r-1))\leq (r-2)(r+1)(r-3),
$
which holds for $r\geq 6$. 

The bounds in \eqref{eq: upper and lower bound for x} 
imply that 
\begin{equation}\label{eq: upper and lower bound for y} 
    \frac{m}{r} \leq y \leq \frac{3m}{3r-2}.
\end{equation}
In particular, $B_1$ is the smallest part in the blowup $B$, by \eqref{eq: upper and lower bound for x} and \eqref{eq: upper and lower bound for y}. 
Therefore, $\delta(B) = m-y \geq \frac{3r-5}{3r-2}m$. If for all $m$, the blowup $B$ has no perfect $H$-factor, then, using Theorem~\ref{existence}, we see that $1 - 1/\chi^*(H) \ge 
\frac{3r-5}{3r-2},$ as needed. Else, fix $m$ such that $B$ has perfect $H$-factors, and let $F$ be an arbitrary $H$-factor of $B$. 
For convenience, we use the notation $c(v_i,K_r) := \sum_{j \neq i}c(v_iv_j)$. 
By \eqref{eq:c(F) proof of Lemma 8.7}, we have 
$$
c(F) = \tau \sum_{1 \leq i < j \leq r}c(v_iv_j) \cdot (|B_i|+|B_j|+C|F|) = 
\tau \cdot \left( c(K_r) \cdot Cm +\sum_{i = 1}^r c(v_i,K_r) \cdot |B_i| \right),
$$
using that $|F| = |B| = m$.
Recall that $c(K_r) = 1$ and $c(v_1,K_r) = r-1$. Hence,
$\sum_{i=1}^{r}c(v_i,{K_r}) = 2c(K_r) = 2$ and 
$\sum_{i=2}^{r}(v_i,K_r) = 3-r$. It follows that 
\begin{align*}
c(F)
&= \tau\left((r-1)x - (r-3) \cdot \frac{m-x}{r-1} + mC\right) = 0,
\end{align*}
by our choice of $x$. 
As $c \in \mathcal{K}(H)$, this implies $\delta_0(H) \ge \delta(B)/m \geq \frac{3r-5}{3r-2},$ as needed.
\end{proof}
\section{When all $r$-cliques have the same discrepancy}\label{sec: K_r same discrepancy}
In this section we consider the situation when all $r$-cliques in the reduced graph $R$ have the same discrepancy sign (i.e. all have positive discrepancy or all have negative discrepancy). 
By symmetry, we may assume that all have positive discrepancy (else swap the colors). 
As usual, we work under the setup described in Section~\ref{sec_app_regularity}. 
The main result of this section is the following:

\begin{lemma}\label{all k positive}
Assume the setup of Section~\ref{sec_app_regularity}.
If $H$ is non-regular and all copies of $K_r$ in $R$ have positive discrepancy, then there exists a perfect $H$-factor in $G'$ with high discrepancy. 
\end{lemma} 

The proof of Lemma \ref{all k positive} is broken into two cases depending on whether or not $H$ is uniform (recall Definition \ref{def:uniform}). These two cases are handled in the following two subsections. Recall the definition of the $r$-partite graph $H^*$ (see Definition \ref{def:H^*}). Namely, recall that if $r=2$ then $H^* = H$, and otherwise $H^*$ is a complete $r$-partite graph satisfying the properties in Lemma \ref{h_star}. Recall also that $V_0$ is the exceptional class in the regular partition of $G'$, and that for a vertex $u \in V(G') \setminus V_0$, we use $V_u \in V(R)$ to denote the part of the partition containing $u$ (so $V_u$ is a vertex of the reduced graph $R$).

\subsection{Proof of Lemma~\ref{all k positive}: $H$ is non-uniform}
Here we prove Lemma~\ref{all k positive} in the case that $H$ is non-uniform.

By Lemma \ref{lem:H* factor}, $G'$ has a perfect $H^*$-factor $F^*$.
Suppose first that there exists a copy $J\in F^*$ of $H^*$ disjoint from $V_0$ with parts $B_1,B_2,\dots B_r\subseteq V(J)$, and vertices $b_1,b'_1 \in B_1$, $b_2 \in B_2$, such that $f(b_1b_2)\neq f(b_1'b_2)$.
% $b_i\in B_i$ ($1\leq i\leq r$) and $b_1'\in B_1 \setminus \{b_1\}$, so that $f(b_1b_2)\neq f(b_1'b_2)$. 
Fix arbitrary $b_i \in B_i$ for $3 \leq i \leq r$. 
Since $H$ is non-uniform, there exists an $r$-coloring $A_1,A_2,\dots,A_r$ of $H$ such that $e_H(A_1,A_2)\neq e_H(A_1,A_3)$, by Claim~\ref{claim:non-uniform}. This implies that there exists a vertex $a\in A_1$ with $d_H(a,A_2)\neq d_H(a,A_3)$. 
Consider the $r$-cliques $L_1 = \{V_{b_1}^R,V_{b_2}^R,\dots V_{b_r}^R\}$ and $L_2 = \{V_{b_1'}^R,V_{b_2}^R,\dots V_{b_r}^R\}$ in $R$. If $L_1,L_2$ have different discrepancies (with respect to $f_R$), then $(L_1 \cup L_2,f_R)$ is a template for $H$ by Lemma~\ref{sharing_much_template} (as $H$ is non-regular), and then $G'$ has a perfect $H$-factor with high discrepancy by Lemma~\ref{template}, completing the proof. We may therefore assume that $f_R(L_1) = f_R(L_2)$. Then we can apply Lemma~\ref{like blowup2} with $x = V_{b_1}^R, y = V_{b'_1}^R, z = V_{b_2}^R$, to conclude that $(L_1 \cup L_2,f_R)$ is a template for $H$. Now we are again done by Lemma~\ref{template}.

So from now on, let us assume that $J \in F^*$ as above does not exists. This means that for every copy $J\in F^*$ of $H^*$ disjoint from $V_0$, if $B_1,\dots,B_r$ denote the parts of $J$, then all bipartite graphs $(B_i,B_j)$ are monochromatic with respect to $f$. In other words, $J$ is a blowup of $(K_r,c)$ for some $2$-edge-coloring $c$ of $K_r$. 
% Let $B_1,B_2,\dots B_r$ be the parts of $J$ and vertices $b_1\in B_1, b_2\in B_2,\dots, b_r\in B_r$. 
Fix arbitrary $b_1\in B_1, b_2\in B_2,\dots, b_r\in B_r$, and consider the $r$-clique $L = \{V_{b_1}^R,V_{b_2}^R,\dots V_{b_r}^R\}$ in $R$. 
By \eqref{eq:f_R}, $L$ is colored by $f_R$ in the same way as $K_r$ by $c$. Hence, $c(K_r) = f_R(L) > 0$, using our assumption that every $r$-clique in $R$ has positive discrepancy. If $(L,f_R)$ is a template for $H$ then we are done by Lemma~\ref{template} as before, and else we have $c \in \mathcal{K}(H)$ by definition. Now, by Lemma~\ref{delta_0}, we see that every perfect $H$-factor of $J$ has positive discrepancy. 

For each $J \in F^*$, let $F_J$ be a perfect $H$-factor in $J$. 
Then $F := \bigcup_{J \in F^*}F_J$ is a perfect $H$-factor of $G'$. We saw that if $J \cap V_0 = \emptyset$ then $f(F_J) > 0$, and so $f(F_J) \geq 1$.
Using that $|V_0|\leq \varepsilon n\ll \frac{n}{|H^*|e(H^*)}$, we obtain 
$$
f(F)\geq \frac{n}{|H^*|}-|V_0|-|V_0| \cdot e(H^*)\geq \gamma n.
$$
This completes the proof in the case that $H$ is non-uniform. 

\subsection{Proof of Lemma~\ref{all k positive}: $H$ is uniform}
Here we prove Lemma~\ref{all k positive} in the case that $H$ is uniform. 
First, by Lemma \ref{lem:H* factor}, $G'$ has a perfect $H^*$-factor $F^*$. The key part of the proof is the following lemma:
\begin{lemma}\label{SameDisc}
Let $J\in F^*$ be a copy of $H^*$ disjoint from $V_0$. Then every perfect $H$-factor $F_J$ of $J$ satisfies
$$
f(F_J) \geq 1,
$$
or there exists a template for $H$ in $(R,f_R)$ of size $r+1$.
\end{lemma}
\begin{proof}
For $r=2$, the statement follows trivially as each edge in $R$ has positive discrepancy and therefore, $R$ (and hence also $J$) are monochromatic. Therefore, let us assume that $r\geq 3$.
By Lemma~\ref{sharing_much_template}, 
we may assume that there do not exist two copies of $K_r$ in $R$ sharing $r-1$ vertices with different discrepancies, as otherwise there is a template for $H$ (since $H$ is non-regular). 
Let $J\in F^*$ be a copy of $H^*$ disjoint from $V_0$. Let $A_1,A_2,...,A_r$ be the clusters of $J$ and fix arbitrary vertices $a_1\in A_1, a_2\in A_2,...,a_r\in A_r$. Let $F_J$ be an arbitrary $H$-factor in $J$.
\begin{claim}\label{Split}
For each $1\leq i<j\leq r$, it holds that either for all vertices $u\in A_i$ and $v,v'\in A_j$, $f(uv)=f(uv')$, or for all vertices $u,u'\in A_i$ and $v\in A_j$, $f(uv)=f(u'v)$.
\end{claim}
\begin{proof}
Without loss of generality, $i = 1,j = 2$. Observe that if the assertion of the claim does not hold, then there exist $u,u'\in A_1$ and $v,v'\in A_2$ such that $f(uv) \neq f(uv'),f(u'v)$. Without loss of generality, $f(uv) = 1$ and $f(uv') = f(u'v) = -1$. 
Since $J$ is disjoint from $V_0$, we get that 
$L_1 := \{V_u^R,V_v^R,V_{a_3}^R,V_{a_4}^R,\dots,V_{a_r}^R\}$ and 
$L_2 = \{V_u^R,V_{v'}^R,V_{a_3}^R,V_{a_4}^R,\dots,V_{a_r}^R\}$ are $r$-cliques in $R$. 
We have $f_R(L_1) = f_R(L_2)$ because, by assumption, every two $r$-cliques in $R$ sharing $r-1$ vertices have the same discrepancy. It follows that 
$$
f(uv)+\sum_{3\leq i\leq r}f(va_i) = f(uv')+\sum_{3\leq i\leq r}f(v'a_i).
$$
Note that $u'v'\in J$, as $J$ is a complete $r$-partite graph. 
By the same argument with $u'$ in place of $u$, we get 
$$
f(u'v)+\sum_{3\leq i\leq r}f(va_i) = f(u'v')+ \sum_{3\leq i\leq r}f(v'a_i).
$$
By subtracting the second from the first equation, we get
$$
f(uv)-f(u'v) = f(uv')- f(u'v').
$$
But this is a contradiction since $f(uv)-f(u'v)=2$ and $f(uv')- f(u'v')\leq 0$.
\end{proof}
We continue with the proof of Lemma~\ref{SameDisc}.
Let us say that $(A_i,A_j)$ is {\em split} for $A_i$ if there exist vertices $u,u'\in A_i$ and a vertex $v\in A_j$ such that $f(uv)\neq f(u'v)$. Note that if for some $i$ and $j$, $(A_i,A_j)$ is split for $A_i$, then all the vertices in $A_i$ have only monochromatic edges to $A_j$ by the above claim. Therefore, each pair $(A_i,A_j)$ can be split at most for one of the two. 
Recall that by assumption, all $r$-cliques in $R$ have positive discrepancy.
Let $$x=\sum_{1\leq i<j\leq r}f_R(V_{a_i}^RV_{a_j}^R)>0.$$

First, let us assume that $J$ does not have a split pair. It follows that for all $1\leq i<j\leq r$, $G'[A_i,A_j]$ is monochromatic with color $f(a_ia_j)=f_R(V_{a_i}^RV_{a_j}^R)$.
As $H$ is uniform, we get that for every $H$-copy $H_J \in F_J$ 
$$
f(H_J) = \sum_{1\leq i<j\leq r}\frac{e(H)}{\binom{r}{2}}f(a_ia_j)=x\frac{e(H)}{\binom{r}{2}}.
$$
This implies that
$$
f(F_J)= \sum_{H_J \in F_J}f(H_J) = \frac{|H^*|}{|H|}\cdot x\frac{e(H)}{\binom{r}{2}}>0.
$$

Next, let us assume without loss of generality that $A_1A_2$ is split for $A_1$. Then, there exist $u,u'\in A_1$ and $v\in A_2$ such that $f(u,v) = 1$ and $f(u',v) = -1$. As the $r$-cliques $L_1 := 
\{V_u,V_v,V_{a_3},\dots,V_{a_r}\}$ and 
$L_2 := \{V_{u'},V_v,V_{a_3},\dots,V_{a_r}\}$ have the same discrepancy and $f_R(V_uV_v) \neq f_R(V_{u'}V_v)$, we may assume by Lemma~\ref{like blowup2} that for every $r$-vertex-coloring of $H$ with parts $B_1,B_2,\dots, B_r$ it holds for all $b\in B_1$ that
\begin{align}\label{eq same degree}
d_H(b,B_2) = d_H(b,B_3),
\end{align}
as otherwise $(L_1 \cup L_2, f_R)$ is a template for $H$ and we are done. 
For every $1\leq i\leq r$, let $S_i\subseteq\{1,2,\dots,r\}$ be the set of indices $j$ such that $A_iA_j$ is split for $A_i$.
\begin{claim}\label{same_split}
For every $1\leq i\leq r$ and $u,v\in A_i$ it holds that 
$$
\sum_{j\in S_i}f(ua_j) = \sum_{j\in S_i}f(va_j),
$$
or there exists a template for $H$ in $(R,f_R)$ of size $r+1$.
\end{claim}
\begin{proof}
Without loss of generality, suppose that there are $u,v\in A_1$ such that 
$$
\sum_{j\in S_1}f(ua_j) \neq \sum_{j\in S_1}f(va_j).
$$
Consider the $r$-cliques $L_1 = \{V_u,V_{a_2},V_{a_3}...V_{a_r}\}$ and $L_2 = \{V_v,V_{a_2},V_{a_3}...V_{a_r}\}$ in $R$. Observe that
$$
f_R(L_1)-f_R(L_2) = \sum_{j=2}^r( f(ua_j) - f(va_j)) = 
% \sum_{j\in S_1}f(ua_j) - \sum_{j\in S_1}f(va_j)\neq 0.
\sum_{j \in S_1}^r( f(ua_j) - f(va_j)) \neq 0.
$$
Here we used that $f(ua_j) = f(va_j)$ for all $j \notin S_1$.
By Lemma~\ref{sharing_much_template}, $(L_1 \cup L_2,f_R)$ is a template for $H$.
\end{proof}
We now conclude the proof of Lemma~\ref{SameDisc}. 
Fix an $H$-copy $H_J \in F_J$, and let $B_i = A_i \cap V(H_J)$, $i = 1,\dots,r$.
By \eqref{eq same degree}, each vertex $a \in B_i$ has the same number of neighbours in $B_j$ for each $j \in [r] \setminus \{i\}$. So this number is $\frac{d_{H_J}(a)}{r-1}$. Furthermore, if $j \in S_i$ then for each $a \in B_i, a' \in B_j$ we have $f(aa') = f(aa_j)$, as all edges between $a$ and $A_j$ have the same color. 
Hence, for each $1 \leq i \leq r$,
$$
f\left( H_J \cap (B_i\times \bigcup_{j\in S_i}B_j)%\text{\MC{Doesn't this assume that $H$ is complete?}} 
\right) = \sum_{a\in B_i}\frac{d_{H_J}(a)}{r-1}\sum_{j\in S_i}f(aa_j)= 
\sum_{a\in B_i}\frac{d_{H_J}(a)}{r-1}\sum_{j\in S_i}f(a_ia_j) =
\frac{e(H)}{\binom{r}{2}}\sum_{j\in S_i}f(a_ia_j),
$$
where the second equality uses Claim~\ref{same_split} with $u = a, v = a_i$, and the last equality uses $\sum_{a\in B_i}d_{H_J}(a) = (r-1)e(H)/\binom{r}{2}$, which holds by the assumption that $e_{H_J}(B_i,B_j) = e(H)/\binom{r}{2}$ for all $i < j$ ($H$ is uniform). Now, we get
$$
f(H_J) = \frac{e(H)}{\binom{r}{2}}\left(\sum_{\substack{1\leq i< j\leq r,\\ i\notin S_j, j\notin S_i}}f(a_ia_j)
+\sum_{1\leq i\leq r}\sum_{j\in S_i}f(a_ia_j)\right).
$$
Indeed, using that $(A_i,A_j)$ can not be split for both $i$ and $j$, we see that each pair $1 \leq i < j \leq r$ appears exactly once in the above two sums. 
Hence,
$$
f(H_J) = \frac{e(H)}{\binom{r}{2}}\sum_{1\leq i< j\leq r}f(a_ia_j) = x\frac{e(H)}{\binom{r}{2}}.
$$
As this holds for every $H$-copy $H_J$ in $F_J$, we get that
$$
f(F_J) = \frac{|H^*|}{|H|}\cdot x\frac{e(H)|H^*|}{\binom{r}{2}|H|}>0.
$$
So $f(F_J) \geq 1$, as required. This proves Lemma~\ref{SameDisc}.
\end{proof}
Using Lemma~\ref{SameDisc}, we can now conclude the proof of Lemma~\ref{all k positive} (for uniform $H$). If $R$ has a template for $H$ of size $r+1$ then we are done by Lemma~\ref{template}. Else, by Lemma~\ref{SameDisc}, for every $H^*$-copy $J \in F^*$ with $J \cap V_0 = \emptyset$, every $H$-factor of $J$ has (strictly) positive discrepancy. 
Let $F$ be a perfect $H$-factor in $G'$, obtained by taking a perfect $H$-factor of each $J \in F^*$. Note that at most $|V_0|$ many $H^*$-copies $J$ contain a vertex of $V_0$, and each $H$-factor in $H^*$ contains at most $e(H^*)$ edges. 
Hence,
$$
f(F)\geq \frac{n}{|H^*|}-|V_0|-|V_0| \cdot e(H^*) \geq \gamma n,
$$
as required.

\section{Violating the $C_4$-condition}\label{sec:no_C4}
In this section we handle graphs $H$ that violate the $k$-wise $C_4$-condition for a certain $k$. This forms an important part in the proofs of our main results. As always, $r$ denotes the chromatic number of $H$. The main result is as follows.
\begin{lemma}\label{no_c4}
Suppose that $H$ violates the $k$-wise $C_4$-condition, where $k \geq \max\{r,5\}$ or $k = r = 4$. Then 
$$
\delta^*(H)\leq \max\{\delta_0(H),1-1/\chi^*(H),1-1/(k-1)\}.
$$
\end{lemma}
\noindent
Before proving Lemma~\ref{no_c4}, let us prove the following important corollary. 
\begin{corollary}\label{cor: no_c4 for r<4}
Let $H$ be an $r$-chromatic graph. If $r \geq 3$ then 
$\delta^*(H)\leq 1 - 1/(r+1)$, and if $r=2$ then $\delta^*(H)\leq 3/4$.
\end{corollary}
\begin{proof}
The key is to observe that an $r$-chromatic graph $H$ fails the $(r+2)$-wise $C_4$-condition. Indeed, take any $r$-coloring $A_1,\dots,A_r$ of $H$. Then, considering the $(r+2)$-coloring $A_1,\dots,A_r,A_{r+1},A_{r+2}$ with $A_{r+1} = A_{r+2} = \emptyset$, we see that $e_H(A_1,A_2) + e_H(A_{r+1},A_{r+2}) - e_H(A_1,A_{r+1}) - e_H(A_2,A_{r+2}) = e_H(A_1,A_2) > 0$ (as $H$ is $r$-chromatic). So indeed $H$ violates the $(r+2)$-wise $C_4$-condition. For $r \geq 3$ (resp. $r = 2$), the corollary now follows by Lemma~\ref{no_c4} applied with $k=r+2 \geq 5$ (resp. $k = 5$). 
\end{proof}

We now proceed with the proof of Lemma~\ref{no_c4}. As always, we work under the setup introduced in Section~\ref{sec_app_regularity}. In particular, we always assume that 
\begin{equation}\label{eq:delta(R) no_c4}
\delta(R)/|R| \geq \max\{\delta_0(H),1-1/\chi^*(H),1-1/(k-1)\}+\eta/2.
\end{equation}
Recall the definition of a $(K_k,+)$- and $(K_k,-)$-star, and the head of such a star (see Definition~\ref{def:star}). Evidently, every 2-edge-colored triangle is either monochromatic or a star. 
The proof of Lemma~\ref{no_c4} is split into two cases: $k = r = 4$ and $k \geq 5$. 
The difference between these cases stems from the fact that the $(K_4,+)$-star has zero discrepancy (while the $(K_k,+)$-star has non-zero discrepancy for $k \geq 5$). 

\subsection{Proof of Lemma~\ref{no_c4}: $k \geq 5$}
Here we prove Lemma~\ref{no_c4} in the case $k \geq 5$. 
If there exists a copy of $K_{k}$ in $R$ which is neither monochromatic nor a star with respect to $f_R$, then, by Lemma~\ref{no_c4_template}, this copy of $K_k$ is a template for $H$, and then by Lemma~\ref{template}, there exists a perfect $H$-factor in $G'$ with high discrepancy. Therefore, let us assume that all the copies of $K_{k}$ in $R$ are either monochromatic or a star.
In the following argument, we make repeated use of the following three facts:
\begin{enumerate}[label=\textbf{F.\arabic*}]
    \item\label{fact1} Every four vertices in $R$ have at least one common neighbor.
    \item\label{fact2} For all $k'<k$, each copy of $K_{k'}\subseteq R$ is contained in some copy of $K_k\subseteq R$.
    \item\label{fact3} For all $3\leq k'\leq k$, every copy of $K_{k'}$ which contains a non-monochromatic triangle must be a star with the head of the triangle being the head of the star. 
\end{enumerate}
\ref{fact1} and \ref{fact2} follow from $\delta(R)>1-1/(k-1)\geq 3/4$, by \eqref{eq:delta(R) no_c4}. And \ref{fact3} follows from \ref{fact2}, since otherwise there is a copy of $K_k$ in $R$ which is neither a star nor monochromatic.
The following claim is an important step in this proof. 
\begin{claim}\label{always_head}
If $v\in V(R)$ is the head of some non-monochromatic triangle $T\subseteq R$. Then, for every triangle $T'\subseteq R$ with $v\in V(T')$, it holds that $f_R(T') = f_R(T)$ (i.e. $T'$ is colored the same way as $T$)
and $v$ is the head of $T'$.
\end{claim}
\begin{proof}
Let $T = \{u,v,w\} \subseteq R$ be as in the statement and let us assume without loss of generality that $f_R(T) = 1$, meaning that $f_R(vu) = f_R(vw) = 1$ and $f_R(uw) = -1$.
Let $T'\subseteq R$ be an arbitrary triangle with $v \in T'$, and write $T' = \{u',v,w'\}$. By \ref{fact1}, the vertices $u,v,w,u'$ have a common neighbor $x \in R$. Note that $R[\{u,v,w,x\}]$ is a copy of $K_4$ in $R$ containing $T$ and thus, by \ref{fact3}, we have $f_R(vx) = -f_R(wx) = 1$. Therefore, $R[\{v,w,x\}]$ is a non-monochromatic triangle with $v$ as its head. By \ref{fact1}, there exists $y\in R$ such that $y$ is a common neighbor of $v,w,x,u'$. By applying \ref{fact3} to the $4$-clique $\{v,w,x,y\}$ we get that $f_R(vy) = -f_R(xy) = 1$, and by applying \ref{fact3} to the $4$-clique $\{v,u',x,y\}$ we get that $f_R(vu')= -f_R(u'y) = 1$. Finally, let $z\in R$ be a common neighbor of $v,y,u',w'$. Using \ref{fact3} as before, we find that $f_R(vz)=-f_R(u'z) = 1$ by considering the $4$-clique $\{v,u',y,z\}$, and that $f(vw') = 1 = -f(u'w') = 1$ by considering the $4$-clique $\{v,u',w',z\}$. 
So $T'$ is indeed a non-monochromatic triangle with $v$ as its head and $f_R(T') = 1$.
\end{proof}
Claim~\ref{always_head} implies that if $v$ is the head of a non-monochromatic triangle, then $v$ is not contained in any monochromatic triangle and must be the head of any (non-monochromatic) triangle containing it. Also, all edges inside $N_R(v)$ (the neighborhood of $v$ in $R$) have the same color.

Now, we can make a further statement about the coloring of the copies of $K_k$ in $R$. Recall that $K_{k,+}$ (resp. $K_{k,-}$) denotes the monochromatic $k$-clique where all edges have color $1$ (resp. $-1$). 
\begin{claim}\label{one_star}
 Either every copy of $K_k$ in $R$ is a copy of $K_{k,+}$ or the $(K_k,-)$-star, or every copy of $K_k$ in $R$ is a copy of $K_{k,-}$ or the $(K_k,+)$-star.
\end{claim}
\begin{proof}
Observe that $K_{k,+}$ and a $(K_k,-)$-star both contain a monochromatic triangle in color $1$, and similarly, $K_{k,-}$ and a $(K_k,+)$-star both contain a monochromatic triangle in color $-1$. Therefore, if the claim does not hold, then $R$ contains monochromatic triangles $L_+$ in color $1$ and $L_-$ in color $-1$. 
By Claim~\ref{always_head}, none of the vertices in $L_+$ and $L_-$ are the heads of any stars of size $3$, because they belong to a monochromatic triangle. Let $v_1,v_2\in L_+$ and $v_3,v_4\in L_-$ and by \ref{fact1}, let $u$ be a common neighbor of $v_1,v_2,v_3,v_4$.
%Note that Claim~\ref{always_head} implies, that all the triangles incident to some vertex which is incident to a non-monochromatic triangle must have the same color with respect to $f_R$.
Note that $u$ cannot be the head of any non-monochromatic triangle, since this would imply (by Claim~\ref{always_head}) that all edges in $N_R(u)$ have the same color, while $f_R(v_1v_2) \neq f_R(v_3v_4).$ 
It follows that the triangles $\{u,v_1,v_2\}$,$\{u,v_3,v_4\}$ are monochromatic, because none of the vertices $u,v_1,\dots,v_4$ can be the head of a non-monochromatic triangle. 
So we have $f_R(uv_1) = f_R(uv_2) = -f_R(uv_3) = -f_R(uv_4) = 1$. 
Let $w$ be a common neighbor of $u,v_1,v_3$ (using \ref{fact1}). Without loss of generality, suppose that $f_R(uw) = 1$. Then the triangle $w,u,v_3$ is not monochromatic, and its head must be $w$. Now, by Claim~\ref{always_head}, the triangle $w,u,v_1$ must also be non-monochromatic with head $w$. This implies that $f_R(uv_1) = f_R(uv_3)$, a contradiction.
\end{proof}
By Claim~\ref{one_star} and without loss of generality, let us assume that every copy of $K_{k}$ in $R$ is either a copy of $K_{k,+}$ or a $(K_{k},-)$-star. 
Note that both $K_{k,+}$ and a $(K_{k},-)$-star have positive discrepancy, because $k \geq 5$. 
We now consider two sub-cases based on whether $H$ is regular.
\paragraph{Case 1:} $H$ is non-regular. If $k=r$, then every copy of $K_r$ in $r$ has positive discrepancy. We then get a perfect $H$-factor in $G'$ with high discrepancy by Lemma~\ref{all k positive}. So let us assume that $k\geq r+1$.
Suppose first that there is a $(K_k,-)$-star $K$ in $R$. Clearly, $K$ contains a $(K_{r+1},-)$-star $K' \subseteq K$. 
By Corollary~\ref{non_reg_clique}, $(K',f_R)$ is a template for $H$ (as $H$ is non-regular). Now, by Lemma~\ref{template}, $G$ has a perfect $H$-factor with high discrepancy, completing the proof in this case. Therefore, we may assume that every copy of $K_{k}$ in $R$ is monochromatic in color 1. Then, by \ref{fact2} with $k'=2$, all edges in $R$ have color $1$. 
By Lemma~\ref{all k positive} again, $G'$ has a perfect $H$-factor with high discrepancy.
\paragraph{Case 2:}
$H$ is $d$-regular for some $d\in \mathbb{N}$.
Let $U\subseteq V(R)$ be the set of vertices which are the heads of a $(K_3,-)$-star in $R$. Observe that if $uv \in E(R)$ has color $-1$ then $u \in U$ or $v \in U$. Indeed, by \ref{fact2}, $uv$ is contained in some triangle in $R$, and this triangle must be a $(K_3,-)$-star (as every triangle in $R$ is either a $K_{3,+}$ or a $(K_3,-)$-star). The head of this star must be $u$ or $v$, so one of them is in $U$. 
We see that $R[V(R)\backslash U]$ only has edges colored $1$.
Additionally, $U$ is an independent set in $R$. To see this, let $u,v\in U$ and assume towards a contradiction that $uv\in E(R)$. By \ref{fact2}, $uv$ is contained in some triangle in $R$. By Claim~\ref{always_head} and the definition of $U$, both $u$ and $v$ must be the head of this triangle, a contradiction.

By \eqref{eq:delta(R) no_c4} and as $k \geq 5$, we have 
$\delta(R)\geq (3/4+\eta/2)|R|$.
Since $U$ is an independent set in $R$,
we must have
$
|U|\leq (1/4-\eta/2)|R|.
$
Therefore, $V_U := \bigcup_{u \in U}V_u$ satisfies
$$|V_U| \leq (1/4-\eta/2)n.$$
Note that all the edges of color $-1$ in $G'$ are incident to either $V_U$ or $V_0$, because all edges in $R$ outside $U$ have color $1$.
By Lemma~\ref{lem:H* factor}, $G'$ has a perfect $H$-factor $F$. 
Since $H$ is $d$-regular, so is $F$. 
Hence, the number of edges of color $-1$ in $F$ is at most $(|V_U| + |V_0|) \cdot d \leq nd/4$.
It follows that 
$$
f(F)\geq \frac{d}{2}n-nd/4 \geq \gamma n.
$$
This concludes the proof.

\subsection{Proof of Lemma~\ref{no_c4}: $k = r = 4$}
Here we prove Lemma~\ref{no_c4} in the case $k = r = 4$. 
If $R$ contains a template for $H$ of size $4$, then by Lemma~\ref{template}, there exists a perfect $H$-factor in $G'$ with high discrepancy, as required. So let us assume that $R$ contains no such template. 

We consider two cases.
Suppose first that the $(K_4,+)$-star is not a template for $H.$ 
We claim that in this case, $\delta^*(H) = \delta_0(H) = 3/4$. 
For the upper bound, recall that if $H$ violates the $k$-wise $C_4$-condition, then it also violates the $(k+1)$-wise $C_4$-condition. Hence, $H$ violates the $5$-wise $C_4$-condition, and by the case $k \geq 5$ of Lemma~\ref{no_c4}, we have $\delta^*(H)\leq 3/4$. For the lower bound, 
let $c$ be the $2$-edge-coloring of $K_4$ corresponding to the $(K_4,+)$-star. Note that $c(K_4) = 0$. Let $B$ be the $3!|H|$-blowup of $(K_4,c)$. By Lemma~\ref{lem:balanced blowup}, there is a perfect $H$-factor of $B$ with discrepancy $0$, as $c(K_4)=0$. As we assumed that the $(K_4,+)$-star is not a template for $H$, we get that $c \in \mathcal{K}(H)$, and hence
$$
\delta_0(H)\geq \delta(B)/|B|=3/4,
$$
as required. 

From now on, let us assume that the $(K_4,+)$-star is a template for $H$, and by symmetry so is the $(K_4,-)$-star. 
As we assumed that $R$ has no template for $H$, we get that $R$ contains no $(K_4,+)$-star and no $(K_4,-)$-star. It now follows, by Lemma~\ref{no_c4_template}, that all copies of $K_4$ in $R$ are monochromatic. 
% By Lemma~\ref{template}, we may then assume that there does not exist a copy of the $(K_4,+)$- or $(K_4,-)$-star in $R$ and by Lemma~\ref{no_c4_template}, we may further assume that all the copies of $K_4$ in $R$ are monochromatic. 
By \eqref{eq:delta(R) no_c4}, we have $\delta(R)/|R| > 2/3$. This implies that each triangle in $R$ is contained in a $K_4$, and hence all triangles in $R$ are monochromatic. 
We claim that all edges of $R$ have the same color. Suppose not. Then, as $R$ is connected (by $\delta(R)/|R| > 2/3$), there exist vertices $u,v,w\in V(R)$ such that $f_R(uv) = 1$ and $f_R(vw) = -1$. Again using $\delta(R)/|R|>2/3$, there exists a common neighbor $x$ of $u,v,w$. It follows that either $x,u,v$ or $x,v,w$ form a non-monochromatic triangle in $R$, depending on the color of $xv$ with respect to $f_R$. 
This gives a contradiction. So we see that $R$ is monochromatic, which means that all edges of $G'$ not touching $V_0$ have the same color. By Lemma~\ref{lem:H* factor}, $G'$ has an $H$-factor $F$.
Now we get
$$
|f(F)|\geq \left(\frac{n}{|H|}-2|V_0|\right)e(H)\geq \gamma n.
$$
\section{Non-regular $H$}\label{sec: nonregular}
In this section we deal with the case that $H$ is non-regular. This comprises the main part of the proofs of Theorems~\ref{tripartite1} and~\ref{rpartite}. We shall prove three key lemmas (Lemmas~\ref{c4_r_1},~\ref{rest_4} and~\ref{high_min_deg}) that are used in the proofs of these theorems. 
The basic idea in the proof of these lemmas is as follows. First, in all cases, the minimum degree assumption implies that the reduced graph $R$ contains $r$-cliques. Then, by Lemma~\ref{all k positive}, we may assume that there exists an $r$-clique $L_1$ with positive discrepancy, as well as an $r$-clique $L_2$ with negative discrepancy. Using Lemma~\ref{connectivity}, we can then connect $L_1,L_2$ with a sequence of $r$-cliques $L_1 = L'_1,\dots,L'_{\ell} = L_2$ with each pair of consecutive cliques $L'_i,L'_{i+1}$ intersecting in at least $r-1$ or at least $r-2$ vertices, depending on the assumed minimum degree of $R$. 
We therefore have two $r$-cliques sharing $r-1$ or $r-2$ vertices, one having positive discrepancy and the other negative. With a slight abuse of notation, we assume that $L_1,L_2$ are such $r$-cliques. Then, either $L_1 \cup L_2$ is a template for $H$ (in which case we are done by Lemma~\ref{template}), or the coloring of $L_1,L_2$ has some specific structure, by the lemmas from Section~\ref{sec:templates}. 
In more involved cases (mainly Lemma~\ref{high_min_deg}), we determine properties of the coloring on a large portion of $R$, under the assumption that $R$ has no small template for $H$. 

In each of the three lemmas we shall make certain assumptions on the residue of $r$ modulo $4$, which correspond to different cases in the proof of Theorem~\ref{rpartite}. We also often assume that $H$ satisfies the $r$-wise $C_4$-condition. (If $H$ violates the $r$-wise $C_4$-condition, then  Lemma~\ref{no_c4} immediately gives the required bounds for Theorem~\ref{rpartite}, as we shall see in Section~\ref{subsec:proof_rpartite}.)
The first lemma is as follows. 

\begin{lemma}\label{c4_r_1}
If $r\not \equiv_4 0$ and $H$ is non-regular, then
$
\delta^*(H)\leq 1-1/r.
$
\end{lemma}
In the proof of Lemma~\ref{c4_r_1}, we may assume that $\delta(R)/|R| > 1 - 1/r$.
This assumption has two important consequences: First, it guarantees that $|L_1 \cap L_2| = r-1$, and second, it implies that every $r$-clique is contained in an $(r+1)$-clique. This allows us to use Lemma~\ref{sharing_much_template} and Corollary~\ref{non_reg_clique} to conclude the proof. The details \nolinebreak follow. 
\begin{proof}[Proof of Lemma~\ref{c4_r_1}]
As always, we work under the setup described in Section~\ref{sec_app_regularity}. In particular, as we are aiming for the bound $\delta^*(H) \leq 1-1/r$, we assume that 
\begin{equation*}\label{eq:degree > 1-1/r}
    \delta(G')/n, \; \delta(R)/|R|\geq (1-1/r+\eta/2).
\end{equation*}
Our goal is to show that $G'$ contains a perfect $H$-factor with high discrepancy.
%, proving that $\delta^*(H)\leq 1-1/r$.
If $R$ has a template for $H$ of size $r+1$, then we are done by Lemma~\ref{template}. We therefore assume that $R$ has no such template. This implies that for every $(r+1)$-clique $M$ in $R$, $M^+$ is regular (with respect to $f_R$), because otherwise $(M,f_R)$ is a template for $H$ by Corollary~\ref{non_reg_clique} (as $H$ is non-regular).   
Next, we need the following very simple claim.
\begin{claim}\label{claim:d-regular (r+1)-clique}
Let $M$ be an $(r+1)$-clique with an edge-coloring $c$, let $d$ be such that $M^+$ is $d$-regular, and let $L \subseteq M$, $|L| = r$. Then $c(L) = (r-1)(d - r/2)$.
\end{claim}
\begin{proof}
$e(L^+) = e(M^+) - d = d(r+1)/2 - d = d(r-1)/2$. Hence, 
$$
c(L)= 2e(L^+) - \binom{r}{2} = 
2 \cdot d(r-1)/2-\binom{r}{2} = d(r-1) - \binom{r}{2} = (r-1)(d - r/2).
$$
\end{proof}

We now continue with the proof of the lemma. 
Suppose first that there exist two copies $M_1,M_2\subseteq R$ of $K_{r+1}$ such that $M_1^+$ is $d$-regular and $M_2^+$ is $d'$-regular for some $d\neq d'$. Let $L_1\subseteq M_1$ and $L_2\subseteq M_2$ of size $r$ each. Since $d\neq d'$, it follows by Claim~\ref{claim:d-regular (r+1)-clique} that $f_R(L_1)\neq f_R(L_2)$. By Lemma~\ref{connectivity} there exists a sequence $L'_1,L'_2,...L'_\ell\subseteq R$ of copies of $K_r$ with $L'_1 = L_1$ and $L'_\ell = L_2$ and such that $L'_i$ and $L'_{i+1}$ share $r-1$ vertices for each $1\leq i\leq\ell-1$. But then, there must exist some $1\leq i\leq\ell-1$ such that $f_R(L'_i)\neq f_R(L'_{i+1})$.
Now $(L'_i \cup L'_{i+1}, f_R)$ is a template for $H$ by Lemma~\ref{sharing_much_template}, in contradiction to our assumption. 

So from now on, we assume that there exists $d\in\mathbb{N}$ such that for every $(r+1)$-clique $M \subseteq R$, $M^+$ is $d$-regular. Trivially, $2 \mid d(r+1).$ 
Note that $d \neq r/2$, because if $r+1$ is even then $r$ is odd and so $d \neq r/2$, and if $r+1$ is odd then $d$ must be even, so $d \neq r/2$ as $r\not\equiv_4 0.$  
Without loss of generality, let us assume that $d> r/2$ (otherwise consider $M^-$ in place of $M^+$, replacing $d$ with $r-d$). We claim that every copy $L$ of $K_r$ in $R$ has positive discrepancy. Indeed, as $\delta(R)/|R| > 1-1/r$, there exists an $(r+1)$-clique $M$ containing $L$. Now, Claim~\ref{claim:d-regular (r+1)-clique}, $f_R(L) = (r-1)(d - r/2) > 0$. Finally, by Lemma~\ref{all k positive}, $G'$ has a perfect $H$-factor with high discrepancy, as required.
\end{proof}

\noindent
The following is the second of the three lemmas. 

\begin{lemma}\label{rest_4}
Suppose that $r\equiv_4 0$. Assume that $H$ is non-regular, fulfills the $r$-wise $C_4$-condition, and violates the $(r+1)$-wise $C_4$-condition. Then
$$
\delta^*(H)\leq \max\{\delta_0(H),1-1/\chi^*(H)\}.
$$
\end{lemma}

The proof of Lemma~\ref{rest_4} proceeds by distinguishing between two cases. The first case is that there exists an $H$-factor which is not balanced-uniform (recall Definition~\ref{def:balanced-uniform}). In this case we will show that $\delta_0(H) = 1-1/r$ and hence $\delta^*(H) \geq 1-1/r$, and this will match the upper bound on $\delta^*(H)$ we get from Lemma~\ref{no_c4}. Here the assumption $r \equiv_4 0$ will play a crucial role. 
The second case is that there exists a non-balanced-uniform $H$-factor. Here we will proceed by finding two $r$-cliques $L_1,L_2$ with $|L_1 \cap L_2| \geq r-2$ and such that $L_1$ has positive discrepancy and $L_2$ has negative discrepancy, as explained above.  We will eventually conclude that $L_1 \cup L_2$ is a template for $H$ by Lemma~\ref{non-balanced-uniform template2}, finishing the proof. The details follow. 

\begin{proof}[Proof of Lemma~\ref{rest_4}]
As always, we work under the setup of Section~\ref{sec_app_regularity}. In particular, we assume that 
$$
\delta(G')/n, \; \delta(R)/|R|\geq \max\{\delta_0(H),1-1/\chi^*(H)\}+\eta/2.
$$
As $\chi^*(H) \geq r-1$ for every $r$-chromatic graph, we have
\begin{equation}\label{eq:(r-2)/(r-1)}
    \delta(R)/|R| > \frac{r-2}{r-1}. 
\end{equation}
Our goal is to show that $G'$ contains a perfect $H$-factor with high discrepancy. 
Since $H$ violates the $(r+1)$-wise $C_4$-condition, we may apply
Lemma~\ref{no_c4} with $k=r+1$ to get
$$
\delta^*(H)\leq 1-1/r.
$$
Hence, if $\max\{\delta_0(H),1-1/\chi^*(H)\} = 1-1/r$ then we are done. 
So from now on we assume that
\begin{equation}\label{eq:second main lemma, <1-1/r}
\max\{\delta_0(H),1-1/\chi^*(H)\}<1-1/r.
\end{equation}
In particular, $\chi^*(H) < r$, which implies that $H$ has an unbalanced $r$-coloring. We now consider two cases. For what follows, recall Definition~\ref{def:balanced-uniform}.
\paragraph{Case 1:} Every $H$-factor is balanced-uniform. We will show that then $\delta_0(H) = 1-1/r$, which would contradict our assumption \eqref{eq:second main lemma, <1-1/r} and hence conclude the proof in this case. Fix a $2$-edge-coloring $c$ of $K_r$ such that $c(K_r)=0$; such a coloring exists because $K_r$ has an even number of edges, as $r\equiv_4 0$. By Lemma~\ref{lem:if balanced-uniform then not a template}, we have $c\in\mathcal{K}(H)$. 
Let $B$ be an $m$-blowup of $(K_r,c)$, where $m$ is divisible by $(r-1)!|H|$. By Lemma~\ref{lem:balanced blowup}, $B$ has a perfect $H$-factor. 
We claim that for every perfect $H$-factor $F_B$ of $B$ it holds that $c(F_B) = 0$. 
Indeed, let $B_1,\dots,B_r$ be the parts of $B$. 
By assumption, $F_B$ is balanced-uniform. Hence, $e_{F_B}(B_i,B_j) = e(F_B)/\binom{r}{2}$ for all $1 \leq i < j \leq r$. Therefore, $c(F_B) = c(K_r) \cdot e(F_B)/\binom{r}{2} = 0$, as required. 
It follows that $\delta_0(H) \geq \delta(B)/|B|=1-1/r$, as claimed. 

\paragraph{Case 2:} There exists a non-balanced-uniform union $F$ of disjoint copies of $H$.\footnote{It is worth noting that the argument in Case 2 only requires that $r \not\equiv_4 1$ (instead of the stronger assumption $r \equiv_4 0$).} By Claim~\ref{claim:non-balanced-uniform}, there exists a balanced $r$-coloring $A_1,\dots,A_r$ of $F$ such that $e_F(A_1,A_2)\neq e_F(A_3,A_4)$.
We may assume that $R$ contains no template for $H$ on at most $r+2$ vertices, as otherwise, by Lemma~\ref{template}, $G'$ contains a perfect $H$-factor of high discrepancy and we are done. 

If $R$ contains an $r$-clique $L$ such that $L^+$ is non-regular (with respect to $f_R$), then by Lemma~\ref{non-balanced-uniform template}, $(L,f_R)$ is a template for $H$, contradicting our assumption. So suppose from now on that every $r$-clique $L$ in $R$ is such that $L^+$ is regular.

Assume first that there exist two $r$-cliques $L_1,L_2\subseteq R$ such that $L_1^+$ is $d$-regular and $L_2^+$ is $d'$-regular for some $d\neq d'$.
Using \eqref{eq:(r-2)/(r-1)} and Item 2 of Lemma~\ref{connectivity}, we obtain a sequence of $r$-cliques $L_1',L_2',\dots,L_\ell'\subseteq R$ with $L_1'=L_1$ and $L_\ell'=L_2$, such that each pair of subsequent $r$-cliques share at least $r-2$ vertices.
So there exist two $r$-cliques $L,L'$  in $R$ sharing at least $r-2$ vertices, such that $L^+$ is $d$-regular and $L'^+$ is $d'$-regular for some $d\neq d'$. Without loss of generality, let us assume that $L_1,L_2$ were such $r$-cliques to begin with. 
If $|L_1 \cap L_2| = r-1$ then $(L_1\cup L_2,f_R)$ is a template for $H$ by Lemma~\ref{sharing_much_template}, and if $|L_1 \cap L_2| = r-2$ then $(L_1\cup L_2,f_R)$ is a template for $H$ by Lemma~\ref{non-balanced-uniform template2}. In either case, we get a contradiction to our assumption. 

Now assume that $r$-cliques $L$ in $R$ are $d$-regular with the same $d$ (in the sense that $L^+$ is $d$-regular). 
Without loss of generality, let us assume that $d\geq (r-1)/2$, as otherwise we may swap the colors, replacing $d$ with $r-1-d$. 
Note that $d \neq \frac{r-1}{2}$ because $r \not \equiv_4 1$.
As $d > \frac{r-1}{2}$, all copies of $K_r$ in $R$ have positive discrepancy. Now, by Lemma~\ref{all k positive}, $G'$ has a perfect $H$-factor with high discrepancy, completing the proof.
\end{proof}

Finally, we arrive at the last of the three main lemmas, Lemma~\ref{high_min_deg}. This lemma deals with the case $r \equiv_4 2,3$. Its proof is by far the most involved part of this section. 
Recall the definition of a butterfly from the introduction. 
\begin{lemma}\label{high_min_deg}
Suppose that $r \geq 3$ and $r\equiv_4 2,3$. Assume that $H$ satisfies the $r$-wise $C_4$-condition and is non-regular. Then 
$$\delta^*(H) \leq 
\begin{cases}
\max\{\delta_0(H), 1-1/\chi^*(H),4/7\} & \text{ $r = 3$ and some butterfly is not a template for $H$,}
\\
\max\{\delta_0(H), 1-1/\chi^*(H)\} & \text{otherwise.}
\end{cases}
$$
\end{lemma}

Let us comment on the proof of Lemma~\ref{high_min_deg}. Similarly to previous proofs, the proof of Lemma~\ref{high_min_deg} begins by finding two $r$-cliques $L_1,L_2$ with $|L_1 \cap L_2| = r-2$ such that $L_1$ has positive discrepancy and $L_2$ has negative discrepancy. As always, we may assume that $R$ contains no small template for $H$, as otherwise we are done by Lemma~\ref{template}. In particular, 
$L_1 \cup L_2$ is not a template for $H$. Then, by Lemma~\ref{template_example}, $H$ is $(s,t)$-structured (for $s,t$ given by that lemma). In the case $r \neq 3$ (namely $r \geq 6$), we can use
Lemma~\ref{lower_delta} to deduce that $\max\{\delta_0(H), 1-1/\chi^*(H)\} \geq \frac{3r-5}{3r-2}$, which allows us to assume that $\delta(R)/|R| > \frac{3r-5}{3r-2}$. This minimum degree assumption is crucial for the proof, allowing us to establish various structural properties of $R$ and $f_R$. Eventually we show that $R$ is strongly tilted towards one of the colors, which allows us to find a perfect $H$-factor of high disrepancy. 

The case $r = 3$ is somewhat different. 
Note that $\frac{3r-5}{3r-2} = \frac{4}{7}$ for $r=3$.  
Big parts of the proof for $r \neq 3$ carry over to the case $r=3$, provided we assume that $\delta(R)/|R| > \frac{4}{7}$. However, we may not make this assumption in all cases, because for some $3$-chromatic graphs $H$, the value of $\delta^*(H)$ is smaller than $\frac{4}{7}$. It turns out that making the assumption $\delta(R)/|R| > \frac{4}{7}$ is justified exactly when some butterfly is not a template for $H$ (cf. Lemma~\ref{butterfly}). 
The proof of Lemma~\ref{high_min_deg} is given in the next subsection. 

\subsection{Proof of Lemma~\ref{high_min_deg}}
By Lemma~\ref{c4_r_1}, we have 
\begin{equation}\label{low_bound}
\delta^*(H)\leq 1-1/r.
\end{equation}
Therefore, if $\max\{\delta_0(H),1-1/\chi^*(H)\}=1-1/r$ then the assertion of Lemma~\ref{high_min_deg} holds. 
So from now on, we assume that $\max\{\delta_0(H),1-1/\chi^*(H)\} < 1-1/r$. In particular,
$\chi^*(H)<r$. By the definition of $\chi^*$, this implies that 
\begin{align}\label{stmt: unbalanced exists}
    \text{$H$ has an unbalanced $r$-coloring.}
\end{align}  
 As always, we work under the setup of Section~\ref{sec_app_regularity}, therefore assuming that 
\begin{equation}\label{eq:delta(R) proof of Lemma 11}
    \delta(R)/|R|\geq 
\max\{\delta_0(H), 1-1/\chi^*(H)\} + \eta/2,
\end{equation}
and
\begin{equation}\label{stmt:4/7}
    \delta(R)/|R| \geq 4/7 + \eta/2 \text{ if $r=3$ and some butterfly is not a template for $H$.}
\end{equation}
Since $\chi^*(H) > r-1$ for every $r$-chromatic $H$, \eqref{eq:delta(R) proof of Lemma 11} implies that
\begin{equation}\label{eq:delta>(r-2)/(r-1)}
\delta(R)/|R| > \frac{r-2}{r-1}.
\end{equation}
We shall show that $G'$ has a perfect $H$-factor with high discrepancy. Throughout the proof, we assume that $R$ contains no template for $H$ on at most $r+2$ vertices, as otherwise we are done by Lemma~\ref{template}. 
The following claim is used multiple times throughout the proof.
\begin{claim}\label{claim:K_r_chain}
There is no sequence $L_1',L_2',\dots,L'_{\ell} \subseteq R$ of copies of $K_r$,  such that $|L_i' \cap L_{i+1}'| \geq r-1$ for every $1 \leq i \leq \ell-1$ and $f_R(L_1')\neq f_R(L_\ell')$.
\end{claim}
\begin{proof}
As $f_R(L_1') \neq f_R(L_\ell')$, there exists $1 \leq i \leq \ell-1$ such that $f_R(L'_i) \neq f_R(L'_{i+1})$. Using that $H$ is non-regular, we get that $(L_i'\cup L_{i+1}',f_R)$ is a template for $H$ by Lemma~\ref{sharing_much_template}. However, we assumed that $R$ has no such template for $H$, a contradiction. 
\end{proof}

By \eqref{eq:delta>(r-2)/(r-1)}, $R$ contains a copy of $K_r$. 
Since $r \equiv_4 2,3$, $K_r$ has an odd number of edges and thus, without loss of generality, let us assume that $R$ contains a copy $L_1$ of $K_r$ with $f_R(L_1)>0$. 
If all copies of $K_r$ in $R$ have positive discrepancy, then we are done by
Lemma~\ref{all k positive}. Suppose therefore that $R$ also contains a copy $L_2$ of $K_r$ with $f_R(L_2)<0$.
By Lemma~\ref{connectivity} and \eqref{eq:delta>(r-2)/(r-1)}, there exists a sequence $L'_1,L'_2,\dots, L'_\ell\subseteq R$ of copies of $K_{r}$ with $L'_1 = L_1$ and $L'_\ell = L_2$ such that every pair of subsequent copies share at least $r-2$ vertices. Therefore, there exist two copies of $K_r$ sharing at least $r-2$ vertices with discrepancies of different signs. Without loss of generality, let us assume that $L_1$ and $L_2$ are such copies.
If $|L_1 \cap L_2| = r-1$ then this is a contradiction to Claim~\ref{claim:K_r_chain}. Suppose then that $|L_1 \cap L_2| = r-2$. Put $V = L_1 \cap L_2$, $L_1 \setminus L_2 = \{x_1,y_1\}$ and $L_2 \setminus L_1 = \{x_2,y_2\}$.

We proceed with the proof of Lemma~\ref{high_min_deg}. 
By assumption, $(L_1 \cup L_2, f_R)$ is not a template for $H$. Hence, by 
Lemma~\ref{template_example}, $H$ is $(s,t)$-structured with 
\begin{equation}\label{define s}
s := \frac{f_R(L_1)-f_R(x_1y_1)-f_R(L_2)+f_R(x_2y_2)}{2(r-2)}
\end{equation}
and 
\begin{equation}\label{define t}  
t := f_R(x_1y_1) - f_R(x_2y_2).
\end{equation}
Crucially, note that $(s,t) \neq (0,0)$. Indeed, if $t = 0$ then $s = \frac{f_R(L_1) - f_R(L_2)}{2(r-2)} \neq 0$ because $f_R(L_1) \neq f_R(L_2)$.

By the definition of being $(s,t)$-structured (see Definition~\ref{structured}), there exists $\rho \in \mathbb{R}$ such that for every $r$-coloring $A_1,\dots,A_r$ of $H$ and for all $1\leq i<j\leq r$, it holds that
\begin{equation}\label{eq:structured proof of 11}
\rho(|A_i|+|A_j|) = s\cdot e_H(A_i\cup A_j, V(H)\backslash(A_i\cup A_j)) + t\cdot e_H(A_i,A_j).
\end{equation}
First we handle the case that $H$ is uniform (recall Definition~\ref{def:uniform}). Then for every $r$-coloring $A_1,\dots,A_r$ of $H$ and for all $1 \leq i < j \leq r$, it holds that 
\begin{align*}
\rho(|A_i|+|A_j|) &= s \cdot e_H(A_i\cup A_j,V(H)\backslash (A_i\cup A_j))+t \cdot e_H(A_i,A_j) \\ &=
(2(r-2)s + t) \cdot \frac{e(H)}{\binom{r}{2}} = 
\frac{f_R(L_1)-f_R(L_2)}{\binom{r}{2}} \cdot e(H),
\end{align*}
where the second equality uses the uniformity of $H$. 
It follows that $\rho \neq 0$ because $f_R(L_1) \neq f_R(L_2)$. We get that $|A_i| + |A_j|$ is the same for all $1 \leq i < j \leq r$. 
By Claim~\ref{claim:all equal},
$|A_1| = \dots = |A_r|$. This means that $H$ only has balanced $r$-colorings, contradicting \eqref{stmt: unbalanced exists}. 

For the rest of the proof, we assume that $H$ is non-uniform. By Lemma~\ref{sharing_little_template} and as $(L_1 \cup L_2,f_R)$ is not a template for $H$ (by assumption), we have that 
\begin{equation}\label{eq:multiple of 2(r-2)}
f_R(L_1)-f_R(x_1y_1)-f_R(L_2)+f_R(x_2y_2)\in\{-4(r-2), -2(r-2), 0, 2(r-2), 4(r-2)\}.
\end{equation}

In the next claim, for the case $r = 3$, we classify the possible colorings of the triangles $L_1,L_2$. Also, for each case, we specify the corresponding values of $s,t$.

\begin{figure}
    \centering
    \includegraphics{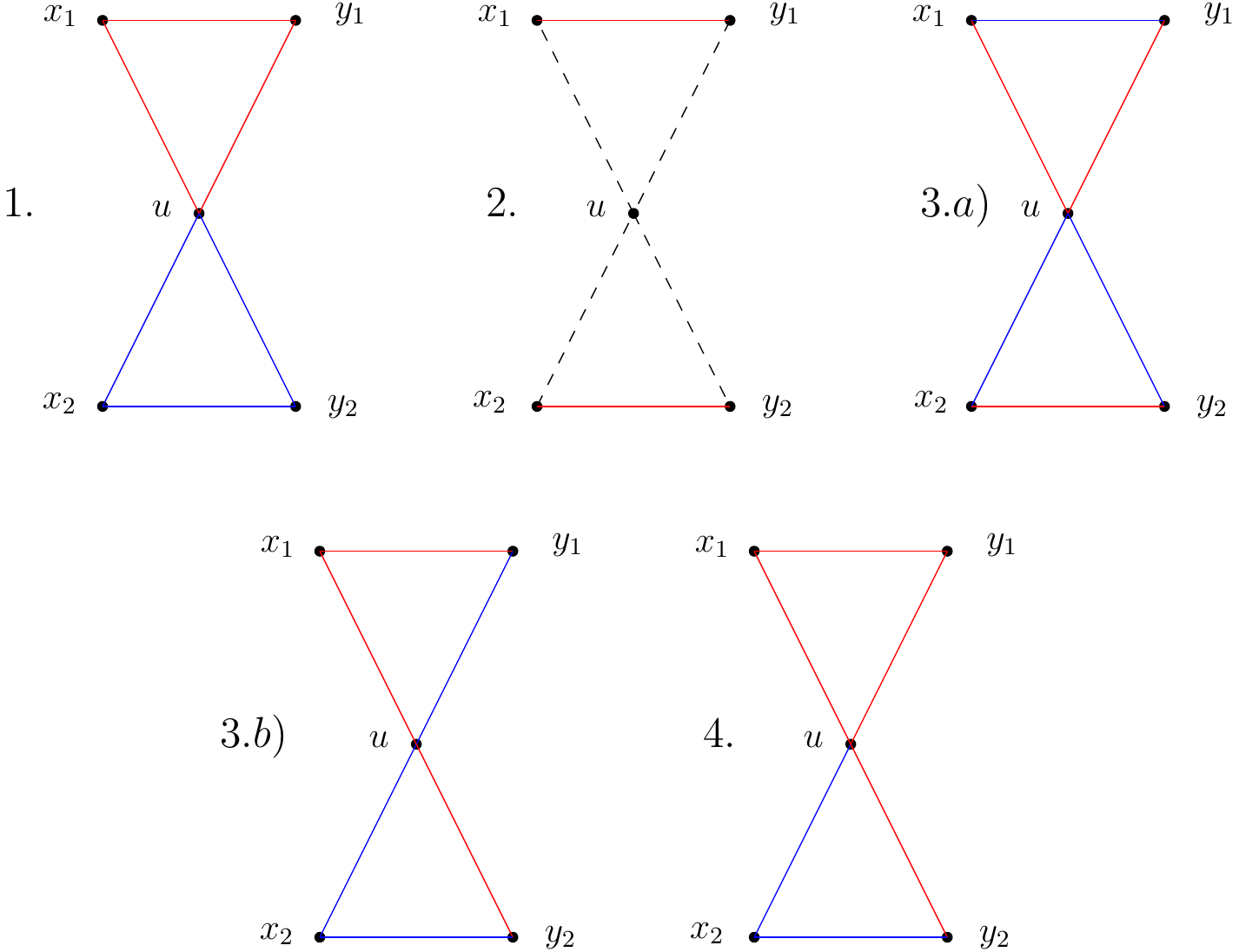}
    \caption{The possible configurations of $L_1, L_2$ corresponding to different cases in Claim~\ref{claim:possible colorings of L1,L2, r=3}.}
    \label{fig:claim11-6}
\end{figure}

\begin{claim}\label{claim:possible colorings of L1,L2, r=3}
    Suppose that $r = 3$. Then one of the following holds:
    \begin{enumerate}
    \item $L_1,L_2$ are monochromatic. In this case $(s,t)=(2,2)$.
    \item $f_R(x_1y_1) = f_R(x_2y_2)$. In this case $t = 0$.
    \item $(L_1 \cup L_2, f_R)$ is a butterfly and $L_1,L_2$ are not monochromatic. In this case $(s,t) \in \{(2,-2),(0,2)\}$.
    \item Exactly one of $L_1$ and $L_2$ is monochromatic and $f_R(x_1y_1)\neq f_R(x_2y_2)$. In this case $(s,t)=(1,2)$.
    \end{enumerate} 
\end{claim}
\begin{proof}
Clearly, if $L_1$ is monochromatic then it has color $1$ (as $f_R(L_1) > 0$), and similarly, if $L_2$ is monochromatic then it has color $-1$. 
If $f_R(x_1y_1) = f_R(x_2y_2)$ then we are in Case 2, and $t = 0$ by \eqref{define t}. 
So let us assume that $f_R(x_1y_1) \neq f_R(x_2y_2)$.
If $L_1$ is monochromatic then $f_R(x_2y_2) = - f_R(x_1y_1) = -1$, so this is covered by Case 1 or Case 4. In both cases, it is immediate to compute $s,t$ from \eqref{define s} and \eqref{define t}.

Suppose then that $L_1$ is not monochromatic. By symmetry (with respect to switching the colors), we can also assume that $L_2$ is not monochromatic. 

Write $L_1 \cap L_2 = \{u\}$. Suppose first 
that $f_R(x_1y_1) = -1$. Then $f_R(x_2y_2) = 1$, and we also must have $f_R(ux_1) = f_R(uy_1) = 1$, $f_R(ux_2) = f_R(uy_2) = -1$ (because $f_R(L_1) > 0$, $f_R(L_2) < 0$). So $(L_1 \cup L_2, f_R)$ is a butterfly (Case 3), and $(s,t) = (2,-2)$ by \eqref{define s} and \eqref{define t}.
Suppose now that
that $f_R(x_1y_1) = 1$. Then, $f_R(x_2y_2) = -1$, and without loss of generality (up to switching $x_1,y_1$ or $x_2,y_2$), we have $f_R(ux_1) = 1$, $f_R(uy_1) = -1$, $f_R(ux_2) = -1$, $f_R(uy_2) = 1$. So again $(L_1 \cup L_2, f_R)$ is a butterfly, and $(s,t) = (0,2)$.
\end{proof}

Next, we address some subcases of the case $r=3$, namely, Cases 1-2 in Claim~\ref{claim:possible colorings of L1,L2, r=3}. These cases need to be handled separately.
\subsubsection{Cases 1-2 of Claim~\ref{claim:possible colorings of L1,L2, r=3}}
Throughout this section we assume that $r = 3$. Our goal is to complete the proof of Lemma~\ref{high_min_deg} in Cases 1-2 of Claim~\ref{claim:possible colorings of L1,L2, r=3}. Case 1 is simple, while Case 2 requires considerable work.
\paragraph{Case 1:} 
By \eqref{stmt: unbalanced exists}, $H$ has an unbalanced $3$-coloring. Fix such a coloring $A_1,A_2,A_3$, and suppose without loss of generality that $|A_1|<|A_3|$. By Claim~\ref{claim:possible colorings of L1,L2, r=3}, $(s,t) = (2,2)$. 
By \eqref{eq:structured proof of 11}, we have
$$
\rho(|A_i|+|A_j|)= 2e_H(A_i\cup A_j, V(H)\backslash(A_i\cup A_j))+2e_H(A_i,A_j)=2e(H)
$$
for every pair $1 \leq i < j \leq 3$. In particular, $\rho \neq 0$.
So $|A_i| + |A_j| = 2e(H)/\rho$ for all $1 \leq i < j \leq 3$. But this implies that
$
|A_1| + |A_2| = |A_2| + |A_3|,
$
in contradiction to $|A_1|<|A_3|$. This completes Case 1.

\paragraph{Case 2:} $f_R(x_1y_1) = f_R(x_2y_2)$. 
By Claim~\ref{claim:possible colorings of L1,L2, r=3}, we have
$t = 0$. Then $s \neq 0$, because $(s,t) \neq (0,0)$. By normalizing, we get that $H$ is $(1,0)$-structured (recall that if $H$ is $(s,t)$ structured then it is also $(\alpha \cdot s, \alpha \cdot t)$-structured). With a slight abuse of notation, we will set $(s,t) = (1,0)$ for the rest of Case 2.
By \eqref{eq:structured proof of 11}, 
\begin{equation}\label{eq:rho opposite edges same general}
    \rho(|A_i|+|A_j|) = 
    e_H(A_i\cup A_j, V(H)\backslash(A_i\cup A_j))= 
    e(H) - e_H(A_i,A_j)
\end{equation}
holds for every $3$-coloring $A_1,A_2,A_3$ of $H$ and for all $1 \leq i < j \leq 3$. Summing this over all $i,j$, we get
$$2\rho|H| = \rho \sum_{1 \leq i < j \leq 3}(|A_i|+|A_j|) = 3e(H) - \sum_{1 \leq i < j \leq 3}e_H(A_1,A_2) = 2e(H).$$ 
Hence,
\begin{equation}\label{eq:e(H) = rho|H|}
    e(H) = \rho|H|.
\end{equation}
This implies that $\rho > 0$. 
Combining \eqref{eq:e(H) = rho|H|} with \eqref{eq:rho opposite edges same general}, we get
\begin{equation}\label{eq:rho opposite edges same}
e_H(A_i,A_j) = \rho(|H| - |A_i| - |A_j|).
\end{equation}
\noindent
Next, we show that a pair of intersecting triangles in $R$ cannot have opposite edges of different color. 
\begin{claim}\label{claim:No triangles with different opposite edges}
    Let $L'_1,L'_2$ be two distinct triangles in $R$ with $|L'_1 \cap L'_2| \geq 1$. Let $e_1,e_2$ be edges with $L'_1 \setminus L'_2 \subseteq e_1$ and $L'_2 \setminus L'_1 \subseteq e_2$. Then $f_R(e_1) = f_R(e_2)$.
\end{claim}
\begin{proof}
Suppose first $|L'_1 \cap L'_2| = 1$, so that $e_1 \cap e_2 = \emptyset$. Assume, for the sake of contradiction, that $f_R(e_1) \neq f_R(e_2)$.
By assumption, $(L'_1 \cup L'_2, f_R)$ is not a template for $H$. Then, by Lemma~\ref{template_example}, $H$ is $(s',t')$-structured with $t' = f_R(e_1) - f_R(e_2) \neq 0$ (the value of $s'$ will not be important). By normalizing, $H$ is $(s'',1)$-structured for some $s''$. 
In addition, we saw above that $H$ is $(1,0)$-structured. So by Lemma~\ref{no_two}, every $3$-coloring of $H$ is balanced, a contradiction.

Suppose now $|L'_1 \cap L'_2| = 2$, and write $e_1 = vw_1$, $e_2 = vw_2$, $L'_1 = \{u,v,w_1\}$, $L'_2 = \{u,v,w_2\}$. Recall that $\delta(R)/|R| \geq 1/2 + \eta/2$ by \eqref{eq:delta>(r-2)/(r-1)}. This implies that for each $i=1,2$, there exists a common neighbour $z_i$ of $u,w_i$, $z_i \neq v$. By the previous case for the triangles $\{u,v,w_2\}$ and $\{u,w_1,z_1\}$, it holds that $f_R(w_1z_1) = f_R(vw_2) = f_R(e_2)$. Similarly, $f_R(w_2z_2) = f_R(vw_1) = f_R(e_1)$. Finally, by considering the triangles $\{u,w_1,z_1\}$ and $\{u,w_2,z_2\}$, we get that $f_R(w_1z_1) = f_R(w_2z_2)$, so $f_R(e_1) = f_R(e_2)$.
\end{proof}
Claim~\ref{claim:No triangles with different opposite edges} means that for every vertex $v\in V(R)$, all edges opposite $v$ in triangles containing $v$ have the same color.
Let $W_+\subseteq V(R)$ be the vertices for which these edges have color $1$, and let $W_- = V(R)\backslash W_+$ be the vertices for which these edges have color $-1$. 
We now consider two cases according to the size of the sets $W_+,W_-$.
% \begin{claim}
% If $|V_+|\geq (1/2+\eta/2)|R|$ or $|V_-|\geq (1/2+\eta/2)|R|$ then there exists a perfect $H$-factor in $G'$ with high discrepancy.
% \end{claim}
\paragraph{Case 2(a):} $|W_+|\geq (1/2+\eta/2)|R|$ or $|W_-|\geq (1/2+\eta/2)|R|$. Without loss of generality, assume that $|W_+|\geq (1/2+\eta/2)|R|$. 
Recall the definition of the graph $H^*$ (see Definition~\ref{def:H^*}). In particular, $H^*$ is a complete $3$-partite graph and has a perfect $H$-factor.
Recall the definition of $V_0$ (i.e., $V_0$ is the exceptional set in the regular partition of $G'$). 
Finally, recall that for a vertex-set $W \subseteq V(R)$, $V_W := \bigcup_{w \in W}V_w \subseteq V(G')$ denotes the union of clusters corresponding to the vertices of $W$. 
Put $V_+ = V_{W_+}$ and $V_- = V_{W_-}$. 
In the following key claim, we show that if $J$ is an $H^*$-copy in $G'$ disjoint from $V_0$, then for every perfect $H$-factor $F$ of $J$, the discrepancy of $F$ can be expressed in terms of the intersection of $V(J)$ with $V_+$ and $V_-$. We will then use this claim to conclude the proof in Case 2(a). 
\begin{claim}\label{claim:f(F_j)}
    Let $J$ be an $H^*$-copy in $G'$ with $V(J) \cap V_0 = \emptyset$. Then for every perfect $H$-factor $F$ of $J$,
    \begin{equation}\label{eq:f(F_J)}
        f(F) =\rho\left( \left|V(F_J)\cap V_+\right|-\left|V(F_J)\cap V_-\right|\right).
    \end{equation}
\end{claim}
\begin{proof} 
Recall that for a vertex $v \in V(G') \setminus V_0$, we use $V^R_v \in V(R)$ to denote the cluster of the regular partition containing $v$. The assumption $V(J) \cap V_0 = \emptyset$ means that $V^R_v$ is well-defined for every $v \in V(J)$.  

Let $B_1,B_2,B_3$ be the parts of $J$. 
We claim that all bipartite graphs $(B_i,B_j)$, $1 \leq i < j \leq 3$, are monochromatic. Indeed, assume by contradiction (and without loss of generality) that there exist $b_1,b_1'\in B_1$ and $b_2\in B_2$ with $f(b_1b_2)\neq f(b'_1b_2)$. Fix an arbitrary $b_3\in B_3$. 
Consider the clusters $V_{b_1}^R,V_{b'_1}^R,V_{b_2}^R,V_{b_3}^R \in V(R)$ corresponding to the vertices $b_1,b'_1,b_2,b_3$, respectively. 
Then $V_{b_1}^R,V_{b_2}^R,V_{b_3}^R$ and $V_{b'_1}^R,V_{b_2}^R,V_{b_3}^R$ are triangles in $R$, and the edges $V_{b_1}^RV_{b_2}^R$ and $V_{b'_1}^RV_{b_2}^R$ opposite $V_{b_3}^R$ in these triangles have different colors (as $f(b_1b_2)\neq f(b'_1b_2)$). 
This is a contradiction, proving our claim that all bipartite graphs $(B_i,B_j)$ are monochromatic. This means that $(J,f_R)$ is a blowup of $(K_3,c)$ for some coloring $c$ of $K_3$. 

Next, we claim that 
$B_i \subseteq V_+$ or $B_i \subseteq V_-$ for each $i=1,2,3$. Let us prove this for $i = 1$. 
Fix arbitrary vertices $b_2\in B_2$ and $b_3\in B_3$. As $J$ is a complete tripartite graph, every vertex $b_1\in B_1$ forms a triangle with $b_2,b_3$. Now, by definition, if $f_R(V_{b_2}^RV_{b_3}^R)=1$ then $V_{b_1}^R \in W_+$ and hence $b_1 \in V_+$ for all $b_1 \in B_1$, and if $f_R(V_{b_2}^RV_{b_3}^R)=-1$ then $V_{b_1}^R \in W_-$ and hence $b_1 \in V_-$ for all $b_1 \in B_1$. This proves our claim. 

Let $H'$ be a copy of $H$ in $F$, and let $A_i := V(H') \cap B_i$, so that $A_1,A_2,A_3$ is a $3$-coloring of $H'$. Observe that if $A_3 \subseteq V_+$ (resp. $A_3 \subseteq V_-$) then all edges between $A_1,A_2$ have color $1$ (resp. $-1$). Moreover, 
$e_{H'}(A_1,A_2) = \rho|A_3|$ by \eqref{eq:rho opposite edges same}. By using this and the analogous statement for the pairs $A_1,A_3$ and $A_2,A_3$, we get the following:
$$
f(H') = \sum_{i \in [3] : \; A_i \subseteq V_+}\rho|A_i| - 
\sum_{i \in [3] : \; A_i \subseteq V_-}\rho|A_i| = 
\rho\left( \left|V(H')\cap V_+\right|-\left|V(H')\cap V_-\right|\right).
$$
Summing the above over all $H$-copies $H'$ in $F$ gives \eqref{eq:f(F_J)}.
\end{proof} 
We now conclude Case 2(a). 
By Lemma~\ref{lem:H* factor}, $G'$ has a perfect $H^*$-factor $F^*$.
Recall that $e(H)=\rho |H|$ (by \eqref{eq:e(H) = rho|H|}) and therefore, any perfect $H$-factor of $H^*$ has exactly $e(H) \cdot |H^*|/|H| = \rho |H^*|$ edges. Let $F$ be a perfect $H$-factor of $G'$, obtained by taking a perfect $H$-factor $F_J$ of each $H^*$-copy $J$ in $F^*$. 
Our assumption $|W_+|\geq (1/2+\eta/2)|R|$ means that 
$|V_+|\geq (1/2+\eta/2)(n-|V_0|)$ and hence $|V_-|\leq (1/2-\eta/2)(n-|V_0|)$. 
There are at most $|V_0|$ copies of $H^*$ in $F^*$ which intersect $V_0$, and for each such copy $J$, $F_J$ has $\rho |H^*|$ edges. For all other $H^*$-copies $J \in F^*$, we use \eqref{eq:f(F_J)} with $F = F_J$.
Combining all of this, we get: 
\begin{equation}\label{eq:f(F) r=3}
\begin{split}
f(F) &= \sum_{J \in F^*}f(F_J) \geq 
\rho \sum_{J \in F^* : V(J) \cap V_0 = \emptyset}\left( \left|V(F_J)\cap V_+\right|-\left|V(F_J)\cap V_-\right|\right) - \rho|H^*| |V_0| \\ &\geq
\rho\left( |V_+| - |H^*||V_0| - |V_-| \right) - \rho|H^*| |V_0| \geq 
\rho \eta (n - |V_0|) - 2\rho|H^*||V_0| \geq 
\rho \eta n/2 \geq
\gamma n.
\end{split}
\end{equation}
The quantity $|V_+| - |H^*||V_0|$ in the second line of \eqref{eq:f(F) r=3} is a lower bound in $\sum_{J}|V(J) \cap V_+|$ over all $J$ with $V(J) \cap V_0 = \emptyset$.
The penultimate inequality in \eqref{eq:f(F) r=3} holds because $|V_0| \leq \varepsilon n$, $\varepsilon \ll \frac{1}{|H|},\eta$, and $H^*$ depends only on $H$ and $\eta$. Finally, the last inequality in \eqref{eq:f(F) r=3} holds because $\gamma \ll \frac{1}{|H|},\eta$, and $\rho > 0$ depends only on $H$. By \eqref{eq:f(F) r=3}, $F$ has high discrepancy. This concludes Case 2(a).

\paragraph{Case 2(b)}
$|W_+|,|W_-|< (1/2+\eta/2)|R|$. We will show that this case is impossible. First, we prove some structural properties of $R$ and $f_R$. 
\begin{claim}\label{|S_+|}
$R[W_+]$ and $R[W_-]$ are triangle-free. 
\end{claim}
\begin{proof} 
We only prove the assertion for $R[W_+]$; the proof for $R[W_-]$ is analogous. 
Let $S\subseteq W_+$ be the largest clique in $W_+$. We need to show that
$|S| \leq 2$.
Suppose by contradiction that $|S|\geq3$. Then, all the edges in $R[S]$ must have color $1$, because each such edge is contained in a triangle in $R[S]$, hence it is the edge opposite to a vertex from $W_+$ in a triangle. By \eqref{eq:delta>(r-2)/(r-1)},
\begin{equation}\label{eq:sum of degrees in S_+}
    \sum_{s\in S}|N_R(s)|\geq (1/2+\eta/2)|R||S|.
\end{equation}
Note that no two vertices in $S$ share a neighbor in $W_-$, because else we get a triangle containing a vertex from $W_-$ whose opposite edge
has color $1$, contradicting the definition of $W_-$. Also, each vertex in $W_+$ can only be connected to at most $|S|-1$ vertices in $S$, as otherwise we add this vertex to $S$ and get a larger complete graph inside $W_+$, contradicting the maximality of $S$. 
It follows that
$$
\sum_{s\in S}|N_R(s)|\leq |W_-|+(|S|-1)|W_+|<(1/2+\eta/2)|R||S|,
$$
using $|W_+|,|W_-|< (1/2+\eta/2)|R|$.
This contradicts \eqref{eq:sum of degrees in S_+}.
\end{proof} 
\begin{claim}\label{claim:properties of W_+,W_-}
        Every vertex in $W_+$ has an edge of color $1$ to some vertex in $W_-$, and every vertex in $W_-$ has an edge of color $-1$ to some vertex of $W_+$.
\end{claim}
\begin{proof}
     We only prove the assertion for vertices $W_+$; the other case is symmetrical. So let $v \in W_+$. Since $\delta(R) \geq (1/2+\eta/2)|R| > |W_+|,|W_-|$ by \eqref{eq:delta>(r-2)/(r-1)}, $v$ must have a neighbour $u$ in $W_+$. 
     Since $\delta(R) > |R|/2$, there is a common neighbour $w$ of $u,v$. Then $w \in W_-$ because $R[W_+]$ is triangle-free by Claim~\ref{|S_+|}. 
     The edge $vw$ has color $1$ because it is the edge opposite $u \in W_+$ in the triangle $u,v,w$. So $v$ indeed has an edge of color $1$ to $W_-$. 
\end{proof}

Claim~\ref{claim:properties of W_+,W_-} implies that there exists a cycle $C$ in the bipartite graph $R[W_+,W_-]$ with edges of alternating color. To see this, consider an orientation of the edges of $R[W_+,W_-]$ where all edges of color $1$ are oriented from $W_+$ to $W_-$ and all edges of color $-1$ are oriented from $W_-$ to $W_+$. In this orientation, every vertex has outdegree at least $1$ (by Claim~\ref{claim:properties of W_+,W_-}), and therefore there exists a directed cycle, which corresponds to a cycle whose edges alternate in color.

By averaging, there is a vertex $v \in V(R)$ which is connected to at least $|C| \cdot \delta(R)/|R| > |C|/2$ vertices of $C$. 
Note that the edges of $C$ of color $1$ as well as the edges of color $-1$ form a perfect matching of $C$. Therefore, there exists an edge of color $1$ in $C$ such that $v$ is connected to both of its endpoints and similarly, an edge of color $-1$ in $C$ such that $v$ is connected to both of its endpoints. 
But then $v$ is contained in a triangle with opposite edge of color $1$ and in a triangle with opposite edge of color $-1$, a contradiction. This completes Case 2(b) and hence Case 2 altogether. 

\subsubsection{Minimum degree at least $\frac{3r-5}{3r-2}$ and the structure of $R$}
We continue with the proof of Lemma~\ref{high_min_deg}. From now on, we will assume that if $r=3$ then Cases 1-2 of Claim~\ref{claim:possible colorings of L1,L2, r=3} do not hold. The following key claim provides a lower bound on $\delta(R)/|R|$.
\begin{claim}\label{claim:(3r-5)/(3r-2)}
It holds that
\begin{equation}\label{eq: delta(R)>3r-5/3r-2}
    \delta(R)/|R| \geq \frac{3r-5}{3r-2} +\eta/2.
\end{equation}
\end{claim}
\begin{proof}
    By \eqref{define s} and \eqref{eq:multiple of 2(r-2)}, we have $s \in \{-2,-1,0,1,2\}$.
    Note also that $t \in \{-2,0,2\}$ by \eqref{define t}, 
    and recall that $(s,t) \neq (0,0)$. 
    We now normalize the parameters $s,t$ in order to apply Lemma~\ref{lower_delta}. Recall that if $H$ is $(s,t)$-structured then it is also $(\alpha \cdot s, \alpha \cdot t)$-structured for every $\alpha \in \mathbb{R}$. If $s = 0$ then by normalizing, we may assume that $t = 1$. And if $s \neq 0$, then by dividing $(s,t)$ by $s$, we may assume that $s = 1$ and $t \in \{-2,-1,0,1,2\}$. 
    In any case, the normalized parameters fit the assumption of Lemma~\ref{lower_delta}. Hence, 
    if $r \geq 6$, then Lemma~\ref{lower_delta} gives
\begin{equation*}
    \max\{\delta_0(H),1-1/\chi^*(H)\}\geq\frac{3r-5}{3r-2}.
\end{equation*}
So \eqref{eq: delta(R)>3r-5/3r-2} follows from \eqref{eq:delta(R) proof of Lemma 11}. 

It remains to prove \eqref{eq: delta(R)>3r-5/3r-2} when $r = 3$. 
Note that $\frac{3r-5}{3r-2} = \frac{4}{7}$ for $r = 3$. 
By assumption, we are in Case 3 or 4 of Claim~\ref{claim:possible colorings of L1,L2, r=3}. Suppose first that we are in Case 3, so $(L_1 \cup L_2, f_R)$ is a butterfly. By assumption, $(L_1 \cup L_2, f_R)$ is not a template for $H$. Hence, there exists a butterfly which is not a template for $H$. Now \eqref{eq: delta(R)>3r-5/3r-2} holds by \eqref{stmt:4/7}. 
Finally, suppose that we are in Case 4 of Claim~\ref{claim:possible colorings of L1,L2, r=3}. Then $H$ is $(1,2)$-structured. Now, by Lemma~\ref{lower_delta_r=3}, $\max\{\delta_0(H),1-1/\chi^*(H)\} \geq \frac{5}{8} > \frac{4}{7}$, so again \eqref{eq: delta(R)>3r-5/3r-2} holds by \eqref{eq:delta(R) proof of Lemma 11}.
\end{proof}

The proof proceeds with a sequence of claims that slowly uncovers the structure of $R$, deducing the color of various edges from the assumption that $R$ does not have templates for $H$. For example, in the case $r \neq 3$, we will eventually show that $R$ is strongly tilted towards one of the colors, which will allow us to find a perfect $H$-factor with high discrepancy. 
The bound \eqref{eq: delta(R)>3r-5/3r-2} will be crucial. 

Recall that $V(L_1) = V\cup\{x_1,y_1\}$ and $V(L_2) = V\cup\{x_2,y_2\}$. Let $X_1\subseteq V(R)$ be the common neighborhood of $y_1$ and $V$, $Y_1\subseteq V(R)$ the common neighborhood of $x_1$ and $V$, $X_2\subseteq V(R)$ the common neighborhood of $y_2$ and $V$, and $Y_2\subseteq V(R)$ the common neighborhood of $x_2$ and $V$. Note that $x_i \in X_i, y_i \in Y_i$ for $i=1,2$.
Using \eqref{eq: delta(R)>3r-5/3r-2},
we get 
\begin{equation}\label{eq:X1,Y1,X2,Y2}
|X_1|, |Y_1|, |X_2|, |Y_2|\geq (r-1)\delta(R)-(r-2)|R| \geq \left(\frac{1}{3r-2}+\eta\right)|R|.
\end{equation}
\noindent
We now establish some properties of the sets $X_1,Y_1,X_2,Y_2$. 
\begin{claim}\label{clm: edges have same color}
For every edge $e\in R[X_1\cup Y_1]$ it holds that $f_R(e)=f_R(x_1y_1)$, and for every edge $e\in R[X_2\cup Y_2]$ it holds that $f_R(e)=f_R(x_2y_2)$.
% or there exists a perfect $H$-factor in $G'$ with high discrepancy.
\end{claim}
\begin{proof}
We only prove the claim for $X_1 \cup Y_1$; the proof for $X_2 \cup Y_2$ is analogous. 
Let us assume by contradiction that there exist vertices $u,v\in X_1\cup Y_1$ such that $uv \in E(R)$ and $f_R(uv) = -f_R(x_1y_1)$. By definition, $u,v$ are adjacent to all vertices in $V$, hence 
$M_1 := V\cup\{u,v\}$ is an $r$-clique in $R$. Without loss of generality, let us assume that $u\in X_1$. Then $u$ is adjacent to $y_1$, so $M_2 := V\cup\{u,y_1\}$ is also an $r$-clique.
Now, $M_1,M_2,L_1$
is a sequence of $r$-cliques with $|M_1 \cap M_2| = |M_2 \cap L_1| = r-1$, so by Claim~\ref{claim:K_r_chain},
$$
f_R(M_1) = f_R(L_1).
$$
Now, consider the two $r$-cliques $M_1,L_2$. We have $V \subseteq M_1 \cap L_2$, so $|M_1 \cap L_2| \geq r-2$. Also, $f_R(M_1) \neq f_R(L_2)$ because $f_R(L_1) \neq f_R(L_2)$.
By Claim~\ref{claim:K_r_chain} we know that $|M_1 \cap L_2| \neq r-1$, 
so $|M_1 \cap L_2| = r-2$. As $(M_1 \cup L_2,f_R)$ is not a template for $H$, we have by Lemma~\ref{template_example} that $H$ is $(s',t')$-structured with 
$$
s' = \frac{f_R(M_1) - f_R(uv) - f_R(L_2) + f_R(x_2y_2)}{2(r-2)}
$$
and
$$
t' = f_R(uv) - f_R(x_2y_2).
$$
As before, $(s',t') \neq (0,0)$ because $f_R(M_1) \neq f_R(L_2)$.
As $f_R(uv)=-f_R(x_1y_1)$, exactly one of $t,t'$ is zero.
Also, for the pair among $(s,t),(s',t')$ where the second coordinate is not zero, we can normalize this coordinate to be $1$. 
Now $H$ satisfies the conditions of Lemma~\ref{no_two}. Hence, by Lemma~\ref{no_two}, all $r$-colorings of $H$ are balanced, contradicting \eqref{stmt: unbalanced exists}.
\end{proof}
\begin{claim}\label{clm: no edges between sides}
The sets $X_1 \cup Y_1$ and $X_2 \cup Y_2$ are disjoint, and $R$ has no edges between $X_1\cup Y_1$ and $X_2 \cup Y_2$.
\end{claim}
\begin{proof}
We first prove the second part of the claim. So assume that there exist $u\in X_1\cup Y_1$ and $v\in X_2\cup Y_2$ such that $uv\in R$. Without loss of generality, let us assume that $u\in X_1$ and $v\in X_2$ (the other three cases are similar). 
Then, by definition, $u,v$ are adjacent to all vertices of $V$, $u$ is adjacent to $y_1$, and $v$ is adjacent to $y_2$. 
So $L_1,V\cup \nolinebreak \{u,y_1\},V \cup \nolinebreak \{u,v\}, V\cup\{v,y_2\},L_2$ is a sequence of copies of $K_r$ with every two consecutive copies sharing at least $r-1$ vertices. This contradicts Claim~\ref{claim:K_r_chain} as $f_R(L_1) \neq f_R(L_2)$.

Now assume by contradiction that there 
is $u \in (X_1 \cup Y_1) \cap (X_2 \cup Y_2)$. 
Without loss of generality, suppose that $u \in X_1 \cap X_2$. In particular, $u$ is adjacent to $y_2$. As $y_2 \in Y_2$, the edge $uy_2$ goes between $X_1$ and $Y_2$. This is a contradiction, as we already showed that there are no edges between $X_1\cup Y_1$ and $X_2 \cup Y_2$. 
\end{proof}
\begin{claim}
Either both $X_1$ and $Y_1$ are not independent sets, or both $X_2$ and $Y_2$ are not independent sets.
\end{claim}
\begin{proof}
Let us assume that one of $X_2,Y_2$ is independent, and show that then $X_1,Y_1$ are not independent. So suppose without loss of generality that $X_2$ is independent. Then $X_2$ and $Y_2$ are disjoint, because $x_2\in X_2$ is connected to all of $Y_2$.
    By Claim~\ref{clm: no edges between sides}, there are no edges between $X_1 \cup Y_1$ and $X_2 \cup Y_2$. Hence, if $X_1$ were also independent, then every vertex in $X_1$ would have degree at most 
$$
|R|-|X_1|-|X_2|-|Y_2|\leq \frac{3r-5}{3r-2}|R|,
$$
using \eqref{eq:X1,Y1,X2,Y2}. 
This
contradicts \eqref{eq: delta(R)>3r-5/3r-2}. By the same argument, $Y_1$ is also not independent.
\end{proof}
\noindent
Without loss of generality, let us assume that neither $X_1$ nor $Y_1$ is an independent set in $R$. For the rest of the proof, fix an edge $u_1v_1\in R[X_1]$. By definition, $u_1,v_1$ are adjacent to $y_1$ and all vertices of $V$. Hence, $M := V \cup \{u_1,v_1,y_1\}$ is a clique of size $r+1$ in $R$. We now show that $M$ is monochromatic.  
\begin{claim}\label{claim : clique monochromatic}
$R[M]$ is monochromatic with respect to $f_R$.
\end{claim}
\begin{proof}
For convenience, put $c = f_R(x_1y_1)$.
By Claim~\ref{clm: edges have same color} and as $x_1,u_1,v_1 \in X_1$, $y_1 \in Y_1$, we have 
\begin{equation}\label{eq:x1,y1,u1,v1}
    f_R(u_1v_1)=f_R(u_1y_1)=f_R(v_1y_1)=f_R(x_1y_1) = c.
\end{equation}
By assumption, $(M,f_R)$ is not a template for $H$. By Corollary~\ref{non_reg_clique}, this means that $R[M]^{c}$ is $d$-regular for some $d$. If $d=r$ then $R[M]$ is monochromatic, so let us assume, by contradiction, that $d < r$. 

Suppose first that $r = 3$. By \eqref{eq:x1,y1,u1,v1}, $u_1,v_1,y_1$ is a monochromatic triangle of color $c$. A regular graph on $4$ vertices containing a triangle must be a complete graph, so $R[M]$ is monochromatic, as required. 

From now on, assume that $r \geq 4$. 
Since $d < r$, $u_1$ has an edge of color $-c$ to some vertex $z \in M$. 
By \eqref{eq:x1,y1,u1,v1}, $z \in V$. Now consider the two $r$-cliques $M_1 := (V \setminus \{z\}) \cup \{y_1,u_1,v_1\}$ and $M_2 := V \cup \{u_1,v_1\}$. Then $|M_1 \cap M_2| = r-1$, so $f_R(M_1) = f_R(M_2)$ by Claim~\ref{claim:K_r_chain}.  
We will now apply Lemma~\ref{sharing_little_template} to $M_1,M_2$ with $e_1 = u_1y_1$ and $e_2 = u_1z$ (and with $(V \setminus \{z\}) \cup \{v_1\}$ in the role of $V$). 
Note that
$$
f_R(M_1) - f_R(e_1) - f_R(M_2) + f_R(e_2) = f_R(e_1) - f_R(e_2) \in \{-2,2\},
$$
using that $f_R(e_2) = -c = -f_R(e_1)$. As $r \geq 4$, it follows that 
$$
f_R(M_1) - f_R(e_1) - f_R(M_2) + f_R(e_2) \notin \{-4(r-2),-2(r-2),0, 2(r-2), 4(r-2)\}.
$$
Now Lemma~\ref{sharing_little_template} implies that $(M_1 \cup M_2, f_R)$ is a template for $H$, contradicting our assumption that $R$ has no such template. 
\end{proof}
By Claim~\ref{claim : clique monochromatic}, $M$ is monochromatic. As $L_1$ intersects the $r$-clique $V \cup \{u_1,y_1\} \subseteq M$ in $r-1$ vertices, we must have $f_R(L_1) = f_R(V\cup\{u_1,y_1\})$ by Claim~\ref{claim:K_r_chain}. Since $f_R(L_1)>0$, it follows that $L_1$ is monochromatic\footnote{It is worth noting that from this point on, we must have $r \leq 7$; namely, the proof is already complete for all $r \geq 8$. Indeed, since all edges in $V$ have color $1$, and $L_2$ has exactly $2r-3$ edges not contained in $V$, we have $0 > f_R(L_2) \geq \binom{r-2}{2} - (2r-3),$ which only holds if $r \leq 7$. So the remaining cases are $r \in \{3,6,7\}$.} in color $1$, and in particular $f_R(x_1y_1)=1$. 

Recall that $H$ is $(s,t)$-structured. Using that $L_1$ is monochromatic in color $1$, we now show that only few options for $s$ and $t$ are possible.
\begin{claim}\label{claim:(s,t)}
The following holds:
\begin{enumerate}
    \item Suppose that $r \neq 3$. Then $H$ is $(1,0)$-structured or $(1,1)$-structured.
    \item Suppose that $r = 3$. Then $f_R(L_2) = -1$ and $H$ is $(1,2)$-structured. 
\end{enumerate}
\end{claim}
\begin{proof}
We begin with the case $r \neq 3$, namely $r \geq 6$. 
If $f_R(x_2y_2) = 1$ then $t = 0$ by definition, recall \eqref{define t}. As $(s,t) \neq (0,0)$, we can normalize to get that $H$ is $(1,0)$-structured. 

Suppose now that $f_R(x_2y_2) = -1$. Then $t = 2$. Also, using that $f_R(L_1) = \binom{r}{2}$ and $f_R(L_2) < 0$, we get
\begin{equation}\label{different choices} 
f_R(L_1) - f_R(x_1y_1) - f_R(L_2) + f_R(x_2y_2) = \binom{r}{2} - f_R(L_2) - 2 \geq \binom{r}{2} - 1 > 2(r-2),
\end{equation}
where the last inequality holds for $r > 3$. Contrasting this with \eqref{eq:multiple of 2(r-2)}, we see that the LHS of \eqref{different choices} equals $4(r-2)$, which implies that $s = 2$ by \eqref{define s}. So $H$ is $(2,2)$-structured and hence $(1,1)$-structured. 

Now suppose that $r = 3$. By assumption, we are in Case 3 or 4 of Claim~\ref{claim:possible colorings of L1,L2, r=3}. Case 3 is impossible because $L_1$ is monochromatic, hence we are in Case 4. So $H$ is $(1,2)$-structured and $f_R(L_2) = -1$ (as $L_2$ is not monochromatic and $f_R(L_2) < 0$).
\end{proof}

\begin{claim}\label{claim:X2,Y2 disjoint}
    $X_2 \cap Y_2 = \emptyset$.
\end{claim}
\begin{proof}
    We will show that $X_2$ is an independent set (the same is true for $Y_2$). This will imply the claim because $x_2 \in X_2$ is adjacent to all vertices in $Y_2$ (by the definition of $Y_2$), so if there existed $u \in X_2 \cap Y_2$, then $x_2u$ would be an edge inside $X_2$, a contradiction. 

    So let us assume, by contradiction, that $R$ has an edge $u_2v_2$ with $u_2,v_2 \in X_2$.
    By the definition of $X_2$, $u_2,v_2$ are adjacent to $V \cup \{y_2\}$.
    
    Suppose first that $r = 3$. By Claim~\ref{clm: edges have same color}, the triangle $M_1:=\{u_2,v_2,y_2\}$ is monochromatic (with color $f_R(x_2y_2)$). Write $V=\{v\}$ (recall that $|V|=r-2$). So $M_2 = \{v,v_2,y_2\}$ is a triangle. Now, $M_1,M_2,L_2 = \{v,x_2,y_2\}$ is a sequence of triangles with $|M_1 \cap M_2| = |M_2 \cap L_2| = 2$. Also, $M_1$ is monochromatic, while $L_2$ is not monochromatic because $f_R(L_2) = -1$ by Claim~\ref{claim:(s,t)}.  
    This contradicts Claim~\ref{claim:K_r_chain}.

    Suppose now that $r \neq 3$, namely $r \geq 6$. Note that $M_0 := V \cup \{y_2,u,v\}$ is a copy of $K_{r+1}$ in $R$. By assumption, $(M_0,f_R)$ is not a template for $H$. Hence, by Corollary~\ref{non_reg_clique}, there is $d' \in \mathbb{N}$ such that $M_0^-$ is $d'$-regular. All edges inside $V$ have color $1$, so all edges of $M_0$ of color $-1$ must touch $\{y_2,u,v\}$. Considering the edges of color $-1$ touching $V$, we get $(r-2)d' = |V|d' \leq 3d'$. As $r \geq 6$, we get that $d' = 0$, i.e. $M_0$ is monochromatic in color $1$. In particular, $uv$ and all edges between $y_2$ and $V$ have color $1$. By Claim~\ref{clm: edges have same color}, $f_R(x_2y_2) = f_R(uv) = 1$. Hence, the only edges of $L_2$ that can have color $-1$ are edges between $x_2$ and $V$. The number of these edges is the same as the number of edges between $y_2$ and $V$, which all have color $1$. So $f_R(L_2) > 0$, a contradiction. This completes the proof of the claim
\end{proof}

We can now complete the proof in the case $r=3$. 
By Claim~\ref{claim:(s,t)}, $H$ is $(1,2)$-structured.
By Lemma~\ref{lower_delta_r=3}, we have $\max\{\delta_0(H),1-1/\chi^*(H)\}\geq 5/8$. Thus, by 
\eqref{eq: delta(R)>3r-5/3r-2},
\begin{equation}\label{eq:5/8 new delta(R)}
    \delta(R)/|R|\geq 5/8.
\end{equation}
This allows us to improve on \eqref{eq:X1,Y1,X2,Y2} as follows:
\begin{equation*}
    |X_2|,|Y_2|\geq 2\delta(R)-|R|\geq |R|/4.
\end{equation*}
By Claim~\ref{claim:X2,Y2 disjoint}, $X_2$ and $Y_2$ are disjoint and therefore, $|X_2\cup Y_2|\geq |R|/2$.
Now recall that $x_1\in X_1$ and by Claim~\ref{clm: no edges between sides}, $x_1$ is not adjacent to any vertex in $X_2\cup Y_2$. This contradicts with \eqref{eq:5/8 new delta(R)}, completing the proof of Lemma~\ref{high_min_deg} for $r=3$. For the rest of the proof, we assume that $r \neq 3$, namely $r \geq 6$. This case is handled in the following subsection.

\subsubsection{Concluding the proof: The case $r \geq 6$}
Recall that $u_1v_1$ is an edge of $R$ with $u_1,v_1 \in X_1$, so that $u_1,v_1$ are adjacent to all vertices in $V \cup \{y_1\}$. 
Let $N\subseteq V(R)$ be the common neighborhood of $u_1$ and $v_1$, and note that $V \subseteq N$.
By Claim~\ref{clm: no edges between sides}, $u_1$ and $v_1$ are not adjacent to any vertex in $X_2\cup Y_2$. Using that $X_2$ and $Y_2$ are disjoint (by Claim~\ref{claim:X2,Y2 disjoint}), 
\begin{equation}\label{eq:size of N}
|N|\geq 2\delta(R) - (|R| - |X_2| - |Y_2|) \geq \left(\frac{3r-6}{3r-2}+\eta\right)|R|,
\end{equation}
where the last inequality uses \eqref{eq: delta(R)>3r-5/3r-2} and \eqref{eq:X1,Y1,X2,Y2}. 
As $\delta(R) > \frac{3r-5}{3r-2}|R|$, each vertex in $R$ is adjacent to all but at most $\frac{3}{3r-2}|R|$ of the vertices, and hence
\begin{equation}\label{eq:minimum degree of N}
    \delta(R[N])\geq |N|-\frac{3}{3r-2}|R| > \frac{r-3}{r-2}|N|.
\end{equation}
The minimum degree of $R[N]$ implies the following:
\begin{align}\label{stmt: cliques are contained in larger cliqeus}
    \text{For each $k<r-1$, every copy of $K_k$ in $R[N]$ is contained in a copy of $K_{r-1}$ in $R[N]$.}
\end{align} 
\begin{claim}\label{clm: all edges have same color}
Every edge inside $N\cup\{u_1,v_1\}$ has color $1$.
\end{claim}
\begin{proof}
Assume by contradiction that there is an edge $e$ of color $-1$ inside $N\cup\{u_1,v_1\}$. By \eqref{stmt: cliques are contained in larger cliqeus}, there is an $(r-2)$-clique $L \subseteq N$ which contains $e \setminus \{u_1,v_1\}$. Using that $u_1,v_1$ are adjacent to all vertices in $N$, we see that $L \cup \{u_1,v_1\}$ is an $r$-clique containing $e$ and $u_1,v_1$. As $f_R(e) = -1$, we have $f_R(L \cup \{u_1,v_1\}) < \binom{r}{2}$. 

Now consider $L$ and $V$, which are both cliques of size $r-2$ contained in $N$. By Item 1 of Lemma~\ref{connectivity} with $k=r-2$ and $J = R[N]$, using \eqref{eq:minimum degree of N}, there is a sequence $M_1,M_2,\dots,M_\ell$ of $(r-2)$-cliques inside $R[N]$, such that $M_1 = L$, $M_\ell = V$, and $M_{i-1},M_i$ share at least $r-3$ vertices for all $1 \leq i < \ell$. Let $M'_i := M_i \cup \{u_1,v_1\}$. Then $M'_i$ is an $r$-clique in $R$, and $M'_{i-1},M'_i$ share at least $r-1$ vertices for all $1 \leq i < \ell$. Also, as $V \cup \{u_1,v_1\}$ is monochromatic in color $1$, we have 
$$
f_R(M'_1) = f_R(L \cup \{u_1,v_1\}) \neq \binom{r}{2} = 
f_R(V \cup \{u_1,v_1\}) = f_R(M'_{\ell}).
$$
This contradicts Claim~\ref{claim:K_r_chain}.
\end{proof}
Let $W= V(R)\backslash N\cup\{u_1,v_1\}$. By Claim~\ref{clm: all edges have same color}, all edges of color $-1$ in $R$ are incident to $W$. Also, by \eqref{eq:size of N}, we have 
\begin{equation}\label{eq:|W|}
|W|\leq \left(\frac{4}{3r-2}-\eta\right)|R|\leq \left(\frac{1}{4}-\eta\right)|R|,
\end{equation}
using $r \geq 6$.
In the following claim we derive some properties of $r$-cliques which intersect $W$ in only one or two vertices.
\begin{claim}\label{clm: need more than one}
The following holds:
\begin{enumerate}
    \item Let $L$ be a copy of $K_r$ in $R$ which has exactly one vertex in $W$. Then $L$ is monochromatic in color $1$. 
    \item Suppose that $H$ is $(1,0)$-structured. 
    Let $L$ be a copy of $K_r$ in $R$ which has exactly two vertices $w_1,w_2$ in $W$. Then $f_R(w_1w_2) = 1$.
\end{enumerate}

\end{claim}
\begin{proof}
We begin with the first item. 
Let $w$ be the unique vertex in $L \cap W$. We claim that there is an $r$-clique $L'$ in $R$ with $L' \subseteq N \cup \{u_1,v_1\}$ and $|L \cap L'| = r-1$. If $u_1 \notin L$ then $L' = (L \setminus \{w\}) \cup \{u_1\}$ is such an $r$-clique (here we use the fact that $u_1$ is adjacent to all vertices in $V(R) \setminus (W \cup \{u_1\})$, and hence to all vertices of $L \setminus \{w\}$). So assume that $u_1 \in L$. By the same argument, we may assume that $v_1 \in L$. Then $L \setminus \{u_1,v_1,w\}$ is a clique of size $r-3$ contained in $N$. By \eqref{stmt: cliques are contained in larger cliqeus}, there is a clique $M'$ of size $r-2$ with $M' \subseteq N$ and $L \setminus \{u_1,v_1,w\} \subseteq M'$. Now $L' = M' \cup \{u_1,v_1\}$ satisfies our requirements. As $|L \cap L'| = r-1$, we have $f_R(L) = f_R(L')$ by Claim~\ref{claim:K_r_chain}. By Claim~\ref{clm: all edges have same color}, $L'$ is monochromatic in color $1$, as $L' \subseteq N \cup \{u_1,v_1\}$. So $L$ is also monochromatic in color $1$. 

We now prove the second item. Put $M' := L \setminus \{w_1,w_2\}$; so $M'$ is a clique of size $r-2$ and $M' \subseteq N \cup \{u_1,v_1\}$. By \eqref{stmt: cliques are contained in larger cliqeus}, there is an $(r-2)$-clique $M'' \subseteq N$ with $M' \setminus \{u_1,v_1\} \subseteq M''$. Now, $L' := M'' \cup \{u_1,v_1\}$ is an $r$-clique with $|L \cap L'| = r-2$, $L \setminus L' = \{w_1,w_2\}$ and
$L' \subseteq N \cup \{u_1,v_1\}$. 
In particular, the edge $e := L' \setminus L$ has color $1$.
Suppose by contradiction that $f_R(w_1w_2) = -1$.
By assumption, $(L' \cup L, f_R)$ is not a template for $H$. By Lemma~\ref{template_example}, $H$ is $(s',t')$-structured with $t' = f_R(e) - f_R(w_1w_2) = 2$ (the value of $s'$ will not be important). 
By normalizing, $H$ is $(\frac{s'}{2},1)$-structured. 
Additionally, $H$ is $(1,0)$-structured by assumption. Now, by Lemma~\ref{no_two}, $H$ has only balanced $r$-colorings, a contradiction to \eqref{stmt: unbalanced exists}.
\end{proof}
Recall the definition of the graph $H^*$ (see Definition~\ref{def:H^*}). In particular, $H^*$ is a complete $r$-partite graph and has a perfect $H$-factor. 
 Recall the definition of $V_0$ in Section~\ref{sec_app_regularity} (namely, $V_0$ is the exceptional set given by the regularity lemma). 
The following key claim shows that if an $H^*$-copy $J$ in $G'$ does not intersect $V_0$, then every perfect $H$-factor of $J$ has only few edges of color $-1$. Using this claim, we then easily complete the proof of the lemma. Recall that $H$ is $(1,0)$- or $(1,1)$-structured by Item 1 of Claim~\ref{claim:(s,t)}. Let $\rho'$ be the corresponding parameter (as in Definition~\ref{structured}), and note that $\rho' > 0$.
Recall that $V_W = \bigcup_{w \in W}V_w \subseteq V(G')$ denotes the union of clusters which correspond to the vertices in $W \subseteq V(R)$. 

\begin{claim}\label{claim:Lemma 13.3 main}
Let $J$ be an $H^*$-copy in $G'$ with $V(J) \cap V_0 = \emptyset$, and let $A_1,\dots,A_r$ be the parts of $J$. Let $I$ be the set of indices $i \in [r]$ such that $A_i\subseteq V_W$. Then for every perfect $H$-factor $F$ of $J$,
\begin{equation}\label{eq:e(F^-) A1,..,Ar}
    e(F^-)\leq \rho'\sum_{i\in I}|A_i| \leq \rho'|V(J) \cap V_W|.
\end{equation}
\end{claim}
\begin{proof}
The right inequality in \eqref{eq:e(F^-) A1,..,Ar} is immediate from the definitions. We prove the left inequality. 
For each $i \in [r]$, choose $a_i \in A_i$ such that 
$a_i \in A_i \setminus V_W$ if $i \notin I$, and else $a_i$ is arbitrary. Set $L = \{a_1,\dots,a_r\}$, so $L \subseteq J$ is an $r$-clique in $G'$.  
Let $L^R \subseteq R$ be the corresponding $r$-clique in $R$, namely $L^R = \{V_{a_1}^R,\dots,V_{a_r}^R\}$. The assumption $V(J) \cap V_0 = \emptyset$ means that the cluster $V_a \in V(R)$ is well-defined for every $a \in V(J)$.

First, suppose that there exist $1\leq i \neq j\leq r$ and $u,v\in A_i$ and $w\in A_j$ so that $f(uw)\neq f(vw)$. 
Without loss of generality, let us assume that $i=1,j=2$. 
Recall that we assume that $H$ is non-uniform. Hence, by Claim~\ref{claim:non-uniform}, there exists an $r$-coloring $B_1,B_2,\dots, B_r$ of $H$ such that $e_H(B_1,B_2)\neq e_H(B_1,B_3)$ and thus, there exists $b\in B_1$ such that $d_H(b,B_2)\neq d_H(b,B_3)$.
Now consider the two $r$-cliques $M_1 := \{V_u^R,V_w^R,V_{a_3}^R,\dots,V_{a_r}^R\}$ and $M_2 := \{V_v^R,V_w^R,V_{a_3}^R,\dots,V_{a_r}^R\}$ in $R$. We have $|M_1 \cap M_2| = r-1$, so $f_R(M_1) = f_R(M_2)$ by Claim~\ref{claim:K_r_chain}. Also, $f_R(V_u^RV_w^R) \neq f_R(V_v^RV_w^R)$. Hence, by Lemma~\ref{like blowup2}, $(M_1 \cup M_2, f_R)$ is a template for $H$, contradicting our assumption that $R$ contains no such template. 

So from now on, we assume that for all $1 \leq i \neq j \leq r$ and $u,v \in A_i, w \in A_j$, it holds that $f(uw)\neq f(vw)$. This means that all bipartite graphs $(A_i,A_j)$ are monochromatic. In other words, $(J,f)$ is a blowup of $(L^R,f_R)$. 
 As all the edges in $R$ of color $-1$ are incident to $W$, all the edges of color $-1$ in $J$ must be incident to $\bigcup_{i\in I}A_i$. Indeed, if $i,j \notin I$ then $A_i,A_j \not\subseteq V_W$, so there must be an edge of color $1$ between $A_i,A_j$. But as $(A_i,A_j)$ is monochromatic, all edges between $A_i,A_j$ have color $1$. 
 
 By Claim~\ref{clm: need more than one}, if $|I|\leq 1$ then $e(F^-) = 0$ and the claim holds trivially. Hence, we assume that $|I|\geq 2$. Let us now distinguish two cases. 
 Suppose first that $H$ is $(1,1)$-structured (with parameter $\rho'$). 
 Then $F$ is also $(1,1)$-structured (with parameter $\rho'$), as $F$ is an $H$-factor. Hence (recall Definition~\ref{structured}), we have
 \begin{equation}\label{eq:(1,1) structured}
     \rho'(|A_i| + |A_j|) = e_F(A_i,A_j) + e_{F}(A_i\cup A_j,V(F)\backslash (A_i\cup A_j))
 \end{equation}
for all $1 \leq i < j \leq r$. Now, summing \eqref{eq:(1,1) structured} over all pairs $i,j$ with $i,j \in I$, we get
\begin{align*}\label{eq:sum e(Ai,Aj) for (1,1) structured}
    \rho' (|I|-1)\sum_{i \in I}|A_i| &= \sum_{i,j \in I}\rho'(|A_i| + |A_j|) = \sum_{i,j \in I}\left[e_F(A_i,A_j) + e_{F}(A_i\cup A_j,V(J)\backslash (A_i\cup A_J)) \right] \nonumber \\ &= 
(|I|-1)\sum_{k \in I, \ell \in [r] \setminus I}e_F(A_k,A_{\ell}) + 
(2|I|-3)\sum_{k,\ell \in I}e_F(A_k,A_{\ell}) \nonumber \\&\geq 
(|I|-1)\sum_{k,\ell : \{k,\ell\} \cap I \neq \emptyset}e_F(A_k,A_{\ell}) \\&\geq 
(|I|-1) \cdot e(F^-),
\end{align*}
where the last inequality holds because every edge of color $-1$ in $F$ is incident to $\bigcup_{i\in I}A_i$. 
Dividing through by $|I|-1 \geq 1$, we get the left inequality in \eqref{eq:e(F^-) A1,..,Ar},
as required.

Now suppose that $H$, and hence also $F$, are $(1,0)$-structured. This means that
\begin{equation}\label{eq:(1,0) structured}
    \rho'(|A_i| + |A_j|) = e_{F}(A_i\cup A_j,V(J)\backslash (A_i\cup A_J))
\end{equation}
for all $1 \leq i < j \leq r$. Summing \eqref{eq:(1,0) structured} over all pairs $(i,j)$ with $i,j\in I $, we get
\begin{align}\label{eq:sum e(Ai,Aj) for (1,0) structured}
    \rho' (|I|-1)\sum_{i \in I}|A_i| &= \sum_{i,j \in I}\rho'(|A_i| + |A_j|) = \sum_{i,j \in I}e_{F}(A_i\cup A_j,V(J)\backslash (A_i\cup A_J)) \nonumber \\ &= 
(|I|-1)\sum_{k \in I, \ell \in [r] \setminus I}e_F(A_k,A_{\ell}) + 
2(|I|-2)\sum_{k,\ell \in I}e_F(A_k,A_{\ell}).
\end{align}
If $|I| \geq 3$ then $2(|I|-2) \geq |I|-1$, so \eqref{eq:sum e(Ai,Aj) for (1,0) structured} counts $e_F(A_k,A_{\ell})$ at least $|I|-1$ times for every $1 \leq k < \ell \leq r$ with $\{k,\ell\} \cap I \neq \emptyset$. Hence, \eqref{eq:sum e(Ai,Aj) for (1,0) structured} is an upper bound for $(|I|-1) \cdot e(F^-)$, and the assertion of the claim follows by dividing \eqref{eq:sum e(Ai,Aj) for (1,0) structured} through by $|I|-1$. Now suppose that $|I|=2$, say $I = \{1,2\}$ without loss of generality. Then, by Item 2 of Claim~\ref{clm: need more than one}, all edges between $A_1$ and $A_2$ have color $1$. Therefore,
$$
e(F^-) \leq \sum_{k \in I, \ell \in [r] \setminus I}e_F(A_k,A_{\ell}) \leq \rho' \sum_{i \in I}|A_i|, 
$$
using \eqref{eq:sum e(Ai,Aj) for (1,0) structured}. So again the left inequality in \eqref{eq:e(F^-) A1,..,Ar} holds.
\end{proof}
We now complete the proof of Lemma~\ref{high_min_deg}. By Lemma~\ref{lem:H* factor}, $G'$ has a perfect $H^*$-factor $F^*$. For each $H^*$-copy $J \in F^*$, let $F_J$ be a perfect $H$-factor of $J$. Let $F = \bigcup_{J \in F^*}F_J$ be the resulting perfect $H$-factor of $G'$. 
We now use Lemma~\ref{lem:(s,t)-structured edge count} to estimate the number of edges of $H$, using that $H$ is $(s',t')$-structured for $s'=1$ and $t' \in\{0,1\}$. By Lemma~\ref{lem:(s,t)-structured edge count},  
$
e(H) = \rho'\frac{r-1}{(2r-4)s'+t'}|H|,
$
so 
$$
\rho'|H|/2 \leq \rho'\frac{r-1}{2r-3}|H| \leq e(H) \leq \rho'\frac{r-1}{2r-4}|H| \leq \rho'|H|.
$$
It follows that $e(F) = \frac{n}{|H|} \cdot e(H) \geq \rho'n/2$.
There are at most $|V_0|$ copies of $H^*$ in $F^*$ intersecting $V_0$, and the $H$-factors of these copies of $H^*$ contain therefore at most $|V_0| \cdot \frac{|H^*|}{|H|} \cdot e(H) \leq \rho'|V_0||H^*|$ edges. Also, if an $H^*$-copy $J \in F^*$ does not intersect $V_0$, then $F_J$ contains at most $\rho'|V(J) \cap V_W|$ edges of color $-1$, by Claim~\ref{claim:Lemma 13.3 main}. 
It follows that
\begin{align}\label{eq: e(F^-) upper bound}
    e(F^-) &\leq \rho'|V_0||H^*| + \rho'\sum_{J \in F^*}|V(J) \cap V_W| = \rho'|V_0||H^*| + \rho'|V_W| \overset{\text{(a)}}{\leq}
\rho' \left( \varepsilon n |H^*| + \left(\frac{1}{4} - \eta \right)n \right) \notag\overset{\text{(b)}}{\leq} 
\\ &\rho'\left( \frac{1}{4} - \frac{\eta}{2} \right)n \leq \frac{e(F)}{2} - \frac{\rho'\eta}{2}n \overset{\text{(c)}}{\leq} \frac{e(F)}{2} - \gamma n.
\end{align}
Here, inequality (a) uses that $|V_0| \leq \varepsilon n$ and that 
$|V_W| \leq \left(\frac{1}{4} - \eta \right)n$ by \eqref{eq:|W|}. Inequality (b) uses that $H^*$ depends only on $H,\eta$ and $\varepsilon \ll \frac{1}{|H|},\eta$. And inequality (c) uses that $\rho' > 0$ depends only on $H$ and $\gamma \ll \frac{1}{|H|},\eta$. So we got that $f(F) = e(F) - 2e(F^-) \geq \gamma n$, namely $F$ has high discrepancy. This completes the proof. \label{sec:non-reg}
\section{Proof of the main results}\label{sec:main}
\subsection{Proof of Theorem~\ref{bipartite}}\label{subsec:proof_bipartite}
\begin{proof}
Let $H$ be a bipartite graph.
% Note that $H$ does not fulfill the $C_4$-condition for $4$, and thus also not for $5$, for the following reason (in fact, any $r$-partite graph does not fulfill the $C_4$-condition for $r+2$). Let $A_1,A_2$ be the parts of a $2$-vertex-coloring of $H$. Let $A_3,A_4$ be empty sets. Then, $A_1,A_2,A_3,A_4$ are the parts of a $4$-vertex-coloring of $H$ and 
% $$e_H(A_1,A_2)+e_H(A_3,A_4)\neq e_H(A_1,A_3)+e_H(A_2,A_4),$$
% as $e_H(A_1,A_2)\neq 0$ and all the other terms are zero.
% By Lemma~\ref{no_c4} applied with $k=5$, we get that $$\delta^*(H) \leq 3/4.$$
% By Lemma~\ref{H_regular} this is tight if $H$ is regular.
% Therefore, let us assume from now on that $H$ is non-regular. 
By Corollary~\ref{cor: no_c4 for r<4}, we have $\delta^*(H) \leq 3/4.$
By Lemma~\ref{H_regular}, this is tight if $H$ is regular.
Therefore, let us assume from now on that $H$ is non-regular.

First, suppose that there exists $\rho$ such that for every connected component $U$ of $H$ it holds that $e_H(U)=\rho|U|$, which corresponds to the second case of Theorem \ref{bipartite}. By Lemma~\ref{claim:bipartite_rho}, we get that 
$$
\delta^*(H)\geq 1/2.
$$
Now, let us show that $\delta^*(H)\leq 1/2$. 
We work in the setup described in Section~\ref{sec_app_regularity}.
In particular, we assume that 
%$\delta(G')/|G'|\geq1/2+\eta/2$ and also 
$\delta(R)/|R|\geq 1/2+\eta/2$. We need to show that $G'$ has a perfect $H$-factor with high discrepancy.

If $R$ is monochromatic, then there exists a perfect $H$-factor in $G'$ with high discrepancy by Lemma~\ref{all k positive}.
%Theorem~\ref{existence} and since $\delta_G\geq1/2\geq 1-1/\chi^*(H)$, there exists a perfect $H$-factor $F$ of $G'$. It follows that 
%$$
%|f(F)|\geq \left(\frac{n}{|H|}-2|V_0|e(H)\right)\geq \gamma n.
%$$
Therefore, let us assume that $R$ is not monochromatic. 
%Suppose that we are in the second case (i.e. there is $\rho$ such that $e_H(U) = \rho|U|$ for every connected component $U$). 
%Then $\delta(G) \geq (1/2+\eta)n$, hence
%$\delta(R)>|R|/2$. 
This implies that there exist vertices $u,v,w\in V(R)$ such that $f_R(uv) = -f_R(vw)=1$. Indeed, $R$ is not monochromatic so there are edges $xy,st$ of different colors. By Lemma~\ref{connectivity} applied with $k=2,$ there is a path in $R$ whose first and last edges are $xy$ and $st.$ On this path, there must be two consecutive edges of different colors, giving the vertices $u,v,w$ as above.

Note that $uv,vw$ are two copies of $K_2$ with different discrepancies and $H$ is non-regular. By Lemma~\ref{sharing_much_template}, $(\{u,v,w\},f_R)$ is a template for $H$, and then by Lemma~\ref{template} (with $r=2$), $G'$ has a perfect $H$-factor with high discrepancy, as required.  

Now suppose that we are in the last case of Theorem \ref{bipartite}, meaning that there are two connected components $U,W$ of $H$ and $\rho \neq \rho'$ such that $e_H(U) = \rho|U|$ and $e_H(W)=\rho'|W|$. In other words, $e_H(U)/|U| \neq e_H(W)/|W|$. Recall that $\delta^*(H)\geq1-1/\chi^*(H)$ trivially holds for every $H$, so we only need to show that $\delta^*(H)\leq1-1/\chi^*(H)$. 
Again, we show that $G'$ has a perfect $H$-factor of high discrepancy under the setting of Section~\ref{sec_app_regularity}.
% Again, as described in Section~\ref{sec_app_regularity}, let us assume that $\delta(G')/|G'|\geq1-1/\chi^*(H)+\eta/2$ as well as $\delta(R)/|R|\geq1-1/\chi^*(H)+\eta/2$ towards showing $\delta^*(H)\leq 1-1/\chi^*(H)$. 
As in the previous case, we may assume that $R$ is non-monochromatic as otherwise we are done by Lemma~\ref{all k positive}.
Thus, there exist edges $e_1,e_2$ with $f_R(e_1) \neq f_R(e_2)$. If $e_1,e_2$ are not disjoint then $(e_1\cup e_2,f_R)$ is a template for $H$ by Lemma~\ref{sharing_much_template}, and if they are disjoint then it is a template by Lemma~\ref{bipartite template}. Either way, we can apply Lemma \ref{template} to conclude that $G'$ has a perfect $H$-factor with high discrepancy. Thus, we get $\delta^*(H) = 1-1/\chi^*(H)$. 
\end{proof}
\subsection{Proof of Theorem~\ref{tripartite1}}\label{subsec:proof_tripartite}
Let $H$ be a graph with $\chi(H) = 3$.
By Corollary~\ref{cor: no_c4 for r<4}, $\delta^*(H) \leq 3/4$. If $H$ is regular, then $\delta^*(H) = 3/4$ by Lemma~\ref{H_regular}.
So suppose from now on that $H$ is non-regular. 
Recall that every graph $H$ satisfies $\delta^*(H) \geq \max\{\delta_0(H),1-1/\chi^*(H)\}$. 
If some butterfly is not a template for $H$, then $\delta^*(H) \geq 4/7$ by Lemma~\ref{butterfly} and $\delta^*(H) \leq  \max\{\delta_0(H),1-1/\chi^*(H),4/7\}$ by Lemma~\ref{high_min_deg}. And if every butterfly is a template for $H$, then $\delta^*(H) \leq \max\{\delta_0(H),1-1/\chi^*(H)\}$ by Lemma~\ref{high_min_deg}. This concludes the proof.  
\subsection{Proof of Theorem~\ref{rpartite}}
\label{subsec:proof_rpartite}
Let $H$ be an $r$-chromatic graph, $r \geq 4$. Throughout the proof, we use the fact that
\begin{equation}\label{eq:alpha(H)}
1-1/(r-1) \leq \max\{\delta_0(H),1-1/\chi^*(H)\} \leq 1-1/r,
\end{equation}
where the first inequality holds because $\chi^*(H) \geq r-1$. 
We begin with the first case of Theorem~\ref{rpartite}, where we assume that $H$ satisfies Condition~\ref{cond1}.  
By Corollary~\ref{cor: no_c4 for r<4}, $\delta^*(H)\leq 1-1/(r+1)$.
We now use Condition~\ref{cond1} to show that $\delta^*(H)\geq 1-1/(r+1).$
Indeed, if $r \equiv_4 0$ then this follows from Lemma~\ref{construction} with $k = r+1 \equiv_4 1$, and if $r \not\equiv_4 0$ then this follows from Lemma~\ref{construction2} with $k = r+1 \not\equiv_4 1$ (using that $H$ \nolinebreak is \nolinebreak regular). 

We now move on to the second case of Theorem~\ref{rpartite}. 
First, we show that if $H$ violates Condition~\ref{cond1} then $\delta^*(H) \leq 1-1/r$. Indeed, if $H$ violates the $(r+1)$-wise $C_4$-condition, then $\delta^*(H) \leq 1 - 1/r$ by Lemma~\ref{no_c4} with $k = r+1$. 
And if $H$ is non-regular and $r \not\equiv_4 0$, then $\delta^*(H) \leq 1 - 1/r$ by Lemma~\ref{c4_r_1}. 
Next, observe that if $H$ satisfies Condition~\ref{cond2}, then $\delta^*(H) \geq 1-1/r$ by Lemmas~\ref{construction}-\ref{construction2}.

Finally, we handle the last case of Theorem~\ref{rpartite}. Here we show that if $H$ violates Conditions~\ref{cond1} and \ref{cond2}, then $\delta^*(H) \leq \max\{\delta_0(H),1-1/\chi^*(H)\} =: \alpha(H)$. This is tight because 
$\delta^*(H) \geq \alpha(H)$ for every graph $H$. 
If $H$ violates the $r$-wise $C_4$-condition, then $\delta^*(H) \leq \alpha(H)$ by Lemma~\ref{no_c4} with $k = r$, using that $1-1/(r-1) \leq \alpha(H)$ by \eqref{eq:alpha(H)}. So suppose that $H$ satisfies the $r$-wise $C_4$-condition. Then, as $H$ violates Condition~\ref{cond2}, it must be that $H$ is non-regular. 
Now, if $r \equiv_4 2,3$, then $\delta^*(H) \leq \alpha(H)$ by Lemma~\ref{high_min_deg}. 
If $r \equiv_4 1$ then $\alpha(H) \geq \delta_0(H) = 1-1/r$ by Lemma~\ref{construction} and $\delta^*(H) \leq 1-1/r$ by Lemma~\ref{c4_r_1}, so $\delta^*(H) \leq \alpha(H)$ holds. 
Suppose now that $r \equiv_4 0$. Then, as $H$ violates Condition~\ref{cond1}, $H$ must violate the $(r+1)$-wise $C_4$-condition. Now $\delta^*(H) \leq \alpha(H)$ holds by Lemma~\ref{rest_4}. 
This completes the proof. 
\section{Examples}\label{sec: examples}
The purpose of this section is to demonstrate that the cases in our theorems are necessary. For the bipartite case, Theorem~\ref{bipartite}, this is fairly easy to see so we only discuss Theorems~\ref{tripartite1} and~\ref{rpartite}. Towards this, we give graphs $H$ as examples for what we consider to be the more interesting cases.
The task of finding examples of $r$-partite graphs becomes much simpler when they have exactly one proper $r$-coloring (up to permutations of the color-labels). To achieve this, we use the following construction in most of the examples:
\begin{enumerate}[label=\textbf{C}]
    \item\label{example_graph} Let $H$ be a graph on vertex-set $V(H)$ with $r$-partition $A_1,A_2,\dots A_r$ and vertices $a_1\in A_1,a_2\in A_2,\dots, a_r\in A_r$. For $1\leq i\leq r$, $a_i$ is connected to every vertex in $\bigcup_{j\neq i}A_j$.
\end{enumerate}
Then, given an $r$-coloring of $a_1,a_2,\dots a_r$, we get that for every $1\leq i\leq r$, all the vertices in $A_i$ must have the same color as $a_i$ and therefore, the coloring is unique up to permutation of the labels. Note that for $1\leq i<j\leq r$, we can add any edges to $H[A_i,A_j]$ and this property does not change. Additionally such a graph $H$, is never regular, unless $|A_1|=|A_2|=\dots=|A_r|$ and $H$ is the complete $r$-partite graph. One constraint that such graphs $H$ have is that for $1\leq i<j\leq r$, $e_H(A_i,A_j) \geq |A_i|+|A_j|-1$, given by the edges incident to $a_i$ and $a_j$.
\begin{itemize}
    \item First, we give a tripartite graph $H$ for which $$\delta_0(H)<\delta^*(H)=1-1/\chi_{cr}(H).$$ Towards this, consider $H$ as described in \ref{example_graph} with $|A_1|=10, |A_2|=11, |A_3|=100$. Note that $hcf(H) = 1$, as $|A_2|-|A_1|=1$. Besides the edges given by $a_1,a_2,a_3$, let there be arbitrary additional edges such that $e_H(A_1,A_2)=e_H(A_1,A_3)=e_H(A_2,A_3)=110$. Note that this means that $H[A_1,A_2]$ is complete and $H[A_2,A_3]$ has no extra edges besides the ones touching $a_2$ or $a_3$. It is not hard to see that $\delta_0(H) = 0$ and $1-1/\chi_{cr}(H)<1-1/r$. Additionally, $H$ can use any butterfly as a template, as otherwise by Lemma~\ref{template_example} $H$ is either $(2,2)$-, $(0,2)$- or $(2,-2)$-structured, which it is clearly not. 
    \item Next, we give an example for a tripartite graph $H$ with $$1-1/\chi^*(H)<\delta^*(H)=\delta_0(H).$$
    
    Let $H$ be a tripartite graph as described in \ref{example_graph} with $|A_1|= 5,|A_2| = 20 ,|A_3|= 21$ and $e_H(A_1,A_2)=28, e_H(A_1,A_3)=42$ and $e_H(A_2,A_3)=252$. It is not hard to check that such $H$ is indeed $(1,2)$-structured for $\rho=14$. Additionally, since $|A_3|-|A_2| = 1$, we have $hcf(H) = 1$ and thus 
    $$1-1/\chi^*(H) = 1-1/\chi_{cr}(H) = 1-\frac{41}{92}<5/8.$$ 
    By Lemma~\ref{lower_delta_r=3}, it follows that $\delta_0(H) \geq 5/8$. As $H$ is non-regular and since $5/8>4/7$, we get by Theorem~\ref{tripartite1} that indeed
    $$1-1/\chi^*(H)<5/8\leq \delta^*(H) = \delta_0(H).$$
    
    \item Next, we give an example corresponding to the second case of Theorem~\ref{tripartite1} such that $\delta^*(H)=4/7>\max\{1-1/\chi^*(H),\delta_0(H)\}$. Towards this, some butterfly should not be a template for $H$. Let $|A_1|= 5,|A_2| = 20 ,|A_3|= 21$ and $e_H(A_1,A_2)=67, e_H(A_1,A_3)=66$ and $e_H(A_2,A_3)=51$. We get that $H$ is non-regular and $(1,-1)$-structured for $\rho = 2$. Consider the butterfly given by $(L,c)$ (see the third graph in Figure~\ref{fig:butterflies}), where $L$ consists of the two triangles $L_1=\{u,v_1,w_1\}$ and $L_2=\{u,v_2,w_2\}$ with 
    $$c(uv_1)=c(uw_1)=-c(uv_2)=-c(uw_2)=-c(v_1w_1)=c(v_2w_2)=-1.$$ 
    We will show that $L$ is not a template for $H$. Note that by Lemma~\ref{template_example}, it is necessary that then, $H$ is $(-2,2)$-structured (or by normalizing $(1,-1)$-structured).
    Let $B$ be an arbitrary blowup of $(L,c)$. Note that given some $H$ constructed as described in \ref{example_graph}, any copy of $H$ in $B$ is either included in $V_{V(L_1)}$ or in $V_{V(L_2)}$. To see this, consider the placement of the three vertices $a_1,a_2,a_3$. As they form a triangle, they must be either on $V_{V(L_1)}$ or $V_{V(L_2)}$. Say they are on $V_{V(L_1)}$. But each vertex of $H$ forms a triangle with two of $a_1,a_2,a_3$ and must therefore also be on $V_{V(L_1)}$. Then, it is not hard to see that since $H$ is $(1,-1)$-structured,
    $L$ is not a template for $H$. We have that $\delta_0(H) = 0$, as $H$ is $(1,-1)$-structured with nonzero $\rho$ and also $1-1/\chi^*(H) = 1-41/92<4/7$ as in the previous example. We then get by the second case in Theorem~\ref{tripartite1} that $\delta^*(H) = 4/7$.
    \item 
    Let us now give an example of an $r$-partite, regular graph $H$ for some $r\geq 4$ which fulfills the $r$-wise $C_4$-condition for some $r\not\equiv_4 0, 1 $ and has $1-1/r>\max\{\delta_0(H),1-1/\chi^*(H)\}$. Note that Theorem~\ref{rpartite} shows that $\delta^*(H) = 1-1/r$ in that case. To find such a graph, the construction given in \eqref{example_graph} is not very helpful, as the only regular graph constructed in such a way is the complete, balanced $r$-partite graph. Thus, let us consider a different construction. For some integer $m$, let $H$ be the graph obtained from the complete $r$-partite graph with parts $A_1, \dots, A_r$ with sizes $|A_1| = (r-2)m+1$ and $|A_2|=|A_3|=\dots = |A_r=(r-2)m$, by removing a matching of size $m$ between every pair $A_i, A_j$ with $2 \le i < j \le r$ such that for any vertex not in $A_1,$ exactly one of its incident edges is removed.
    
    Counting the edges per vertex, it is not hard to see that for $v\in A_1$, we have that $d_H(v)= (r-1)(r-2)m$ and for $v\in A_i$ for $2\leq i\leq r$ we have that $v$ is connected to everything but the vertices in $A_i$ and one other vertex. It follows that these vertices are incident to $(r-2)^2m-1+(r-2)m+1=(r-1)(r-2)m$ edges. Therefore, $H$ is regular. Additionally, we have for $2\leq i<j\leq r$, $$e_H(A_i,A_j) = (r-2)^2m^2-m$$ and for all $2 \le i \le r$, $$e_H(A_1,A_i)=(r-2)m\cdot((r-2)m+1).$$ Then, it is not hard to see that since $r\geq 4$, $H$ has only one $r$-coloring (up to permutation of the labelling) and the $r$-wise $C_4$-condition holds for this coloring, but the $(r+1)$-wise $C_4$ condition does not. To see the latter, consider the natural $r$-coloring of $H$ and add an additional empty color class $A_0$. Then the $4$-cycle $A_0, A_1, A_2, A_3$ shows that $H$ violates the $(r+1)$-wise $C_4$-condition. The coloring $A_1, \dots, A_r$ shows that $\chi^*(H)<r$. Let us also prove that $\delta_0(H)<1-1/r$. To see this, consider a $b$-blowup $B$ of $(K_r,c)$ for some $b\in\mathbb{N}$ and $2$-edge-coloring $c$ of $K_r$. As $r\not\equiv_4 0,1$, $c(K_r)\neq 0$. It is not hard to see that there is at most one way (up to permutations of $A_2,A_3,\dots, A_r$) to find a perfect $H$-factor $F$ in $B$. Note that this $H$-factor uses the same amount of edges in every bipartite graph $B[V_v,V_u]$, where $u,v\in V(K_r)$. Therefore, $c(F)$ has the same sign as $c(K_r)$ and is non-zero. It follows that $\delta_0(H)<1-1/r$.
\end{itemize}
\bibliographystyle{abbrv}
\bibliography{sources.bib}
\end{document}